\documentclass[10pt]{article}         

\usepackage{mathrsfs}
\usepackage{amsmath}
\usepackage{amsfonts}
\usepackage{amssymb}
\usepackage{amsthm}
\usepackage{caption}
\usepackage{subfigure}
\usepackage{cite}
\usepackage{color}
\usepackage{graphicx}
\usepackage{subfigure}
\usepackage{float}
\usepackage{paralist}
\usepackage[colorlinks,linkcolor=blue,anchorcolor=blue,citecolor=blue,  urlcolor=blue]{hyperref}
\usepackage{indentfirst}

\usepackage{authblk}
\usepackage[T1]{fontenc}
\usepackage{amscd}
\usepackage{latexsym}
\usepackage{verbatim}
\usepackage{enumerate}
\usepackage{fullpage}
\usepackage{xcolor}
\usepackage{titlesec}
\usepackage{titletoc}
\usepackage{nameref}
\usepackage{appendix}
\usepackage{doi}

\newtheorem{theorem}{Theorem}[section]
\newtheorem{lemma}{Lemma}[section]
\newtheorem{definition}{Definition}[section]
\newtheorem{notation}{Notation}[section]
\newtheorem{remark}{Remark}[section]
\newtheorem{proposition}{Proposition}[section]
\numberwithin{equation}{section}
\title{Incompressible limits of the Patlak-Keller-Segel model and its stationary state}
\author[a,b]{Qingyou He\thanks{E-mail: qyhe.cnu.math@qq.com (Q. Y. He)}}
\author[a,b]{Hai-Liang  Li\thanks{
E-mail:		hailiang.li.math@gmail.com (H.-L. Li)}}
\author[c]{Beno\^{\i}t Perthame\thanks{
E-mail:		benoit.perthame@sorbonne-universite.fr (B. Perthame)}}
    \affil[a]{School of Mathematical Sciences,
	Capital Normal University, Beijing 100048, P.R. China}
\affil[b]{Academy for Multidisciplinary Studies, Capital Normal University, Beijing~100048, P.R. China}
\affil[c]{Sorbonne Universit{\'e}, CNRS, Universit\'{e} de Paris, Inria, Laboratoire Jacques-Louis Lions UMR7598, F-75005 Paris}
\date{}           

\begin{document}
\maketitle
\begin{abstract}
We complete previous results about the incompressible limit of both the $n$-dimensional $(n\geq3)$ compressible Patlak-Keller-Segel (PKS) model and its stationary state. As in previous works, in this limit, we derive the weak form of a geometric free boundary problem of Hele-Shaw type, also called congested flow. In particular, we are able to take into account the unsaturated zone, and establish the complementarity relation which describes the limit pressure by a degenerate elliptic equation. Not only our analysis uses a completely different framework than previous approaches, but  we also establish a novel uniform $L^3$ estimate of the pressure gradient, regularity \`a la Aronson-B\'enilan,  
and a uniform $L^1$ estimate for the time derivative of the pressure. Furthermore, for the Hele-Shaw problem, we prove the uniqueness of solutions, meaning that the incompressible limit of the PKS model is unique.
In addition,  we establish the corresponding incompressible limit of the stationary state for the PKS model with a given mass, where, different from the case of PKS model, we obtain the uniform bound of pressure and the uniformly bounded support of density.
\end{abstract}

\\
2010 {\bf MSC.} 35B45; 35K57; 35K65; 35Q92; 76N10;
76T99


\section{Introduction}

The Patlak-Keller-Segel (PKS) model can be used to describe the collective dynamics  of a large
number of individual agents interacting  through a diffusive signal. For instance, it appears for the chemotaxis phenomena of various types of cells, aggregation dynamics of crowds or to describe the gravitational collapse, see~\cite{rr26, MRSV11,CCH}, and references therein. With a source term, it is used as a mechanical description of tumor growth \cite{r34, rBC}. Including nonlinear diffusivity and Newtonian interactions, the PKS model is written
\begin{equation}\label{d1}
\begin{cases}
\partial _t\rho_m=\Delta \rho_m^m+\nabla\cdot(\rho_m\nabla\mathcal{N}\ast\rho_m)\quad\text{for} \quad (x,t)\in\mathbb{R}^n\times\mathbb{R}_+, \quad n\geq 3,
\\[5pt]
\rho_{m}(x,0)=\rho_{m,0}(x)\geq0\quad\text{for }x\in\mathbb{R}^n,  \qquad  \qquad  m>2-2/n, \qquad  \text{(subcritical case)}.
\end{cases}
\end{equation}	
For chemotaxis, $\rho_m(x,t) \geq 0$ represents the cell density and $\mathcal{N}\ast\rho_m$ represents the chemical substance concentration  obtained by convolution with  the Newtonian potential
\begin{equation*}
\mathcal{N}(x)=\frac{-1}{n(n-2)\alpha_n|x|^{n-2}}\qquad\text{for } \quad
x\in\mathbb{R}^n\backslash \{0\}, \qquad \Delta \mathcal{N} = \delta,
\end{equation*}
with $\alpha_n>0$  being the volume of the $n$-dimensional unit ball and $\delta$  the Dirac measure. The conservation of mass for the Cauchy problem Eq.~\eqref{d1}  holds
\begin{equation*}
\int_{\mathbb{R}^n}\rho_m(x,t)dx=\int_{\mathbb{R}^n}\rho_{m,0}dx:=M, \qquad \forall\ t \geq 0.
\end{equation*}
For solutions of Eq.~\eqref{d1}, the pressure denotes a power of the density (Darcy's law) as
\begin{equation}\label{d6}
P_m:=\frac{m}{m-1}\rho_m^{m-1}, \qquad \quad P_{m,0}:=\frac{m}{m-1}\rho_{m,0}^{m-1}.
\end{equation}
We can rewrite Eq.~\eqref{d1}  for the density $\rho_m$ and pressure $P_m$ in terms of the transport
equation with the effective velocity $u_m$ as
\begin{equation}\label{d2}
\partial _t\rho_m=\nabla\cdot(\rho_mu_m), \qquad u_m:=\nabla
P_m+\nabla\mathcal{N}\ast\rho_m.
\end{equation}	
By a direct computation, the pressure satisfies the equation
\begin{equation}\label{d3}
\partial_t P_m=(m-1)P_m(\Delta P_m+\rho_m)+\nabla P_m\cdot(\nabla
P_m+\nabla\mathcal{N}\ast\rho_m).
\end{equation}
The competition between the degenerate diffusion and the nonlocal aggregation is
the main characteristic of Eq.~(\ref{d1}) or Eq.~(\ref{d2}). This is well
represented by the free energy functional 
\begin{equation}\label{a58}
F_m(\rho_m)= \frac{1}{m-1}\int_{\mathbb{R}^n}\rho_m^mdx-
\frac{1}{2}\int_{\mathbb{R}^n}|\nabla\mathcal{N}\ast\rho_m|^2dx.
\end{equation}
It satisfies the energy identity, which shows that $F_m(\rho_m)$ is non-increasing with time, 
\begin{equation}\label{f1}
\frac{dF_m(\rho_m)}{dt}+\int_{\mathbb{R}^n}\rho_m|\nabla(P_m+\mathcal{N}\ast\rho_m)|^2dx=0.
\end{equation}
 Since $\frac{\delta F_m(\rho_m)}{\delta \rho_m}=P_m+\mathcal{N}\ast\rho_m$ represents the
chemical potential, there exists a gradient flow structure for the PKS model,
\begin{equation}\label{gf}
\frac{dF_m(\rho_m)}{dt}+\int_{\mathbb{R}^n}\rho_m|\nabla\frac{\delta
F_m(\rho_m)}{\delta\rho_m}|^2dx=0.
\end{equation}

Solutions $\rho_{m,s}$ of the stationary PKS system (SPKS) satisfy that the free energy $F_m(\rho_m(t))$ is  constant in  time and thus are determined by 
\begin{equation}\label{SPKS}
\nabla\rho_{m,s}^m+\rho_{m,s}\nabla\mathcal{N}\ast\rho_{m,s}=0\quad\text{for
}x\in\mathbb{R}^n, \qquad P_{m,s}=\frac{m}{m-1}\rho_{m,s}^{m-1}.
\end{equation}
\\

Since a decade, and the paper~\cite{5} motivated by tumor growth, a large literature has been devoted to studying the incompressible (Hele-Shaw) limit, which means the limit as $m\to\infty$, for several variants of the porous medium equations (see below).   In particular, establishing this limit when Newtonian interactions are included, as in  Eq.~\eqref{d1} and  Eq.~\eqref{SPKS}, has been a long standing question solved in~\cite{CKY_2018} based both on optimal transportation methods and viscosity solutions. 
\\

\paragraph{Incompressible limit.} Our purpose is to complete the understanding, from~\cite{CKY_2018},  of the incompressible (Hele-Shaw) limit for the PKS model Eq.~\eqref{d1} in various directions. Firstly, we introduce a third approach based on weak solutions as described below. In particular, our assumptions on the initial data are more general (not necessarily  patches data), and the method can easily be extended to source terms when mass varies. Secondly, we prove new regularity results: an $L^3$ estimate on $\nabla P_m$ and  regularity \`a la Aronson-B\'enilan showing bounds on the second derivatives of the pressure $P_m$. Thirdly, we can prove directly an estimate on the time derivative of the pressure based on a new idea since a direct approach would not work. Finally, we prove a new uniqueness theorem for the limiting Hele-Shaw problem.
\\

Following~\cite{CKY_2018}, the Hele-Shaw limit system writes
\begin{align}\label{hs}
\begin{cases}
\partial_{t}\rho_{\infty}=\Delta
P_{\infty}+\nabla\cdot(\rho_{\infty}\nabla\mathcal{N}\ast\rho_{\infty}), \qquad
&\text{in } D'\big(\mathbb{R}^n\times \mathbb{R}_+\big),
\\[5pt]
(1-\rho_{\infty})P_{\infty}=0,  \qquad \quad 0\leq\rho_{\infty}\leq1,   &\text{a.e. }(x,t)\in \mathbb{R}^n\times
\mathbb{R}_+.
\end{cases}
\end{align}
This is a weak version of the geometric Hele-Shaw problem including chemotaxis. We also prove the
complementarity relation (in distributional sense)
\begin{equation}\label{d5}
P_{\infty}(\Delta P_{\infty}+\rho_{\infty})=0.
\end{equation}
It describes the limit pressure by a degenerate elliptic equation once we know the regularity of the 
set $\{ p_\infty > 0 \}$, which is a major challenge for
the Hele-Shaw problem, see \cite{rCS,r20,r60, GKM} and reference therein. Furthermore, with Eqs.~\eqref{hs}--\eqref{d5} at hand, the limiting free energy functional easily follows,
\begin{equation}\label{f3}
\begin{cases}
F_{\infty}(\rho_\infty)=\frac{1}{2}\int_{\mathbb{R}^n}\rho_{\infty}\mathcal{N}\ast\rho_{\infty}dx, \qquad 0\leq\rho_\infty\leq 1,
\\[10pt]
\frac{dF_{\infty}(\rho_\infty(t))}{dt}+\int_{\mathbb{R}^n}
\rho_{\infty}(t)|\nabla\big(P_\infty(t)+\mathcal{N}\ast\rho_{\infty}(t)\big)|^2dx=0.
\end{cases}
\end{equation}
Compared with
the free energy~\eqref{a58}, the diffusive effect is replaced by the height
constraint $\rho_\infty\leq 1$. In the end, we extend the uniqueness \cite{4,r46} of solution to the PKS model Eq.~\eqref{d1} to the uniqueness of solution to the Hele-Shaw limit system Eq.~\eqref{hs}.

In the stationary case, the incompressible (Hele-Shaw) limit from the SPKS model Eq.~\eqref{SPKS} as $m\to \infty$, is represented as
\begin{equation}\label{HSSPKS}
\begin{cases}
\nabla P_{\infty,s}+\rho_{\infty,s}\nabla\mathcal{N}\ast\rho_{\infty,s}=0,&\text{in }\mathcal{D}'(\mathbb{R}^n),
\\[5pt]
(1-\rho_{\infty,s})P_{\infty,s}=0, \qquad \quad  0\leq \rho_{\infty,s}\leq 1,\quad &\text{a.e. }x\in\mathbb{R}^n.
\end{cases}
\end{equation}
As before, this corresponds to vanishinf dissipation for the free energy $F_\infty(\rho_{\infty}(t))$.
\\

The limits~\eqref{hs}--\eqref{f3} can be formally derived from the PKS model Eq.~\eqref{d1}. Indeed,  taking the limit as $m\to\infty$ in Eq.~\eqref{d1}, we  formally obtain the
first equation in~\eqref{hs}.  Since we can prove that the limit pressure $P_{\infty}$ is  bounded, from ~\eqref{d6} we recover $\eqref{hs}_2$. Also, we can formally attain the complementarity relation Eq.~\eqref{d5} thanks to a direct calculation of Eq.~\eqref{d3} as $m\to\infty$.
In addition, from~\eqref{f1}, we can formally obtain the limit energy functional~\eqref{f3}. It should be emphasized that the structure of
gradient flow as in~\eqref{gf} is still present in a weak form as in the optimal
transportation approach, cf.~\cite{r18,MRS14}.  Similarly, the incompressible limit Eq.~$\eqref{HSSPKS}_1$ is formally derived from the SPKS model Eq.~\eqref{SPKS} as $m\to\infty$. As it is wellknown, establishing rigorously these limits faces deep difficulties due to the nonlinearities and weak regularity; the limit $\rho_\infty$ is discontinuous in space and $P_{\infty}$ can undergo discontinuities in time.

\paragraph{Review of literature.} As mentioned earlier, several approaches are possible to overcome the above mentioned difficulties. Optimal transportation methods are used in conservative cases, and the incompressible limit is the so-called {\em congested flows}. This  method was intitiated in \cite{MRS10}, and is well adapted for the tansitions from discrete to continuous models~\cite{MRSV11}. It was extended to the two-species case in~\cite{Laborde}.   The case of Newtonian drift, and the limit $m\to \infty$ was proved in~\cite{CKY_2018}.
\\

Another approach is by viscosity solutions, see for instance~\cite{IY,r20}, for an external drift see~\cite{LeiKim,r108}  and again~\cite{CKY_2018} for Newtonian drifts. In particular, this approach can handle source terms as initiated in~\cite{KimPS}. It has the advantage of handling specifically the free boundary in the limit with minimal assumptions for this purpose.
\\

Our approach is by weak solutions as defined below (see Def.~\ref{def:WS}) and is motivated by tumor growth models  of the form \begin{equation}\label{tg}
\partial_t\rho_m=\Delta \rho_m^m+\rho_m G(P_m)\quad\text{for }m>1,
\end{equation}
where $G(P)$ is a given decreasing function satisfying $G(P_M)=0$ for some threshold $P_M>0$. This problem was first solved in \cite{5}  using regularity as introduced by Aronson and B\'enilan~\cite{r32} and $BV$ estimates. The method was extended to include a drift, see~\cite{r50}, to replace Darcy's law by Brinkman's law~\cite{r51,r52}  and to a system with nutrients in~\cite{r35} using a new estimate in $L^4_{t,x}$ for $\nabla P_m$. Recently, multispecies problems were handled in~\cite{GPS2019,r53}, and a major improvement for compactness followed by~\cite{rPrX,r62}, see also~\cite{IgN} and the most advanced version in~\cite{david:hal-03636939}.

Furthermore, let us recall that for the porous medium equation (PME), i.e., when $G\equiv 0$, the problem leads to the so-called {\em mesa problem} and was also treated in a large literature, see for instance \cite{r4,r5,r10,r12} and references therein.  The weak formulation and the variational formulation  (using the so-called Baiocchi variable), of Hele-Shaw type were first introduced in~\cite{r107,r17} respectively.

Concerning the Keller-Segel model, with $m$ fixed,  very much is known and methods are nowadays well established.  Important progresses have been made recently on global existence, large time behaviors, critical mass and finite time blow-up  for the multi-dimensional  PKS model. In particular, the solutions with different diffusion exponent exhibit different behaviors. For diffusion exponent $1\leq m<2-2/n$ (supercritical case), the diffusion is dominant at the parts of low density and the aggregation is dominant at the parts of high density, then the solution to  Eq.~\eqref{d1} exhibits finite time blow-up for large mass and global existence in time for small mass, cf.~\cite{BL13,CC,CW,Sugi06,S2}. For $m=2-2/n$ (critical case), there exists a critical mass $M_c>0$ such that the solution  blows up in a finite time for the initial mass $M>M_c$, \cite{Sugi06, CC, CW}, and exists globally in time for the initial mass $M<M_c$, see~\cite{BCL,CCH} and reference therein. And for diffusion exponent $m>2-2/n$ (subcritical case), the diffusion dominates at the parts of high density, the solution to this model is uniformly bounded and  exists globally in time without any restriction on the size of the initial data, cf.~\cite{CCJ,S2,sb2014,BL13}. In addition, the large time behaviors have been investigated extensively, one can refer to~\cite{S7,YY,IY,LBW,r18} and references therein.

The SPKS model Eq.~\eqref{SPKS} has also been widely studied. For existence of solutions, see \cite{rr10,MF,LPL,rr2}, for uniqueness see \cite{rr12,r2,rr30}, and for radial symmetry see \cite{r18,r12,rr30}. Critical points of free energy $F_m(\rho_m)$ in~\eqref{a58} have been studied, see \cite{MF,rr10,LPL,rr2} and references therein. For the multi-dimensional SPKS model with more general attractive potential, the authors in~\cite{r18} proved that the solution is radially decreasing symmetric up to a translation obtained by the method of continuous Steiner symmetrization, then it was proved in~\cite{rr12} uniqueness $(m\geq2)$ and non-uniqueness $(1<m<2)$ of the solution to the SPKS  with general attractive potential.  Before that, the authors in~\cite{rr2} proved that all compactly supported solutions to the 3-dimensional SPKS model~Eq.~\eqref{SPKS} with $m>4/3$ must be radially symmetric up to a translation, hence obtaining  uniqueness of the solution among compactly supported functions. Furthermore, for the same case, the authors in~\cite{rr29} proved,  in 3 dimensions,  uniqueness of the solution among radial functions for a given mass, and their method can handle general potential when $m>2-2/n$. Similar results were obtained in~\cite{rr30} for 2-dimensional case with $m>1$ by an adapted moving plane technique. Carrillo et al. in~\cite{rr10} showed the existence and compact support  property of the radially symmetric solutions using  dynamical system arguments.

\paragraph{Difficulties and novelties}
 However, it should be emphasized that the arguments for passing to incompressible limit in \cite{5,r35,r50} cannot be applied directly to  Eq.~\eqref{d1}. This is due to the Newtonian drift in the PKS model, eventhough it is of lower order than the diffusion term. Its singularity gives rise to new and essential challenges for rigorously establishing the incompressible limits~\eqref{hs}--\eqref{f3}.  Indeed, for the models of tumor growth as Eq.~\eqref{tg}, the source term $\rho_mG(P_m)$ helps the authors to obtain a uniform $L^1$ estimate for the time derivative of both the density and the pressure by Kato's inequality. But, for Eq.~\eqref{d1}, on the one hand, the nonlocal Newtonian interaction leads to the absence of comparison principle, which means that it is impossible to get a uniform bound for the pressure. On the other hand, one of main challenges is to obtain a uniform $L^1$ estimate for the time derivative of pressure without the help of the source term, despite the effect of nonlocal interaction. Thus, it is  difficult to gain the desired compactness on not only the density but also the pressure for the PKS model.  Besides, using the weak formulation approach for the incompressible limit for the SPKS model Eq.~\eqref{SPKS} is a new and interesting topic for the diffusion-aggregation equations by the methof of weak solutions, see~\cite{CKY_2018} for viscosity solution methods.
\\ 

Therefore, to achieve our goals, we develop new estimates and strategies as follows:
\\[2pt]
$\bullet$  We  obtain the complementarity relation Eq.~\eqref{d5} for the PKS model Eq.~\eqref{d1}. We first derive a uniform $L^3$ estimate on the pressure gradient in the spirit of \cite{r35}. Then, we establish the uniform  Aronson-B\'enilian (AB) estimates in $L^3\cap L^1$ as initiated in~\cite{GPS2019}. In particular, we show a decay rate for the AB estimate in $L^3$ under the form
\begin{equation*}
         \|\min\big\{\Delta P_m+\rho_m,0\big\}\|_{L^3(Q_T)}^3\leq \frac{C(T)}{m}.
\end{equation*}
\\
$\bullet$  In addition, we establish a new uniform $L^1$ estimate for the time derivative of pressure. To our knowledge, this is the first time such an estimate is obtained  for the high-dimensional porous medium equation (Dracy's law) with a nonlocal attractive interaction since working directly on the pressure is not sufficient.
\\
$\bullet$  To prove the uniqueness of the solution to the Hele-Shaw limit system Eq.~\eqref{hs}, the key is to show that the limit pressure is somehow monotone to the limit density. Suppose that $P_i\rho_i=P_i$ and $0\leq \rho_i\leq 1$ for $i=1,2$ hold, we find $ (P_1-P_2)(\rho_1-\rho_2)\geq0$.
\\
$\bullet$  To establish the incompressible limit of the SPKS model with a given mass, we gain the uniform bound of the pressure and the uniformly bounded support of the density.

\paragraph{Notations}  We use the following notations and definitions.
\begin{notation}\label{d11} We set
\\
$\bullet$ $Q_{T}=\mathbb{R}^n\times(0,T)$, \quad
$Q=\mathbb{R}^n\times(0,\infty)$.
\\ $\bullet$
$B_{R}:=\{x:|x|\leq R\},\quad R>0$.\\
$\bullet$
$|f(x)|_{+}=\max\{f(x),0\}$, \quad $|f(x)|_-=-\min\{f(x),0\}$.
\\ $\bullet$
$\nabla^2f:\nabla^2g:=\sum\limits_{i,j=1}^{n}\partial_{ij}^2f\partial_{ij}^2g,\quad
(\nabla^2f)^2:=\nabla^2f:\nabla^2f=\sum\limits_{i,j=1}^{n}(\partial_{ij}^2f)^2$.
\end{notation}

Also, we use  $C$ as a generic constant independent of time $t$ and diffusion
exponent $m$, $C(T)$ or  $C(T,R)$ denote generic constants only depending on the time $T$ or on $T$ and  $R>0$.

\begin{definition}[Weak solution]\label{def:WS}The weak solutions of the PKS model Eq.~\eqref{d1} and the SPKS model Eq.~\eqref{SPKS} are defined as follows:
\\
$\bullet$We recall that a weak solution to Eq.~\eqref{d1} means that for all $T>0$ and all test function $\varphi \in C_0^{\infty}(Q_T)$, such that $\varphi(T)=0$, it holds
\[
  \int_{Q_T} \Big[ \rho_m    \partial_t \varphi +\rho_m^m    \Delta \varphi -\rho_m \nabla \varphi .\nabla {\mathcal N}\ast
\rho_m \Big] dx dt =  \int_{\mathbb{R}^n} \rho_{m,0} \varphi(0) dx .
\]
For that $\rho_m$, $\rho_m^m$ and $\rho_m \nabla {\mathcal N}\ast  \rho_m$ are supposed to be integrable.\\
$\bullet$ A weak solution to Eq.~\eqref{SPKS}
is defined for all test function $\varphi \in C_0^{\infty}(\mathbb{R}^n)$ as
\[
\int_{\mathbb{R}^n} \Big[ \nabla\rho_{m,s}^m\cdot\nabla\varphi + \rho_{m,s}\nabla {\mathcal N}\ast
\rho_{m,s}\cdot\nabla\varphi \Big]dx=0,
\]
where $\nabla\rho_{m,s}^m$ and $\rho_{m,s} \nabla {\mathcal N}\ast  \rho_{m,s}$ are supposed to be integrable.
\end{definition}

\section{Main results}

To state our main results on the incompressible limit of PKS model, we need assumptions on the initial data $\rho_{m,0}$.  Firstly, for  $\rho_{m,0}$, we assume
\begin{equation}\label{c1}
\begin{cases}
\int_{\mathbb{R}^n}  \rho_{m,0} (x) dx =: M <\infty,\quad
\|\rho_{m,0}^{m+1}\|_{L^1(\mathbb{R}^n)}\leq C,\quad\|\rho_{m,0}\|_{L^\infty(\mathbb{R}^n)}<\infty,&\\[2pt]
\|\rho_{m,0}-\rho_{\infty,0}\|_{L^1(\mathbb{R}^n)}\to0,\quad\text{as }m\to\infty,&\\
{\rm supp}(\rho_{m,0})\subset B_{R_m}\text{ for some constant } R_m> 1.&
\end{cases}
\end{equation}
Secondly, for some results, in particular the Aronson-B\'enilan estimate, we also need additional regularity assumptions on the initial data,
\begin{align}
&\|P_{m,0}\|_{L^2(\mathbb{R}^n)}+\| \nabla P_{m,0}\|_{L^2(\mathbb{R}^n)}\leq C,\label{c2}\\[2pt]
&\||\Delta P_{m,0}+\rho_{m,0}|_-\|_{L^1(\mathbb{R}^n)\cap L^2(\mathbb{R}^n)}\leq C.\label{c3}
\end{align}
Furthermore, a compatibility condition is also needed for obtaining the $L^1$ estimate of the time derivative for the pressure,
\begin{equation}\label{c4}
\||(m+1)\rho_{m,0}^{m+1}(\Delta
P_{m,0}+\rho_{m,0})+\nabla\rho_{m,0}^{m+1}\cdot(\nabla
P_{m,0}+\nabla\mathcal{N}\ast\rho_{m,0})|_-\|_{L^1(\mathbb{R}^n)}\leq C.
\end{equation}
Finally, to show the compact support of the solution of the Hele-Shaw limit system Eq.~\eqref{hs}, we need an additional uniform support assumption 
\begin{equation}\label{c5}
{\rm supp}(\rho_{m,0})\subset B_{R_0} \quad \text{ for a fixed constant }R_0>0.
\end{equation}
\begin{remark} Let $\varphi\in C_{0}^{\infty}(\mathbb{R}^n)$ and $0\leq \varphi\leq \frac 1 2$, one immediately  verifies that the initial data $\rho_{m,0}=\varphi$ satisfies the assumptions \eqref{c1}--\eqref{c4}.
\end{remark}

 Assumption \eqref{c1} guarantees global existence of solutions to the Cauchy problem \eqref{d1}  because $m>2-2/n$, as mentioned earlier. 
We also recall in Appendices \ref{AAA} and \ref{sec:comsupp} that solutions satisfy, for some $\mathcal{R}_{m}(T)$, 
\[
\|\rho_m\|_{L^\infty(Q_T)}\leq C(m,T), \qquad  {\rm supp}(\rho_m(T))\subset B_{\mathcal{R}_{m}(T)}, \forall \ T>0. 
\]  

We now gather several uniform regularity estimates, and then establish the stiff limit of the PKS model as $m \to \infty$. 

\begin{theorem}[Uniform bounds and compactness] \label{tAE}
Assume~\eqref{c1}, then the global solution  $\rho_m$ to  the Cauchy problem  Eq.~\eqref{d1} in the sense of Def.~\ref{d1} satisfies for any $T>0$, 
\[
\sup_{0\leq t\leq T}\|\rho_m(t)\|_{L^{q}(\mathbb{R}^n)}+
\|\rho_m^m\|_{L^2(Q_T)}+\|\nabla\rho_m^m\|_{L^2(Q_T)}\leq C(T),\quad \forall\ q\in[1,m+1],
\]
\[
\int_{Q_T}\nabla\rho_m^m\cdot\nabla\rho_m^{p-1}dxdt\leq C(T,p), \quad  1<p\leq2
,\]
\[
\|\rho_{m}\|_{L^{(2+\frac 2n)m+ \frac  2 n} (Q_T)}+\|\nabla P_m\|_{L^2(Q_T)}\leq C(T),
\]
\[\sup_{0\leq t \leq T} \||\rho_m(t)-1|_+\|_{L^2(\mathbb{R}^n)}\leq\frac{C(T)}{\sqrt m},\]
\[
\|\partial_{t}\nabla\mathcal{N}\ast\rho_{m}\|_{L^2(Q_T)}+\|\nabla\mathcal{N}\ast\rho_m\|_{L^{\infty}(Q_T)}
+ \sup\limits_{0\leq t\leq T}\|\nabla\mathcal{N}\ast\rho_{m}(t)\|_{L^2(\mathbb{R}^n)}\leq C(T),
\]
\[
\sup\limits_{0\leq t\leq T}\|\nabla^2\mathcal{N}\ast\rho_m(t)\|_{L^q(\mathbb{R}^n)}\leq C(T,q),
\]
where $C(T,q)\sim \frac{1}{q-1}$ for $0<q-1\ll 1$ and $C(T,q)\sim q$ for $q\gg
1$ and $ m>n-1$.
\end{theorem}

Thanks to these estimates we may extract subsequences, still denoted by $\rho_m$ such that, as $m\to\infty$, $\rho_m$ converges weakly in $L^2(Q_T)$ to
$\rho_{\infty}\in L^\infty\big(0,T;L_+^1(\mathbb{R}^n)\big)$, $P_{m}$ and $\rho_m P_m$ converge weakly in $L^2(Q_T)$ to the same limit $P_{\infty}\in L^2\big(0,T;H^1(\mathbb{R}^n)\big)$, and $\nabla\mathcal{N}\ast\rho_m$ converges strongly in $L^2_{loc}(Q_{T})$ to
$\nabla\mathcal{N}\ast\rho_{\infty}\in\mathcal{C}\big(0,T;L^2(\mathbb{R}^n)\big)\cap L^\infty\big(0,T; L^\infty(\mathbb{R}^n)\cap \dot{W}^{1,q}(\mathbb{R}^n)\big)\text{ for }1<q<\infty$. We have the
\begin{theorem}[Stiff limit] \label{tAEbis}
With assumption \eqref{c1}, this limit, $(\rho_\infty, P_\infty)$ satisfies the Hele-Shaw limit system in the sense of Def.~\ref{def:WS} as
\begin{align}
	&\partial_{t}\rho_{\infty}-\Delta
P_{\infty}=\nabla\cdot(\rho_{\infty}\nabla\mathcal{N}\ast\rho_{\infty}),
&\text{in }\mathcal{D}'(Q_{T}),\label{z6}
\\
&(1-\rho_{\infty})P_{\infty}=0,  \qquad  0\leq\rho_{\infty}\leq1,  &\text{a.e. in } Q_{T}. \label{z8}
\end{align}
\end{theorem}

 Then, using the additional assumptions~\eqref{c2}--\eqref{c3} on the initial data, we obtain the higher regularity estimates on the pressure. We can establish the 
\begin{theorem}[Complementarity relation and semi-harmonicity] \label{t15}
 Assume $m>\max\{n-1,\frac{5n-2}{n+2}\}$  and that the initial data satisfies \eqref{c1}--\eqref{c4}, then the global weak solution $\rho_m$ to  \eqref{d1} satisfies the additional regularity estimates 
 \[
 \|\sqrt{P_m}\nabla P_m\|_{L^2(Q_T)}+\sup\limits_{0\leq t\leq
T}\|\nabla P_m(t)\|_{L^2(\mathbb{R}^n)}\leq C(T),
\]
\[
\|\nabla P_m\|_{L^3(Q_T)}+\|\sqrt{P_m}\nabla^2P_m\|_{L^2(Q_T)}\leq C(T),
\]
\[
\|\sqrt{P_m}\omega_m\|_{L^2(Q_T)}^2+\||\omega_m|_{-}\|_{L^3(Q_T)}^{3}\leq \frac{C(T)}{m}, \qquad \omega_m:=\Delta P_m+\rho_m,
\]
\[
\sup\limits_{0\leq t\leq T}\||\omega_m(t)|_-\|_{L^2\cap L^1(\mathbb{R}^n)}+\sup\limits_{0\leq t\leq T}\|\Delta P_m(t)\|_{L^1(\mathbb{R}^n)}+\|\partial_tP_m\|_{L^1(Q_T)}\leq C(T).
\]

Furthermore, after the extraction of subsequences, as $m\to\infty$, $\nabla P_m$ converges strongly in
$L^2_{loc}(Q_T)$ to $\nabla P_\infty\in L^3(Q_T)\cap L^\infty(0,T;L^2(\mathbb{R}^n))$, and the complementarity relation and semi-harmonicity  hold
\begin{equation}\label{z15}
	P_{\infty}(\Delta P_{\infty}+\rho_{\infty})=0, \qquad \Delta P_{\infty}+\rho_{\infty}\geq 0, \quad \text{ in }
\mathcal{D}'(Q_{T}).
\end{equation}
It follows that
\begin{equation*}
(1-\rho_\infty)\nabla P_\infty=0,\quad\text{a.e. in } Q_T.
\end{equation*}
\end{theorem}

For the Hele-Shaw system Eqs.~\eqref{z6}--\eqref{z8}, the weak solution to the Cauchy problem is unique.
\begin{theorem}[Uniqueness]\label{t14}
Being given two global weak solutions $\rho_{i}$ for $i=1,2$ to the Hele-Shaw system~\eqref{z6}--\eqref{z8} with the initial assumption
$\rho_{1}(x,0)=\rho_{2}(x,0)\in \dot{H}^{-1}(\mathbb{R}^n)$, we have
\begin{equation*}
\rho_1=\rho_2,\quad P_1=P_2,\quad \text{a.e. in }Q .
\end{equation*}
\end{theorem}

Next,   we establish that the limit free energy functional $F_{\infty}\big(\rho_\infty(t)\big)$, 
 with $0\leq \rho_\infty\leq1$,  is non-increasing as time increases.
\begin{theorem}[Compact support and limit energy functional]
Under the initial assumptions~\eqref{c1}--\eqref{c5}, the limit  $(\rho_{\infty},P_{\infty})$, as in
Theorems~\ref{tAEbis}--\ref{t15}, are compactly supported for any finite time. For some $\mathcal{R}_0\geq  \max(R_0, \sqrt{4n + \frac{4n^2}{n-2}})$,  we have
\begin{equation*}
{\rm supp} (P_\infty(t))\subset {\rm supp} (\rho_\infty(t))\subset B_{\mathcal{R}(t)},
\end{equation*}
\[
\mathcal{R}(t):=\big(\mathcal{R}_{0}+n\|\nabla\mathcal{N}\ast\rho_\infty\|_{L^\infty(Q)}\big)e^{\frac{t}{n}}-n\|\nabla\mathcal{N}\ast\rho_\infty\|_{L^\infty(Q)}.
\]
Furthermore, the limit energy dissipation holds for a.e. $t\in[0,\infty)$, 
\begin{equation*}
\frac{dF_{\infty}(\rho_\infty(t))}{dt}+\int_{\mathbb{R}^n}
\rho_{\infty}(t) \big|\nabla(P_\infty(t)+\mathcal{N}\ast\rho_{\infty}(t)) \big|^2dx=0\quad\text{ with }\ 0\leq \rho_\infty\leq 1.
\end{equation*}
\end{theorem}	

 From \cite{r18,rr10,rr12}, we know that the solution to the SPKS model Eq.~\eqref{SPKS} are radially decreasing symmetric up to a translation and compactly supported. This allows us to  gather some useful  a priori estimates in order to prove the compactness for $\rho_{m,s} $, $P_{m,s}$, and $\mathcal{N}\ast\rho_{m,s}$. Then, we can derive  the incompressible limit of the SPKS model Eq.~\eqref{SPKS}.  
 \begin{theorem}[Incompressible limit for stationary state]\label{ILSPKS} Let $m\geq 3$, $\rho_{m,s}$ be a weak solution to the SPKS model Eq.~\eqref{SPKS} in the sense of Def.~\ref{d1} with a given mass $\|\rho_{m,s}\|_{L^1(\mathbb{R}^n)}=M>0$,  $\int_{\mathbb{R}^n}x\rho_{m,s}(x)dx=0$. We define $R_m(M)>0$ satisfying $B_{R_m(M)}:={\rm supp}(\rho_{m,s})$ and $\alpha_m:=\rho_{m,s} (0)=\|\rho_{m,s}\|_{L^\infty(\mathbb{R}^n)}$, then the following regularity estimates hold,
\begin{equation*}
\alpha_{m}^{m-1}\leq\alpha_{m}+\frac{2M}{n(n-2)\omega_n}\leq\frac{1+\sqrt{1+\frac{8M}{n(n-2)\omega_n}}}{2}+\frac{2M}{n(n-2)\omega_n},
 \end{equation*}
 \begin{equation*}
R_m(M)\leq R_*(M),\quad\|\rho_{m,s}\|_{L^1\cap L^\infty(\mathbb{R}^n)}\leq C,
 \end{equation*}
\begin{equation*}
\quad\|\nabla \mathcal{N}\ast\rho_{m,s}\|_{L^2\cap L^\infty(\mathbb{R}^n)}\leq C,\quad\|\nabla^2 \mathcal{N}\ast\rho_{m,s}\|_{L^p(\mathbb{R}^n)}\leq C(M,p),
\end{equation*}
\begin{equation*}
\||\omega_{m,s}|_-\|_{L^3\cap L^1(\mathbb{R}^n)}^3\leq \frac C m, \qquad \quad \omega_{m,s}=\Delta P_{m,s}+\rho_{m,s},
\end{equation*}
\begin{equation*}
\|P_{m,s}\|_{L^1\cap L^\infty(\mathbb{R}^n)}+\|\nabla P_{m,s}\|_{L^1\cap L^\infty(\mathbb{R}^n)}+\|\Delta P_{m,s}\|_{L^1(\mathbb{R}^n)}\leq C,
\end{equation*}
 where $R_*(M)=\log\Big(1+\exp\Big[2n(n-1)\Big(\frac{1+\sqrt{1+\frac{8M}{n(n-2)\omega_n}}}{2}+\frac{2M}{n(n-2)\omega_n}\Big)\Big]^{1/2}\Big)$,  $C(M,p)\sim\frac{1}{p-1}$ for $0<p-1\ll1$ and $C(M,p)\sim p$ for $p\gg1$.

 Furthermore,  after extracting subsequences, as $m\to\infty$, $P_{m,s}$ converges strongly in $L^\infty(\mathbb{R}^n)$ to $P_{\infty,s}\in L^1(\mathbb{R}^n)\cap L^\infty(\mathbb{R}^n)$, $\rho_{m,s}$ converges weakly in $L^p(\mathbb{R}^n)$  for $1< p<\infty$ to $\rho_{\infty,s}\in L^1(\mathbb{R}^n)\cap L^\infty(\mathbb{R}^n)$, $\nabla P_{m,s}$ converges strongly in $L^p(\mathbb{R}^n)$  for $1\leq p<\infty$ to $\nabla P_{\infty,s}\in L^1(\mathbb{R}^n)\cap L^\infty(\mathbb{R}^n)$, and $\nabla \mathcal{N}\ast\rho_{m,s}$ locally converges strongly in $L^p(\mathbb{R}^n)$  for $1\leq p<\infty$ to $\nabla \mathcal{N}\ast\rho_{\infty,s}\in L^2(\mathbb{R}^n)\cap L^\infty(\mathbb{R}^n)\cap\dot{W}^{1,q}(\mathbb{R}^n)$ for $1<q<\infty$.
Therefore, the incompressible (Hele-Shaw) limit of the SPKS model Eq.~\eqref{SPKS} satisfies
\begin{align*}
&\|\rho_{\infty,s}\|_{L^1(\mathbb{R}^n)}=M,&&\int_{\mathbb{R}^n}x\rho_{\infty,s}dx=0,\\
&0\leq \rho_{\infty,s}\leq 1, \quad (1-\rho_{\infty,s})P_{\infty,s}=0,  &&\text{a.e. in }\mathbb{R}^n,\\
&\nabla P_{\infty,s}+\rho_{\infty,s}\nabla\mathcal{N}\ast\rho_{\infty,s}=0, &&\text{a.e. in }\mathbb{R}^n,\\
&\Delta P_{\infty,s}+\rho_{\infty,s}\geq 0,&&\text{in }\mathcal{D}'(\mathbb{R}^n).
\end{align*}
Moreover, it holds for $R(M)>0$ satisfying $|B_{R(M)} |_n=M$ that
\begin{equation*}
\rho_{\infty,s}=\chi_{\{P_{\infty,s}>0\}}=\chi_{B_{R(M)} },\quad \text{a.e. in }\mathbb{R}^n.
\end{equation*}
\end{theorem}
\begin{remark}
The results in Theorem~\ref{ILSPKS} show that the incompressible limit of the SPKS model Eq.~\eqref{SPKS} is the stationary state of the Hele-Shaw problem Eqs.~\eqref{z6}--\eqref{z15}.
\end{remark}

\section{Bounds, compactness and stiff limit}

We  establish the estimates in Theorem~\ref{tAE} and then the stiff limit in Theorem~\ref{tAEbis}.
\\

We begin with the a priori regularity results on the density $\rho_m$, and then treat the nonlocal term.
\begin{lemma}[Regularity estimate on density and pressure] \label{l10}
Assume that the initial data satisfies \eqref{c1}. Let $\rho_m$ be the weak solution to Eq.~\eqref{d1}  in the sense of Def.~\ref{d1}, then it follows
\begin{align}
&\sup_{0\leq t\leq T}\|\rho_m(t)\|_{L^{m+1}(\mathbb{R}^n)}+
\|\nabla\rho_m^m\|_{L^2(Q_T)}\leq C(T),
\label{est:rho1}
\\[2pt]
&\sup_{0\leq t\leq T}\|\rho_m(t)\|_{L^{p}(\mathbb{R}^n)}+\int
\hskip-4pt \int_{Q_T}\nabla\rho_m^m\cdot\nabla\rho_m^{p-1}dxdt\leq C(T), \;
 1<p\leq2,
\label{est:rho2}
\\[2pt]
&\|\rho_{m} \|_{L^{(2+2/n)m+2/n} (Q_T)}+\|\rho_m\nabla
P_m\|_{L^2(Q_T)}\leq C(T),
\label{est:rho3}
\\[2pt]
&\|P_m\|_{L^2(Q_T)}+\|\rho_m^m\|_{L^2(Q_T)}+\|\nabla P_m\|_{L^2(Q_T)}\leq C(T).
\label{est:rho4}
\end{align}
\end{lemma}

\begin{proof}
For \eqref{est:rho1}, we multiply Eq.~\eqref{d1} by $\rho_m^m$ and integrate by parts on
$\mathbb{R}^n$, we find
\begin{equation*}
\begin{aligned}
\frac{1}{m+1} \frac{d} {dt} \int_{\mathbb{R}^n}\rho_m^{m+1}dx+\int_{\mathbb{R}^n}|\nabla\rho_m^m|^2dx&\leq \frac{m}{m+1}\int_{\mathbb{R}^n}\rho_m^{m+2}dx\\
&  \leq \left[ \| \rho_m\|_{L^{1}(\mathbb{R}^n)}^{1-\theta}  \|\rho_m\|_{L^{\frac{2mn}{n-2}}(\mathbb{R}^n)}^{\theta}  \right]^{m+2} \\
& \leq  \| \rho_m\|_{L^{1}(\mathbb{R}^n)}^{(1-\theta)(m+2)}  \|\rho_m^m\|_{L^{\frac{2n}{n-2}}(\mathbb{R}^n)}^{\theta \frac{m+2}{m}}, 
\end{aligned}
\end{equation*}
where we have used interpolation inequality with $\frac 1{m+2}=1-\theta+ \theta  \frac{n-2}{2mn}$. 
We notice that  $(1-\theta) (m+2) \leq 1$. Using this and Sobolev's inequality (Theorem~\ref{t8}), we get
\begin{align}
\frac{1}{m+1} \frac d {dt} \int_{\mathbb{R}^n}\rho_m^{m+1}dx+
\int_{\mathbb{R}^n}|\nabla\rho_m^m|^2dx
\leq& \max (1,\| \rho_m\|_{L^{1}(\mathbb{R}^n)})   \| \nabla
\rho_m^m\|_{L^{2}(\mathbb{R}^n)}^{\theta \frac{m+2}{m}} \notag
\\[2pt]
 \leq& C + \frac 12 \| \nabla \rho_m^m\|_{L^{2}(\mathbb{R}^n)}^2, \notag
\label{m1}
\end{align}
where we have used ${\theta \frac{m+2}{m}}<2$ for $m>2- \frac 2 n$.

After time integration, we get the inequality
\begin{equation}
 \int_{\mathbb{R}^n}\rho_m(t)^{m+1}dx + \frac {m+1}2 \int_0^t\ \| \nabla
 \rho_m^m\|_{L^{2}(\mathbb{R}^n)}^2
 \leq \int_{\mathbb{R}^n}\rho_{m,0}^{m+1}dx + Ct(m+1),
 \label{est:m+1}
\end{equation}
which implies \eqref{est:rho1} in Lemma~\ref{l10};
\[
\|\rho_m(t)\|_{L^{m+1}(\mathbb{R}^n)}\leq (C+ Ct(m+1))^{\frac 1{m+1}} \leq C(T).
\]

A similar calculation\ gives \eqref{est:rho2} and we have
\[
\frac{1}{p}\int_{\mathbb{R}^n}\rho_m(t)^pdx+\frac{4m(p-1)}{(m+p-1)^2}  
\int_{Q_t} |\nabla\rho_m^{\frac{m+p-1}{2}}|^2dxds=
\frac{1}{p}\int_{\mathbb{R}^n}\rho_{m,0}^pdx  +\frac{p-1}{p}  \int_{Q_t}  \rho_m^{p+1}dxds.
 \]
Interpolating between $L^{m+1}(\mathbb{R}^n)$ and $L^{1}(\mathbb{R}^n)$, we know that the terms on the right hand
side are controlled and thus the gradient term is under control. It remains to
notice that
\[
\frac{4m(p-1)}{(m+p-1)^2}\int_{Q_T}
|\nabla\rho_m^{\frac{m+p-1}{2}}|^2dx
=m(p-1)\int_{Q_T} \rho_m^{m+p-3} |\nabla\rho_m|^2dx
=\int_{Q_T} \nabla\rho_m^m .\nabla\rho_m^{p-1}dx,
\]
and \eqref{est:rho2}  is proved.
\\

We turn to \eqref{est:rho3}. Thanks to the interpolation inequality, we have, for
$\alpha \geq 0$ and $t\leq T$,
\begin{equation*}
\begin{aligned}
 \int_{\mathbb{R}^n}\rho_{m}(t)^{2m+\alpha}dx&\leq
\| \rho_{m}(t)\|_{L^{m+1}(\mathbb{R}^n)}^{(1-\theta) (2m+\alpha)}
\| \rho_m(t)\|_{L^{\frac{2mn}{n-2}}(\mathbb{R}^n)}^{\theta (2m+\alpha)}
\\
&= \| \rho_{m}(t) \|_{L^{m+1}(\mathbb{R}^n)}^{(1-\theta) (2m+\alpha)}
\| \rho_m^m(t) \|_{L^{\frac{2n}{n-2}}(\mathbb{R}^n)}^{\theta \frac{2m+\alpha}{m}}
\end{aligned}
\end{equation*}
with
\[
\frac{1}{2m+\alpha}= \frac{1-\theta}{m+1} + \frac{\theta (n-2)}{2mn}, \qquad
0\leq \theta \leq 1.
\]
By Sobolev's inequality and the estimate in $L^{m+1}$, we obtain
\[
 \int_{\mathbb{R}^n}\rho_{m}(t)^{2m+\alpha}dx \leq C(T) ^{(1-\theta) (2m+\alpha)}
 \| \nabla \rho_m^m\|_{L^{2}(\mathbb{R}^n)}^{\theta \frac{2m+\alpha}{m}}.
\]
It remains to choose $\alpha$ such that $\theta \frac{2m+\alpha}{m}=2$ and we
find, integrating in time,
\[
\int_{Q_T} \rho_{m}(t)^{2m+\alpha}dx \leq C(T) ^{(1-\theta)
 (2m+\alpha)}
 \int_{Q_T} | \nabla \rho_m^m|^2dx.
\]
To compute the value of $\alpha$, we write the condition successively as
\[
2m= (2m+\alpha)\theta =  (2m+\alpha) \big[\frac 1 {m+1}-\frac{1}{2m+\alpha} \big]
\big[  \frac{1}{m+1}  - \frac{n-2}{2nm} \big]^{-1},
\]
\[
2m= \big[ 2m+\alpha -(m+1) \big] \big[  1- \frac{(n-2)(m+1)}{2mn} \big]^{-1},
\]
\[
2m  - \frac{(n-2)(m+1)}{n}= 2m+\alpha - (m+1) , \qquad \alpha= \frac 2 n (m+1).
\]
This gives the first statement of \eqref{est:rho3}. Then, since $\nabla\rho_m^m=\rho_m\nabla P_m$, we have from \eqref{est:rho1}
$\|\rho_m\nabla P_m\|_{L^2(Q_T)}\leq C(T)$ and  \eqref{est:rho3} is proved.
\\

The first estimate of \eqref{est:rho4} is obtained by interpolation and  Sobolev's inequality for gradient (Theorem~\ref{t8}) between two estimates in~\eqref{est:rho1}, for $\gamma=0,1$.
\begin{equation*}
\begin{aligned}
\int_{Q_T} \rho_m^{2m-2\gamma}dxdt &\leq \big(\int_{Q_T} \rho_m^{2m+1}dxdt\big)^{\frac{2m-2\gamma-1}{2m}} \big(\int_{Q_T} \rho_mdxdt \big)^{\frac{2\gamma+1}{2m}}\\
&\leq C(T)\int_{0}^{T}\hskip-6pt \big(\int_{\mathbb{R}^n}(\rho_m^m)^{\frac{2n}{n-2}}dx\big)^{\frac{n-2}{n}} \big(\int_{\mathbb{R}^n}\rho_m^\frac{n}{2}dx \big)^{\frac{2}{n}}dt\\
&\leq C(T)\int_{Q_T} |\nabla\rho_m^{m}|^2dxdt\leq C(T).
\end{aligned}
\end{equation*}
 We prove the second estimate of \eqref{est:rho4} by means of the estimates $\eqref{est:rho1}_{2nd}$--$\eqref{est:rho2}_{2nd}$,
\begin{equation*}
\begin{aligned}
\int_{Q_T} |\nabla P_m|^2dxdt=&m^2\int_{Q_T} \rho_m^{2(m-2)}|\nabla \rho_m|^2dxdt\\
\leq&\int_{Q_T} (\frac{2m^2}{m-1}\rho_m^{m-1}+\frac{m^2(m-3)}{m-1}\rho_m^{2m-2}|\nabla\rho_m|^2dxdt\\
=&\int_{Q_T} (\frac{2m}{m-1}\nabla\rho_m^{m}\cdot\nabla\rho_m+\frac{m-3}{m-1}|\nabla\rho_m^m|^2)dxdt \leq C(T).
\end{aligned}
\end{equation*}
\end{proof}

Now, we turn to the nonlocal term.
\begin{lemma}[Regularity of the nonlocal term] \label{l18}
Assume  $m>n-1$ and \eqref{c1}, let $\rho_m$ be a weak solution to Eq.~\eqref{d1}  in the sense of Def.~\ref{d1}, then it
holds for any $T>0$ that
\begin{align}
&\|\nabla\mathcal{N}\ast\rho_m\|_{L^{\infty}(Q_T)}\leq
C(T),&&
\|\partial_{t}\nabla\mathcal{N}\ast\rho_{m}(t)\|_{L^2(Q_T)}\leq C(T),\label{est:con1}\\
&\sup\limits_{0\leq t\leq
T}\|\nabla^2\mathcal{N}\ast\rho_m(t)\|_{L^q(\mathbb{R}^n)}\leq C(T,q),
&&\sup\limits_{0\leq t\leq
T}\|\nabla\mathcal{N}\ast\rho_{m}(t)\|_{L^2(\mathbb{R}^n)}\leq C(T),\label{est:con2}
\end{align}
where $C(T,q)\sim \frac{1}{q-1}$ for $0<q-1\ll 1$ and $C(T,q)\sim q$ for $q\gg
1$. Furthermore, after extraction,it holds  that
	\begin{equation*}
	\nabla\mathcal{N}\ast\rho_{m}\to\nabla\mathcal{N}\ast\rho_{\infty},\ stongly\
in\ L^2_{loc}(Q_{T}),\ as\ m\to\infty.
	\end{equation*}
\end{lemma}
\begin{proof}
For the first estimate of \eqref{est:con1}, by means of Lemma~\ref{l10}, we obtain the $L^{\infty}$ estimate for
$\nabla\mathcal{N}\ast\rho_m(t)$ with $m>n-1$  because
\begin{equation}\label{m22}
\begin{aligned}
|\nabla\mathcal{N}\ast\rho_{m}(t)| &\leq
C\int_{\mathbb{R}^n}\frac{\rho_{m}(y,t)}{|x-y|^{n-1}}dy
\\
&\leq C\int_{|x-y|\leq
1}\frac{\rho_{m}(y,t)}{|x-y|^{n-1}}dy+C\int_{|x-y|>1}\frac{\rho_{m}(y,t)}{|x-y|^{n-1}}dy
\\
&\leq
C\Big(\int_{|x-y|\leq1}\frac{1}{|x-y|^{\frac{(n-1)(n+\varepsilon)}{n-1+\varepsilon}}}dy\Big)^{\frac{n-1+\varepsilon}{n+\varepsilon}}\Big(\int_{|x-y|\leq1}\rho_{m}^{n+\varepsilon}(y,t)dy\Big)^{\frac{1}{n+\varepsilon}}+
C
\\
&\leq C(T)\qquad  \forall t\in[0,T]\text{ and for some }0<\varepsilon\ll 1.
\end{aligned}
\end{equation}

Let the Laplace inverse operator $\Delta^{-1}=\mathcal{N}\ast$ act on
Eq.~\eqref{d1}, we get a new equation
 \begin{equation}\label{m10}
 \partial _t\mathcal{N}\ast\rho_m=
 \rho_m^m+\nabla\cdot\mathcal{N}\ast(\rho_m\nabla\mathcal{N}\ast\rho_m).
 \end{equation}
 Then, using the singular integral theory for Newtonian potential
 (Lemma~\ref{l12}), \eqref{m22}, and Lemma~\ref{l10}, we obtain
\begin{equation}\label{m23}
\begin{aligned}
\int_{\mathbb{R}^n}|\nabla\nabla\cdot\mathcal{N}\ast(\rho_m\nabla\mathcal{N}\ast\rho_m)|^{2}dx\leq& C\int_{\mathbb{R}^n}|\rho_m\nabla\mathcal{N}\ast\rho_m|^{2}dx\\
\leq&C\|\nabla\mathcal{N}\ast\rho_m\|^2_{L^{\infty}(Q_T)}\int_{\mathbb{R}^n}\rho_m^{2}dx\\
\leq& C(T).
\end{aligned}
\end{equation}
Due to Eq. \eqref{m10}, we use \eqref{m23} and Lemma~\ref{l10}, then it follows
\begin{equation*}
\|\partial_t\nabla\mathcal{N}\ast\rho_m\|_{L^2(Q_T)}\leq \|\nabla
\rho_m^m\|_{L^2(Q_T)}+\|\nabla\nabla\cdot\mathcal{N}\ast(\rho_m\nabla\mathcal{N}\ast\rho_m)\|_{L^2(Q_T)}\leq
C(T)
\end{equation*}
and the second bound of \eqref{est:con1} is proved.

For the first estimate of \eqref{est:con2}, we again use the singular integral theory for Newtonian
potential (Lemma~\ref{l12}), and we have  for all $t\in[0,T]$,
\begin{equation}\label{m25}
\|\nabla^2\mathcal{N}\ast\rho_m(t)\|_{L^q(\mathbb{R}^n)}\leq
C(q)\|\rho_m\|_{L^q(\mathbb{R}^n)}\leq C(T,q),
\end{equation}
where $C(T,q)\sim \frac{1}{q-1}$ for $0<q-1\ll 1$ and $C(T,q)\sim q$ for $q\gg
1$.
\\

And for the second bound of \eqref{est:con2}, thanks to the Hardy-Lilttlewood-Sobolev inequality
(Theorem~\ref{t7}) and Lemma~\ref{l10},  we get for all $t\in[0,T]$,
\begin{equation}\label{m43}
\begin{aligned}
\int_{\mathbb{R}^n}|\nabla\mathcal{N}\ast\rho_{m}(t)|^2dx=&-\int_{\mathbb{R}^n}(\Delta\mathcal{N}\ast\rho_{m})\mathcal{N}\ast\rho_{m}dx
=-\int_{\mathbb{R}^n}\rho_{m}\mathcal{N}\ast\rho_{m}dx\\
\leq& C(n)\|\rho_{m}\|_{L^{\frac{2n}{n+2}}(\mathbb{R}^n)}^2
\leq C(T).
\end{aligned}
\end{equation}

The last statement of Lemma \ref{l18}  follows from  Sobolev's compactness embeddings.
\end{proof}

In order to obtain convergence rate  on $|\rho_{m}-1|_+$ in Theorem~\ref{tAE}, it remains to establish the
\begin{lemma}[convergence rate  on $|\rho_{m}-1|_+$]$\label{l6}$
	Under the initial assumptions \eqref{c1}, let $\rho_{m}$ be the weak solution
to the Cauchy problem Eq.~\eqref{d1}  in the sense of Def.~\ref{d1} with $m>n-1$, then
	\begin{equation*}
\sup_{0\leq t \leq T} \||\rho_{m}(t)-1|_+\|_{L^2(\mathbb{R}^n)}\leq
\frac{C(T)}{\sqrt m}.
	\end{equation*}

\end{lemma}

\begin{proof}
Since, for $m>n-1\geq 2$ and $\rho_m \geq 1$, we have
\begin{equation*}
\rho_m^{m+1} \geq\frac{m(m+1)}{2}(\rho_m-1)^{2},
\end{equation*}
we conclude
\begin{equation*}
sgn(|\rho_m-1|_+)\rho_m^{m+1}\geq \frac{m(m+1)}{2}|\rho_m-1|_+^{2}.
\end{equation*}

From \eqref{est:m+1}, we obtain
\begin{align*}
\int_{\mathbb{R}^n}|\rho_m(t)-1|_+^2dx\leq \frac{2}{m(m+1)}
\int_{\mathbb{R}^n}\rho_{m}^{m+1}(t)dx\leq\frac{C(T)}{m}\quad\text{for all }0\leq t\leq T.
\end{align*}
\end{proof}
\begin{remark}
	The result of Lemma~\ref{l6} implies that larger diffusion exponent means
stronger diffusive effect on the zone of high density.
\end{remark}


In the following, with  the regularity estimates in Lemmas~\ref{l10}--\ref{l6} in hand, we prove the stiff limit statements in
Theorem~\ref{tAEbis}.

We recall that, thanks to the a priori regularity estimates in
Lemmas~\ref{l10}--\-ref{l6},  $\rho_m$  has a weak limit $\rho_\infty$ with $
\rho_\infty \leq 1$ in $L^p(Q_T)$ for $1<p<\infty$, $P_m$ has a weak limit $P_\infty$ in $L^2(Q_T)$, and  we
have locally strong convergence of $\nabla \mathcal{N}\ast\rho_{m}$ to $\nabla
\mathcal{N} \ast \rho_{\infty}$ in $L^2(Q_T)$.
\\

\noindent\emph{Proof of Eq.~\eqref{z6}}. The stiff limit equation \eqref{z6} in Theorem~\ref{tAEbis}  follows immediately with these weak limits and the definition of weak solutions in  Def.~\ref{def:WS},
 where the nonlinear term $\rho_m \nabla \mathcal{N}\ast\rho_{m}$ can pass to the limit $\rho_\infty\nabla\mathcal{N}\ast\rho_\infty$ by weak-strong convergence, and the another nonlinear term $\rho_m^m=\frac{m-1}{m}\rho_m P_m$ weakly converges to $P_\infty$ in $L^2(Q_T)$ from \eqref{m37}.
\\

\noindent\emph{Proof of Eq.~\eqref{z8}, $\rho_\infty P_\infty=P_\infty,\text{ a.e. } (x,t)\in Q_T$}. For the case of  tumor growth model in~\cite{5}, the proof is obtained because $\rho_m$  converges strongly, which is not available here. Therefore, we argue in two steps. We  firstly prove that after extraction,
 \begin{equation}\label{m37}
 \rho_m P_m\rightharpoonup P_\infty, \text{ weakly in }L^2(Q_T), \text{ as
 }m\to\infty.
 \end{equation}
For that, thanks to the relation $\rho_m P_m=\frac{m}{m-1}\rho_m^{m}\leq 2\rho_m^{m}$, it
 follows from Lemma~\ref{l10} that  $\rho_mP_m$ is bounded in $L^2(Q_T)$ and
 thus, after extraction, has a weak limit in $L^2(Q_T)$, which we call
 $Q_\infty$.

Due to Young's inequality, we have
 \begin{equation*}
 P_{m}=\frac{m}{m-1}\rho_m^{m-1}\leq
 \frac{m}{m-1}(\frac{m-1}{m}\rho_m^m+\frac{1}{m})=\frac{m-1}{m}\rho_mP_m+\frac{1}{m-1}.
 \end{equation*}
 In the weak limit, we obtain
 \[
 P_\infty \leq Q_\infty.
 \]

For the reverse inequality, we consider $A>1$ and $m$
sufficiently large. Then, we have
 \begin{equation}
 \rho_mP_m=\rho_m\min\{A,P_m\}+\rho_m|P_m-A|_+ \leq
 (\frac{m-1}{m}A)^\frac{1}{m-1}P_m+\rho_m|P_m-A|_+.
 \label{eqAPm}
 \end{equation}
 We can estimate the last term by 
\begin{equation*}
\rho_m|P_m-A|_+\leq
\chi_{\{P_m>A\}}\frac{m}{m-1}\rho_m^m=\chi_{\{P_m>A\}}\frac{m}{m-1}\frac{\rho_m^{2m}}{\rho_m^{m}}\leq
\frac{\rho_m^{2m}}{(\frac{m-1}{m}A)^{m/(m-1)}},
 \end{equation*}
and thus, for any non-negative smooth test function $\varphi\in
C_0^{\infty}(Q_T)$, we conclude
 \begin{equation*}
 \limsup\limits_{m\to\infty} \int_{Q_T}\rho_m|P_m-A|_+\varphi
 dxdt\leq \limsup\limits_{m\to\infty}
\int_{Q_T}\frac{\rho_m^{2m}}{(\frac{m-1}{m}A)^{m/(m-1)}} dxdt\leq \frac  C A.
 \end{equation*}
On the other hand, $(\frac{m-1}{m}A)^\frac{1}{m-1}$ converges strongly to $1$.
Therefore by weak-strong convergence $(\frac{m-1}{m}A)^\frac{1}{m-1} P_m$ weakly
converges to $P_\infty$. Passing to the weak limit in \eqref{eqAPm}, we conclude that, for all $A>1$
\[
Q_\infty \leq P_\infty +\frac C A .
\]
We may take $A \to \infty$ and find the desired result, namely \eqref{m37}.
\\

 Secondly, we prove that  $ \rho_m P_m\rightharpoonup \rho_\infty P_\infty$. For any
 smooth test function $\varphi\in C_{0}^{\infty}(Q_T)$, we have, recalling the
 strong convergence proved in Lemma~\ref{l18},
 \begin{equation}\label{m38}
 \begin{aligned}
 \lim\limits_{m\to\infty}&\int_{Q_T}\rho_{m}P_{m}\varphi
 dxdt= \lim\limits_{m\to\infty}  \int_{Q_T}\Delta\mathcal{N}\ast\rho_{m}P_{m}\, \varphi dxdt\\
 =&-\lim\limits_{m\to\infty}  
 \int_{Q_T}\nabla\mathcal{N}\ast\rho_{m}\cdot\nabla P_{m} \, \varphi
 dxdt-\lim\limits_{m\to\infty}   \int_{Q_T}\nabla\mathcal{N}\ast\rho_{m}\cdot\nabla\varphi P_{m} dxdt\\
 =&-\int_{Q_T}\nabla\mathcal{N}\ast\rho_{\infty}\cdot\nabla
 P_{\infty}\varphi dxdt-  \int_{Q_T}\nabla\mathcal{N}\ast\rho_{\infty}\cdot\nabla\varphi P_{\infty} dxdt\\
 =&\int_{Q_T}\rho_{\infty} P_{\infty} \varphi dxdt.
 \end{aligned}
 \end{equation}
 This means that $\rho_m P_m\rightharpoonup \rho_\infty P_\infty$ and we have
 obtained the result.
 \\
 
 \noindent\emph{Proof of Eq.~\eqref{z8},  $0\leq \rho_\infty\leq1$,  a.e. in $Q_T$.}  It is directly obtained by Lemma~\ref{l6} and theinequality  $\|\rho_\infty\|_{L^2(Q_T)}\leq \liminf\limits_{m\to\infty}\|\rho_m\|_{L^2(Q_T)} \leq C(T)$.

\section{Additional regularity estimates for pressure}

The classical Aronson-B\'enilan(AB) estimate \cite{r32,5} provides regularity for the pressure $P_m$. But the
nonlocal interaction results in the absence of comparison principle and the
$L^{\infty}$ bound from below  are missing.
Therefore, we prove uniform AB-type estimates in $L^3\&L^1$ versions, adapting the method in \cite{r35,r34}.
This refgularity is interesting by itself and is used to establish the complementarity relation Eq.~\eqref{z15} which is equivalent to proving
the strong compactness of the sequence $\{\nabla P_m\}_{m>1}$ in $L^2_{loc}(Q_T)$.
\\

 In this section, we need to further assume $m>\max \{n-1,\frac{5n-2}{n+2}\}$ because of inequality~\eqref{a53}.

\begin{lemma}\label{l11}
Under the initial assumptions
\eqref{c1}--\eqref{c2}, let $\rho_m$ be a weak solution to the Cauchy problem Eq.~\eqref{d1}  in the sense of Def.~\ref{d1}, then it holds
\begin{equation}\label{a16}
\|\sqrt{P_m}\nabla P_m\|_{L^2(Q_T)}\leq C(T) \qquad \forall T>0.
\end{equation}
\end{lemma}
\begin{proof}
Multiplying Eq. \eqref{d3} by $P_m$ and integrating on $\mathbb{R}^n$, then we
have \begin{equation}\label{a14}
\frac{1}{2}\frac{d}{dt}\int_{\mathbb{R}^n}P_m^2dx+(2m-3)\int_{\mathbb{R}^n}P_m|\nabla
P_m|^2dx=(m-\frac{3}{2})\int_{\mathbb{R}^n}P_m^2\rho_mdx.
\end{equation}
Due to Sobolev's inequality for gradient (Theorem~\ref{t8}), Holder's inequality,
and Lemma~\ref{l10}, we obtain
\begin{equation}\label{a15}
\begin{aligned}
(m-\frac{3}{2})\int_{\mathbb{R}^n}P_m^2\rho_mdx&\leq
(m-\frac{3}{2})(\int_{\mathbb{R}^n}P_m^{\frac{2n}{n-2}}dx)^{\frac{n-2}{n}}(\int_{\mathbb{R}^n}\rho_m^\frac{n}{2}dx)^{\frac{2}{n}}\\
&\leq (m-\frac{3}{2})C(n)C(T)\int_{\mathbb{R}^n}|\nabla P_m|^2dx.
\end{aligned}
\end{equation}
Taking \eqref{a15} into  \eqref{a14} and integrating \eqref{a14} on $[0,T]$ gives
\eqref{a16}.
 \end{proof}

We are going to establish the uniform $L^3$ estimate for the pressure gradient. Recently, David and Perthame \cite[Theorem 3.2]{r35} proved a uniform sharp
 $L^4$ estimate for the pressure gradient. In contrast, we obtain here a uniform
 $L^3$ estimate for the pressure grdient by adapting their proof  to take into account that the nonlocal
 interaction term resulting  in the absence of a uniform bound for the pressure.
\begin{theorem}[$L^3$ estimate for pressure gradient]\label{t10}
Under the initial assumptions \eqref{c1}--\eqref{c2}, let $\rho_m$ be a weak solution to the Cauchy problem Eq.~\eqref{d1}  in the sense of Def.~\ref{d1}, then it holds for any
       given $T>0$ and $m>n-1$ that
       \begin{align}
       &\sup_{0\leq t\leq T}\|\nabla P_m(t)\|_{L^2(\mathbb{R}^n)}\leq
       C(T),&&\|\sqrt{P_m}(\Delta
       P_m+\rho_m)\|_{L^2(Q_T)}\leq\frac{C(T)}{\sqrt{m}},\label{est:pg1}\\
       &\|\sqrt{P_m}\nabla^2P_m\|_{L^2(Q_T)}\leq C(T),&&
       \|\nabla P_m\|_{L^3(Q_T)}\leq C(T)\label{est:pg2}.
\end{align}
\end{theorem}

 \begin{proof}
 We multiply the pressure Eq.~\eqref{d3} by $-(\Delta P_m+\rho_m)$
 and integrate that on $\mathbb{R}^n$, then it follows
\begin{equation}\label{a2}
\begin{aligned}
\frac{1}{2}\frac{d}{dt}& \int_{\mathbb{R}^n}|\nabla
P_m|^2dx-\partial_t\int_{\mathbb{R}^n}\rho_m^mdx+(m-1)\int_{\mathbb{R}^n}P_m(\Delta
P_m+\rho_m)^2dx\\
&+\int_{\mathbb{R}^n}|\nabla P_m|^2\Delta P_mdx+\int_{\mathbb{R}^n}|\nabla
P_m|^2\rho_mdx
+\int_{\mathbb{R}^n}\nabla P_m\cdot\nabla\mathcal{N}\ast\rho_m\Delta P_mdx\\
&+\int_{\mathbb{R}^n}\rho_m\nabla P_m\cdot\nabla\mathcal{N}\ast\rho_mdx=0.
\end{aligned}
\end{equation}
Integrating by parts, we have
\begin{equation*}
\begin{aligned}
\int_{\mathbb{R}^n}  |\nabla P_m|^2  \Delta P_mdx &=\int_{\mathbb{R}^n}P_m\Delta(|\nabla P_m|^2)dx\\
=&2\int_{\mathbb{R}^n}P_m\nabla P_m\cdot\nabla(\Delta
P_m)dx+2\int_{\mathbb{R}^n}P_m(\nabla^2P_{m})^2dx\\
=&-2\int_{\mathbb{R}^n}P_m|\Delta P_m|^2dx-2\int_{\mathbb{R}^n}|\nabla
P_m|^2\Delta P_mdx
+2\int_{\mathbb{R}^n}P_m(\nabla^2P_m)^2dx.
\end{aligned}
\end{equation*}
Hence, it holds
\begin{equation}\label{a3}
\int_{\mathbb{R}^n}|\nabla P_m|^2\Delta
P_mdx=-\frac{2}{3}\int_{\mathbb{R}^n}P_m|\Delta
P_m|^2dx+\frac{2}{3}\int_{\mathbb{R}^n}P_m(\nabla^2P_m)^2dx.
\end{equation}
Similarly, integrating by parts, we obtain
\begin{equation}\label{a4}
\begin{aligned}
\int_{\mathbb{R}^n} & \nabla P_m\cdot\nabla\mathcal{N}\ast\rho_m\Delta P_mdx\\
=&-\sum\limits_{i,j}\int_{\mathbb{R}^n}\partial_{ij}^2
P_m\partial_{i}\mathcal{N}\ast\rho_m\partial_{i}P_mdx-\sum\limits_{ij}\int_{\mathbb{R}^n}\partial_i
P_m\partial_{ij}^2\mathcal{N}\ast\rho_m\partial_{j}P_mdx\\
=&\frac{1}{2}\int_{\mathbb{R}^n}|\nabla
P_m|^2\rho_mdx+\int_{\mathbb{R}^n}P_m\nabla\rho_m\cdot\nabla
P_mdx+\int_{\mathbb{R}^n}P_m\nabla^2\mathcal{N}\ast\rho_m:\nabla^2 P_mdx\\
=&\frac{3m-1}{2m-2}\int_{\mathbb{R}^n}\nabla
P_m\cdot\nabla\rho_m^mdx+\int_{\mathbb{R}^n}P_m\nabla^2\mathcal{N}\ast\rho_m:\nabla^2
P_mdx.
\end{aligned}
\end{equation}
Thus, inserting both \eqref{a3} and \eqref{a4} into \eqref{a2}, we have
\begin{equation}\label{a7}
\begin{aligned}
\frac{1}{2}\frac{d}{dt}&  \int_{\mathbb{R}^n}|\nabla P_m|^2dx  -\frac{d}{dt}  \int_{\mathbb{R}^n}\rho_m^mdx+(m-1)\int_{\mathbb{R}^n}P_m(\Delta
P_m+\rho_m)^2dx\\
&\qquad \qquad +\frac{2}{3}\int_{\mathbb{R}^n}P_m(\nabla^2P_m)^2dx+\frac{3m-1}{2m-2}\int_{\mathbb{R}^n}\nabla
P_m\cdot\nabla\rho_m^mdx\\
\leq&\frac{2}{3}\int_{\mathbb{R}^n}P_m|\Delta
P_m|^2dx-\int_{\mathbb{R}^n}P_m\nabla^2P_m:\nabla^2\mathcal{N}\ast\rho_mdx\\
\leq& \frac{2}{3}\int_{\mathbb{R}^n}P_m|\Delta
P_m|^2dx+\frac{1}{3}\int_{\mathbb{R}^n}P_m(\nabla^2P_m)^2dx+\frac{3}{4}\int_{\mathbb{R}^n}P_m(\nabla^2\mathcal{N}\ast\rho_m)^2dx,
\end{aligned}
\end{equation}
where the last inequality follows from
\[
\big| -\int_{\mathbb{R}^n}\nabla^2P_m:\nabla^2\mathcal{N}\ast\rho_m
P_m dx \big| \leq\frac{1}{3}\int_{\mathbb{R}^n}P_m(\nabla^2P_m)^2dx+\frac{3}{4}\int_{\mathbb{R}^n}P_m(\nabla^2\mathcal{N}\ast\rho_m)^2dx.
\]
It easily follows from Lemma~\ref{l10} and Sobolev's inequality that
\begin{equation}\label{a8}
\begin{aligned}
\int_{\mathbb{R}^n}\rho_m^mdx\leq \int_{\mathbb{R}^n}P_{m}\rho_{m}dx
\leq \frac{1}{4}\int_{\mathbb{R}^n}|\nabla P_m|^2dx+C(T).
\end{aligned}
\end{equation}
Similarly, thanks to Lemma~\ref{l18}, the singular integral theory for Newtonian
potential (Lemma~\ref{l12}), Holder's inequality and Young's inequality, then we
have
\begin{equation}\label{a9}
\begin{aligned}
\frac{3}{4}\int_{\mathbb{R}^n}P_m(\nabla^2\mathcal{N}\ast\rho_m)^2dx
\leq&\frac{3}{4}\sum_{ij}(\int_{\mathbb{R}^n}P_m^{\frac{2n}{n-2}}dx)^{\frac{n-2}{2n}}(\int_{\mathbb{R}^n}|\partial_{ij}^2\mathcal{N}\ast\rho_m|^{\frac{4n}{n+2}}dx)^{\frac{n+2}{2n}}
\\
\leq&\frac{3}{4}\sum_{ij}(C(n)\int_{\mathbb{R}^n}|\nabla
P_m|^2dx)^{\frac{1}{2}}(C(n)\int_{\mathbb{R}^n}\rho_m^{\frac{4n}{n+2}}dx)^{\frac{n+2}{2n}}
\\
\leq&\frac{1}{8}\int_{\mathbb{R}^n}|\nabla
P_m|^2dx+\frac{9}{8}C(n)n^2(C(n)\int_{\mathbb{R}}\rho_m^{\frac{4n}{n+2}}dx)^{\frac{n+2}{n}}\\
\leq&\frac{1}{8}\int_{\mathbb{R}^n}|\nabla P_m|^2dx+C(T).
\end{aligned}
\end{equation}
Integrating \eqref{a7} on $[0,t]$ for any $t\in(0,T]$ and using both \eqref{a8}
and \eqref{a9}, then we obtain
\begin{equation}\label{a10}
\begin{aligned}
\frac{1}{4}\int_{\mathbb{R}^n} & |\nabla
P_m(t)|^2dx+\int_{\mathbb{R}^n}\rho_{m,0}^mdx+(m-\frac{7}{3})\int_{0}^{t}\int_{\mathbb{R}^n}P_m(\Delta
P_m+\rho_m)^2dxds\\
&\quad +\frac{1}{3}\int_{0}^{t}\int_{\mathbb{R}^n}P_m(\nabla^2
P_m)^2dxds+\frac{3m-1}{2m-2}\int_{0}^{t}\int_{\mathbb{R}^n}\nabla
P_m\cdot\nabla\rho_m^mdxds\\
\leq&\frac{4}{3}\int_{0}^{t}\int_{\mathbb{R}^n}P_m\rho_m^2dxds+\frac{1}{8}\int_{0}^{t}\int_{\mathbb{R}^n}|\nabla
P_m|^2dxds+C(T)t+\frac{1}{2}\int_{\mathbb{R}^n}|\nabla P_{m,0}|^2dx,
\end{aligned}
\end{equation}
where $\frac{2}{3}\int_{0}^{t}\int_{\mathbb{R}^n}P_m(\Delta
P_m)^2dxds\leq\frac{4}{3}\int_{0}^{t}\int_{\mathbb{R}^n}P_m\rho_m^2dxds
+\frac{4}{3}\int_{0}^{t}\int_{\mathbb{R}^n}P_m(\Delta P_m+\rho_m)^2dxds$ is used.
It easily follows from Lemma~\ref{l10} and Sobolev's inequality that
\begin{equation*}
\begin{aligned}
\frac{4}{3}\int_{0}^{t}\int_{\mathbb{R}^n}P_m\rho_m^2dxds\leq
\frac{1}{3}\int_{0}^{t}\int_{\mathbb{R}^n}|\nabla P_m|^2dxds+C(T)\leq C(T).
\end{aligned}
\end{equation*}
Inserting this  into \eqref{a10} and by virtue of Lemma~\ref{l10}, we have,
for all $t\in[0,T]$,
\begin{equation}\label{a12}
\begin{aligned}
\frac{1}{4}\int_{\mathbb{R}^n} & |\nabla
P_m(t)|^2dx+(m-\frac{7}{3})\int_{0}^{t}\int_{\mathbb{R}^n}P_m(\Delta
P_m+\rho_m)^2dxds\\
&+\frac{1}{3}\int_{0}^{t}\int_{\mathbb{R}^n}P_m(\nabla^2 P_m)^2dxds\leq C(T) .
\end{aligned}
\end{equation}
Therefore, it follows from \eqref{a12} that
\[
\sup_{0\leq t\leq T} \|\nabla P_m(t)\|_{L^2(Q_T)}^2+m 
\int_{Q_T}P_m(\Delta P_m+\rho_m)^2dxdt+  
\int_{Q_T}P_m(\nabla^2P_m)^2dxdt\leq C(T),
\]
and thus \eqref{est:pg1} and the first estimate of \eqref{est:pg2} are obtained.

For the second bound of \eqref{est:pg2}, the above inequality and Lemma~\ref{l11} lead to
\begin{equation*}
\begin{aligned}
\int_{Q_T}|\partial_{i}P_m|^3dxdt
&=\int_{Q_T}\partial_{i}P_m\partial_{i}P_m  |\partial_{i}P_m|dxdt  \leq2\int_{Q_T}P_m|\partial_{ii}P_m||\partial_{i}P_m|dxdt\\
&\leq \int_{Q_T}P_m(\partial_{ii}P_m)^2dxdt+  
\int_{Q_T}P_m(\partial_{i}P_m)^2dxdt \leq C(T).
\end{aligned}
\end{equation*}
Therefore, we have obtained the second estimate of \eqref{est:pg2}.
\end{proof}

Next, our goal is to establish the uniform Aronson-B\'enilan (AB) estimate which uses the new variable
\begin{equation}\label{a21b}
       \omega_m:=\Delta P_m+\rho_m.
\end{equation}
 \begin{theorem}[Aronson-B\'enilan estimate]$\label{t11}$Assume that initial data
 $\rho_{m,0}$ satisfies \eqref{c1}--\eqref{c3}, let $\rho_m$ be a weak solution to the Cauchy problem Eq.~\eqref{d1}  in the sense of Def.~\ref{d1}, then
\begin{align}
&\sup_{0\leq t\leq T}\||\omega_m(t)|_-\|_{L^2(\mathbb{R}^n)}\leq
C(T),&&
\||\omega_m|_{-}\|_{L^3(Q_T)}^3\leq \frac{C(T)}{m},\label{est:ab1}\\
&\sup_{0\leq t\leq T}\||\omega_m(t)|_{-}\|_{L^1(\mathbb{R}^n)}\leq
C(T),&&
\sup_{0\leq t\leq T}\|\Delta P_m(t)\|_{L^1(\mathbb{R}^n)}\leq C(T)\label{est:ab2}.
\end{align}
\end{theorem}

\begin{proof}
We rewrite the equation of the density
       \begin{equation}\label{aa17}
       \begin{aligned}
       \partial_{t}\rho_m&=\Delta
       \rho_m^m+\nabla\cdot(\rho_m\nabla\mathcal{N}\ast\rho_m)\\
       &=\rho_m(\Delta P_m+\rho_m)+\nabla\rho_m\cdot(\nabla
       P_m+\nabla\mathcal{N}\ast\rho_m)\\
       &=\rho_m\omega_m+\nabla\rho_m\cdot(\nabla
       P_m+\nabla\mathcal{N}\ast\rho_m),
       \end{aligned}
       \end{equation}
       and the pressure equation is
       \begin{equation*}
       \partial_{t}P_m=(m-1)P_m\omega_m+\nabla P_m\cdot\nabla P_m+\nabla
       P_m\cdot\nabla\mathcal{N}\ast\rho_m.
       \end{equation*}
Then, we compute
       \begin{equation}\label{aa18}
       \begin{aligned}
       \partial_t\Delta P_m=&(m-1)\Delta(P_m\omega_m)+\nabla(\Delta
       P_m)\cdot(\nabla\mathcal{N}\ast\rho_m+\nabla P_m)\\
       &+\nabla P_m\cdot\nabla
       \omega_m+2\nabla^2P_m:(\nabla^2P_m+\nabla^2\mathcal{N}\ast\rho_m).
       \end{aligned}
       \end{equation}
 Combining \eqref{aa17} and \eqref{aa18}, the equation of $\omega_m$ is
       \begin{equation*}
       \begin{aligned}
       \partial_{t}\omega_m=&(m-1)\Delta(P_m\omega_m)+\nabla\omega_m\cdot\nabla\mathcal{N}\ast\rho_m+2\nabla^2P_m:(\nabla^2P_m+\nabla^2\mathcal{N}\ast\rho_m)\\
       &+\rho_m\omega_m+2\nabla P_m\cdot\nabla\omega_m.
       \end{aligned}
       \end{equation*}
Thus, we have
\begin{equation}\label{a20}
\begin{aligned}
\partial_t\omega_m\geq&(m-1)\Delta(P_m\omega_m)+\nabla\omega_m\cdot\nabla\mathcal{N}\ast\rho_m-\frac{1}{2}(\nabla^2\mathcal{N}\ast\rho_m)^2&+2\nabla P_m\cdot\nabla\omega_m+\rho_m\omega_m,
\end{aligned}
\end{equation}
where  we use that
\begin{equation*}
\begin{aligned}
2\nabla^2P_m:(\nabla^2P_m+\nabla^2\mathcal{N}\ast\rho_m)&\geq
2\nabla^2P_m:\nabla^2P_m-2\nabla^2P_m:\nabla^2P_m-\frac{1}{2}(\nabla^2\mathcal{N}\ast\rho_m)^2\\
&= -\frac{1}{2}(\nabla^2\mathcal{N}\ast\rho_m)^2.
\end{aligned}
\end{equation*}
Multiplying \eqref{a20} by $-2|\omega_m|_-$, due to Kato's inequality, we obtain
\begin{equation}\label{aa21}
\begin{aligned}
\partial_{t}|\omega_m|_-^2
\leq&
2(m-1)\Delta(P_m|\omega_m|_-)|\omega_m|_-+\nabla|\omega_m|_-^2\cdot\nabla\mathcal{N}\ast\rho_m+2\nabla|\omega_m|_-^2\cdot\nabla
P_m\\
&+(\nabla^2\mathcal{N}\ast\rho_m)^2|\omega_m|_-+2\rho_m|\omega_m|_-^2,
\end{aligned}
\end{equation}
where $\omega_m sgn(|\omega_m|_-)=-|\omega_m|_-$. Integrating \eqref{aa21} on
$\mathbb{R}^n$ and integrating by parts, we find
\begin{equation}\label{a99}
\begin{aligned}
\frac{d}{dt}\int_{\mathbb{R}^n}|\omega_m|_-^2dx\leq&
2(m-1)\int_{\mathbb{R}^n}\Delta(P_m|\omega_m|_-)|\omega_m|_-dx-2\int_{\mathbb{R}^n}|\omega_m|_-^2(\Delta
P_m+\rho_m)dx\\
&+3\int_{\mathbb{R}^n}|\omega_m|_-^2\rho_mdx+\int_{\mathbb{R}^n}(\nabla^2\mathcal{N}\ast\rho_m)^2|\omega_m|_-dx\\
=&
2(m-1)\int_{\mathbb{R}^n}\Delta(P_m|\omega_m|_-)|\omega_m|_-dx+2\int_{\mathbb{R}^n}|\omega_m|_-^3dx\\
&+3\int_{\mathbb{R}^n}|\omega_m|_-^2\rho_mdx+\int_{\mathbb{R}^n}(\nabla^2\mathcal{N}\ast\rho_m)^2|\omega_m|_-dx.
\end{aligned}
\end{equation}
Recalling the definition of $\omega_m$, we compute
\begin{equation}\label{OMEGA}
\begin{aligned}
2(m-1)\int_{\mathbb{R}^n} & \Delta(P_m|\omega_m|_-)|\omega_m|_-dx
=-2(m-1)\int_{\mathbb{R}^n}\nabla(P_m|\omega_m|_-)\nabla|\omega_m|_-dx\\
=&-(m-1)\int_{\mathbb{R}^n}\nabla
P_m\cdot\nabla|\omega_m|_-^2dx-2(m-1)\int_{\mathbb{R}^n}P_m|\nabla|\omega_m|_-|^2dx\\
=&(m-1)\int_{\mathbb{R}^n}\Delta
P_m|\omega_m|_-^2dx-2(m-1)\int_{\mathbb{R}^n}P_m|\nabla|\omega_m|_-|^2dx\\
=&(m-1)\int_{\mathbb{R}^n}(\omega_m-\rho_m)|\omega_m|_-^2dx-2(m-1)\int_{\mathbb{R}^n}P_m|\nabla|\omega_m|_-|^2dx\\
=&-(m-1)\int_{\mathbb{R}^n}
|\omega_m|_-^3dx-(m-1)\int_{\mathbb{R}^n}\rho_m|\omega_m|_-^2dx-2(m-1)\int_{\mathbb{R}^n}P_m|\nabla|\omega_m|_-|^2dx.
\end{aligned}
\end{equation}
 And, inserting this in \eqref{a99}, we get
\begin{equation}\label{a22}
\begin{aligned}
\frac{d}{dt}\int_{\mathbb{R}^n}& |\omega_m|_-^2dx+2(m-1)\int_{\mathbb{R}^n}P_m|\nabla|\omega_m|_-|^2dx\\
&+(m-4)\int_{\mathbb{R}^n}\rho_m|\omega_m|_-^2dx
+(m-3)\int_{\mathbb{R}^n}|\omega_m|_-^3dx\\
\leq&\int_{\mathbb{R}^n}(\nabla^2\mathcal{N}\ast\rho_m)^2|\omega_m|_-dx.
\end{aligned}
\end{equation}
 Thanks to Young's inequality, Lemma~\ref{l10}, and the singular integral theory
 for Newtonian potential (Lemma~\ref{l12}), we have
\[
\begin{aligned}
\int_{\mathbb{R}^n}(\nabla^2\mathcal{N}\ast\rho_m)^2|\omega_m|_-dx&\leq\sum_{ij}\frac{2n}{3^{3/2}}\int_{\mathbb{R}^n}|\partial_{ij}^2\mathcal{N}\ast\rho_m|^3dx+\sum_{ij}\frac{1}{n^2}\int_{\mathbb{R}^n}|\omega_m|_-^3dx\\
&\leq
\sum_{ij}\frac{2n}{3^{3/2}}C(n)\int_{\mathbb{R}^n}\rho_m^3dx+\int_{\mathbb{R}^n}|\omega_m|_-^3dx\\
&\leq \int_{\mathbb{R}^n}|\omega_m|_-^3dx+C(T).
\end{aligned}
\]
Inserting this into \eqref{a22}, we arrive at
\[\begin{aligned}
\frac{d}{dt}\int_{\mathbb{R}^n} |\omega_m|_-^2dx&+2(m-1)\int_{\mathbb{R}^n}P_m|\nabla|\omega_m|_-|^2dx
+(m-4)\int_{\mathbb{R}^n} [\rho_m|\omega_m|_-^2 +|\omega_m|_-^3]dx\\
\leq& C(T).
\end{aligned}
\]
After time integration, we obtain
\begin{equation*}
\begin{aligned}
\sup_{0\leq t\leq T}\int_{\mathbb{R}^n}  |\omega_m(t)|_-^2dx& +2m  
\int_{Q_T}P_m|\nabla|\omega_m|_-|^2dxdt+m  \int_{Q_T} [\rho_m|\omega_m|_-^2d
+|\omega_m|_-^3 ]dxdt\\
\leq& C(T).
\end{aligned}
\end{equation*}
This proves \eqref{est:ab1} of Theorem \ref{t11}.
\\

Next, we multiply \eqref{a20} by $-sgn(|\omega_m|_-)$, then we get
\[
\begin{aligned}
\partial_t|\omega_m|_-\leq&(m-1)\Delta(P_m|\omega_m|_-)+\nabla|\omega_m|_-\cdot\nabla\mathcal{N}\ast\rho_m+sgn(|\omega_m|_-)\frac{1}{2}(\nabla^2\mathcal{N}\ast\rho_m)^2\\
&+2\nabla P_m\cdot\nabla|\omega_m|_-+\rho_m|\omega_m|_-.
\end{aligned}
\]
After integration on $\mathbb{R}^n$, due to the singular integral theory
for Newtonian potential (Lemma~\ref{l12}), and Holder's inequality, we
attain
\[
\begin{aligned}
\frac{d}{dt}\int_{\mathbb{R}^n}|\omega_m|_-dx
\leq&2\int_{\mathbb{R}^n}|\omega_m|_-\rho_mdx+2\int_{\mathbb{R}^n}|\omega_m|_-^2dx+\int_{\mathbb{R}^n}\frac{1}{2}(\nabla^2\mathcal{N}\ast\rho_m)^2dx\\
\leq&
\int_{\mathbb{R}^n}\rho_m^2dx+3\int_{\mathbb{R}^n}|\omega_m|_-^2dx+C(n)\int_{\mathbb{R}^n}\rho_m^2dx 
\leq  C(T).
\end{aligned}
\]
Integrating in $ t$ as before,  we conclude the first estimate of \eqref{est:ab2} for Theorem~\ref{t11}, namely
\begin{equation}\label{a28}
\sup_{0\leq t\leq T}\int_{\mathbb{R}^n}|\omega_m(t)|_-dx\leq C(T).
\end{equation}

Finally, since $|\Delta P_m| \leq |\Delta P_m+\rho_m| + \rho_m =\Delta P_m+\rho_m + 2 |\omega_m|_- + \rho_m $ , we find
\begin{equation*}
\begin{aligned}
\int_{\mathbb{R}^n}|\Delta P_m|dx&\leq \int_{\mathbb{R}^n}\Delta
P_m+\rho_mdx+2\int_{\mathbb{R}^n}|\omega_m|_-dx+\int_{\mathbb{R}^n}\rho_mdx\\
&\leq 2C+2\int_{\mathbb{R}^n}|\omega_m|_-dx,
\end{aligned}
\end{equation*}
therefore, the second bound of \eqref{est:ab2} follows from \eqref{a28} that
\begin{equation*}
\sup_{0\leq t\leq T}\int_{\mathbb{R}^n}|\Delta P_m(t)|dx\leq C(T).
\end{equation*}
The proof of Theorem~\ref{t11} is completed.
\end{proof}

In the end, we are about to justify the $L^1$ time derivative estimate for the pressure. It is not easy to obtain such an estimate, but that is
useful to get locally strong compactness of the pressure gradient sequences $\{\nabla P_m\}_{m>1}$. We make full use of Kato's inequality and the specific form of the Newtonian potential to achieve our goal.

We first give two useful preliminary lemmas.

\begin{lemma}\label{l15}
Assume  that the initial data $\rho_{m,0}$ satisfies the assumptions
\eqref{c1}--\eqref{c2}, let $\rho_m$ be a weak solution to the Cauchy problem Eq.~\eqref{d1}  in the sense of Def.~\ref{d1}, then it follows
\begin{align*}
 \int_{Q_T}\big(\nabla\rho_m^{m}\cdot\nabla\rho_m^{m+1}+|\nabla\rho_m^{m+1}|^2+|\nabla\rho_m^{m+2}|^2)dxdt\leq
 C(T).
\end{align*}
\end{lemma}
\begin{proof}
These estimates can be written as $L^2(Q_T) $ bounds on $\rho_m^{\frac 32 } \nabla P_m$, $\rho_m^2 \nabla P_m$, $\rho_m^{\frac 52 } \nabla P_m$. They
are obvious consequences of~\eqref{est:rho4} and   \eqref{a16}.
\end{proof}
\begin{lemma}\label{l16}
Under the initial assumptions \eqref{c1}--\eqref{c3}, let the pair $(P_m,\rho_m)$ be a weak solution to the PKS model Eq.~\eqref{d1}in the sense of Def.~\ref{d1}, then it holds
\begin{equation*}
\|\partial_{t}\rho_m-\nabla\rho_m\cdot\nabla\mathcal{N}\ast\rho_m\|_{L^1(Q_T)}\leq
C(T).
\end{equation*}
\end{lemma}
\begin{proof}
Since
\begin{equation*}
\partial_t\rho_m-\nabla\rho_m\cdot\nabla\mathcal{N}\ast\rho_m\geq\rho_m\omega_m,
\end{equation*}
we have
\begin{equation*}
|\partial_t\rho_m-\nabla\rho_m\cdot\nabla\mathcal{N}\ast\rho_m|_- \leq \rho_m|\omega_m|_- \leq \frac 12 [\rho_m^2+ |\omega_m|_-^2]  .
\end{equation*}
Consequently, using mass conservation, Theorem~\ref{t11} and Lemma~\ref{l10},
\begin{equation*}
\begin{aligned}
 \int_{Q_T} |\partial_t &\rho_m-\nabla\rho_m\cdot\nabla\mathcal{N}\ast\rho_m|dxdt
\\
=&  \int_{Q_T}(\partial_t\rho_m-\nabla\rho_m\cdot\nabla\mathcal{N}\ast\rho_m)dxdt+2
\int_{Q_T}|\partial_t\rho_m-\nabla\rho_m\cdot\nabla\mathcal{N}\ast\rho_m|_-dxdt
\\
\leq & 2\int_{Q_T}\rho_m^{2}dxdt+\int_{Q_T}|\omega_m|_-^2dxdt
 \leq  C(T).
\end{aligned}
\end{equation*}
\end{proof}

In the following, we give the $L^1$ time derivative estimate of pressure.
\begin{theorem}[$L^1$ time derivative estimate of pressure]$\label{t12}$
Under the initial assumptions \eqref{c1}--\eqref{c4}, let $\rho_m$ be a weak solution to the Cauchy problem \eqref{d1} for the PKS in the sense of Def.~\ref{d1}, then
\begin{equation*}
\|\partial_t P_m\|_{L^1(Q_T)}\leq C(T).
\end{equation*}
\end{theorem}

\begin{proof}
We cannot work directly on $\partial_t P_m$ because of the power arising in a remainder term, and thus we use $\partial_t \rho_m^{m+1}$.
For this reason, we rewrite the cell density equation \eqref{d1} with two formulas
\begin{align}
&\partial_t\rho_m=\rho_m(\Delta P_m+\rho_m)+\nabla \rho_m\cdot(\nabla P_m+\nabla
\mathcal{N}\ast\rho_m),\label{a29}\\
&\partial_{t}\rho_m=\Delta\rho_m^m+\rho_m^2+\nabla\rho_m\cdot\nabla\mathcal{N}\ast\rho_m,\label{a30}
\end{align}
and we give two useful equations
\begin{align}
&\partial_t \rho_m^{m+1}=(m+1)\rho_m^{m+1}(\Delta
P_m+\rho_m)+\nabla\rho_m^{m+1}\cdot(\nabla
P_m+\nabla\mathcal{N}\ast\rho_m),\label{a31}\\
&\Delta\rho_m^{m+1}=\frac{m+1}{m}(\rho_m\Delta\rho_m^m+\nabla\rho_m\cdot\nabla\rho_m^m).\label{a32}
\end{align}
With the help of  Kato's inequality, we differentiate Eq.~\eqref{a31}
with respect to the time and multiply this by $-sgn(|\partial_t\rho_m|_-)$, then
it holds
\[
\begin{aligned}
\partial_t &|\partial_t\rho_m^{m+1}|_-\leq
(m+1)\Big[|\partial_t\rho_m^{m+1}|_- [\Delta P_m+\rho_m]+ \rho_m^{m+1}|\partial_t\rho_m|_- +\rho_m^{m+1}\Delta |\partial_tP_m|_- \Big]
\\
&+\nabla|\partial_t\rho_m^{m+1}|_-\cdot \nabla [P_m+\mathcal{N}\ast\rho_m]+\nabla\rho_m^{m+1}\cdot\nabla|\partial_tP_m|_-
-sgn_-(\partial_t\rho_m )\nabla\rho_m^{m+1}\cdot\nabla\mathcal{N}\ast\partial_t\rho_m
\end{aligned}
\]
and after integration by parts on $\mathbb{R}^n$ and insertion of Eq.~\eqref{a32},  we find
\[ \begin{aligned}
\frac{d}{dt}\int_{\mathbb{R}^n}|\partial_t\rho_m^{m+1}|_-dx
\leq  & \int_{\mathbb{R}^n}  |\partial_t\rho_m^{m+1}|_-  [m \Delta P_m +(m+1) \rho_m ]dx +m \int_{\mathbb{R}^n}  |\partial_tP_m|_- \Delta \rho_m^{m+1}dx
\\
&+  \int_{\mathbb{R}^n}  | \nabla\rho_m^{m+1}\cdot\nabla\mathcal{N}\ast\partial_t\rho_m|dx\\
=&(m+1)\int_{\mathbb{R}^n}|\partial_t\rho_m^{m+1}|_- \rho_m dx+(m+1)\int_{\mathbb{R}^n}|\partial_t\rho_m^m|_-\rho_m\Delta P_mdx\\
&+(m+1)\int_{\mathbb{R}^n}|\partial_t\rho_m^m|_-\Delta\rho_m^mdx+(m+1)\int_{\mathbb{R}^n}|\partial_tP_m|_-\nabla\rho_m\cdot\nabla\rho_m^m dx\\
&+  \int_{\mathbb{R}^n}  | \nabla\rho_m^{m+1}\cdot\nabla\mathcal{N}\ast\partial_t\rho_m|dx.\\
\end{aligned}
\]
We insert Eqs.~ \eqref{a29}--\eqref{a30} into this inequality, and use Eq.~\eqref{m10} for the last term.

\begin{equation*}
\begin{aligned}
\frac{d}{dt}\int_{\mathbb{R}^n} &|\partial_t\rho_m^{m+1}|_-dx \leq
(m+1)\int_{\mathbb{R}^n}|\partial_t\rho_m^{m+1}|_-\rho_mdx+(m+1) \int_{\mathbb{R}^n}|\partial_t P_m|_-\nabla\rho_m\cdot\nabla\rho_m^mdx
\\
&+(m+1)\int_{\mathbb{R}^n}|\partial_t\rho_m^m|_-(\partial_t\rho_m-\rho_m^2-\nabla\rho_m\cdot\nabla\mathcal{N}\ast\rho_m-\nabla\rho_m\cdot\nabla
P_m)dx\\
&+ (m+1)\int_{\mathbb{R}^n}|\partial_t\rho_m^m|_-(\partial_t\rho_m-\rho_m^2-\nabla\rho_m\cdot\nabla\mathcal{N}\ast\rho_m)dx\\
&+\int_{\mathbb{R}^n}\nabla\rho_m^m\cdot\nabla\rho_m^{m+1}dx+\int_{\mathbb{R}^n}|\nabla\rho_m^{m+1}\cdot\nabla\mathcal{N}\ast(\nabla\cdot(\rho_m\nabla\mathcal{N}\ast\rho_m))|dx.
\end{aligned}
\end{equation*}
The two terms with $(m+1)|\partial_t P_m|_-\nabla\rho_m\cdot\nabla\rho_m^m$ and $-(m+1)|\partial_t\rho_m^m|_- \nabla\rho_m\cdot\nabla
P_m$  cancel due to $|\partial_t P_m|_-\nabla\rho_m\cdot\nabla\rho_m^m=|\partial_t\rho_m^m|_- \nabla\rho_m\cdot\nabla
P_m$,
then it holds
\begin{equation}\label{a38}
\begin{aligned}
\frac{d}{dt}\int_{\mathbb{R}^n} &|\partial_t\rho_m^{m+1}|_-dx+(m-1)\int_{\mathbb{R}^n}|\partial_t\rho_m^{m+1}|_-\rho_mdx
+2(m+1)\int_{\mathbb{R}^n}|\partial_t\rho_m^m|_-|\partial_t\rho_m|_-dx\\
\leq&\underbrace{2(m+1)\int_{\mathbb{R}^n}|\partial_t\rho_m^m|_-\nabla\rho_m\cdot\nabla\mathcal{N}\ast\rho_mdx}_{\mathcal{A}}+\int_{\mathbb{R}^n}\nabla\rho_m^m\cdot\nabla\rho_m^{m+1}dx\\
&+\underbrace{\int_{\mathbb{R}^n}|\nabla\rho_m^{m+1}\cdot\nabla\mathcal{N}\ast(\nabla\cdot(\rho_m\nabla\mathcal{N}\ast\rho_m))|dx}_{\mathcal{B}}.
\end{aligned}
\end{equation}
For $\mathcal{A}$ and $\mathcal{B}$, we have
\[
\begin{aligned}
\mathcal{A}=&2(m+1)m\int_{\mathbb{R}^n}\rho_m^{m-1}|\partial_t\rho_m|_-\nabla\rho_m\cdot\nabla\mathcal{N}\ast\rho_mdx\\
\leq&(m+1)m\int_{\mathbb{R}^n}\rho_m^{m-1}|\partial_t\rho_m|_-^2dx+(m+1)m\int_{\mathbb{R}^n}\rho_m^{m-1}|\nabla\rho_m\cdot\nabla\mathcal{N}\ast\rho_m|^2dx\\
\leq&(m+1)\int_{\mathbb{R}^n}|\partial_t\rho_m^{m}|_-|\partial_t\rho_m|_-dx+(m+1)\|\nabla\mathcal{N}\ast\rho_m\|_{L^{\infty}(Q_T)}^2\int_{\mathbb{R}^n}\nabla\rho_m^m\cdot\nabla\rho_mdx,
\end{aligned} \]
\[
\begin{aligned}
\mathcal{B}\leq&
\frac{1}{2}\int_{\mathbb{R}^n}|\nabla\rho_m^{m+1}|^2dx+\frac{1}{2}\int_{\mathbb{R}^n}|\nabla\mathcal{N}\ast(\nabla\cdot(\rho_m\nabla\mathcal{N}\ast\rho_m))|^2dx\\
\leq&
\frac{1}{2}\int_{\mathbb{R}^n}|\nabla\rho_m^{m+1}|^2dx+\frac{1}{2}C(n)\|\nabla\mathcal{N}\ast\rho_m\|_{L^{\infty}(Q_T)}^2\int_{\mathbb{R}^n}\rho_m^2dx.
\end{aligned}
\]
From Lemma~\ref{l10} in which we let $p=2$, we control in $L^1(Q_T)$ the term $\nabla \rho_m^m \cdot \nabla\rho_m$. The terms in $\mathcal{B}$ are also controled thanks to the bounds (in particular in Lemma~\ref{l15}), as well as the second term in final expression of $\mathcal A$. All together the known bounds reduce~\eqref{a38} to
\begin{equation*}
\frac{d}{dt}\int_{\mathbb{R}^n}|\partial_t\rho_m^{m+1}|_-dx+\frac{m^2-1}{m+2} \int_{\mathbb{R}^n}|\partial_t\rho_m^{m+2}|_- dx+(m+1)\int_{\mathbb{R}^n}|\partial_t\rho_m^m|_-|\partial_t\rho_m|_-dx
\leq (m+1) C(T).
\end{equation*}
Therefore, it holds
\begin{equation}\label{a53}
\int_{Q_T}|\partial_t\rho_m^{m+2}|_-dxdt+\int_{Q_T} |\partial_t\rho_m^m|_-|\partial_t\rho_m|_-dxdt\leq C(T).
\end{equation}
Taking account of  Lemma~\ref{l10} and Theorem~\ref{t10}, we use Sobolev's
inequality and obtain
\begin{equation*}
\begin{aligned}
\sup_{0\leq t\leq T}\int_{\mathbb{R}^n}\rho_m^{m+2}dx
\leq &C\sup_{0\leq t\leq T}\int_{\mathbb{R}^n}\rho_m^{3}P_mdx
\leq C\sup_{0\leq t\leq T}\|\nabla P_m(t)\|_{L^2(\mathbb{R}^n)}+C(T)\\
\leq& C(T).
\end{aligned}
\end{equation*}
Thus, combining the above inequality and \eqref{a53}, we get
\begin{equation}\label{a80}
\begin{aligned}
\int_{Q_T}|\partial_t\rho_m^{m+2}|dxdt&\leq  
 \int_{Q_T}\partial_t\rho_m^{m+2}dxdt+2  
 \int_{Q_T}|\partial_t\rho_m^{m+2}|_-dxdt\\
&\leq
\int_{\mathbb{R}^n}\rho_m^{m+2}(T)dx-\int_{\mathbb{R}^n}\rho_{m,0}^{m+2}dx+2 \int
\hskip-4pt \int_{Q_T}|\partial_t\rho_m^{m+2}|_-dxdt\\
&\leq C(T).
\end{aligned}
\end{equation}
Furthermore, combining \eqref{a80}, Lemma~\ref{l18}, and Lemma~\ref{l15}, we have
\begin{equation}\label{a42}
\begin{aligned} 
\int_{Q_T} |\partial_t & \rho_m^{m+2}-\nabla\rho_m^{m+2}\cdot\nabla\mathcal{N}\ast\rho_m|dxdt\\
\leq&\int_{Q_T} \Big[|\partial_t\rho_m^{m+2}|+\frac{1}{2}|\nabla\rho_m^{m+2}|^2+\frac{1}{2}|\nabla\mathcal{N}\ast\rho_m|^2\Big]dxdt
\leq C(T).
\end{aligned}
\end{equation}
By Lemma~\ref{l16} and \eqref{a42}, we obtain
\[
\begin{aligned}
\int_{Q_T}|\partial_t\rho_m^{m-1}&-\nabla\rho_m^{m-1}\cdot\nabla\mathcal{N}\ast\rho_m|dxdt \\
=& 
\int_{Q_T}|\partial_t\rho_m^{m-1}-\nabla\rho_m^{m-1}\cdot\nabla\mathcal{N}\ast\rho_m| \big[ \chi_{\{\rho_m\leq\frac{1}{2}\}} +
\chi_{\{\rho_m>\frac{1}{2}\}} \big]
dxdt\\
\leq&(m-1)\frac{1}{2^{m-2}}  
\int_{Q_T}|\partial_t\rho_m-\nabla\rho_m\cdot\nabla\mathcal{N}\ast\rho_m|dxdt\\
&+\frac{8(m-1)}{m+2}  
\int_{Q_T}|\partial_t\rho_m^{m+2}-\nabla\rho_m^{m+2}\cdot\nabla\mathcal{N}\ast\rho_m|dxdt 
\leq C(T).
\end{aligned}
\]
Combining this with Lemma~\ref{l10} and Lemma~\ref{l18}, we end up with
\begin{equation*}
\begin{aligned}
\int_{Q_T} & |\partial_t P_m|dxdt\leq   \int_{Q_T}|\partial_t P_m-\nabla
P_m\cdot\nabla\mathcal{N}\ast\rho_m|dxdt+\int_{Q_T}|\nabla
P_m\cdot\nabla\mathcal{N}\ast\rho_m|dxdt\\
\leq&\frac{m}{m-1}\int_{Q_T}|\partial_t \rho_m^{m-1}-\nabla
\rho_m^{m-1}\cdot\nabla\mathcal{N}\ast\rho_m|dxdt+\frac{1}{2}  
\int_{Q_T} \big[|\nabla P_m|^2dxdt + |\nabla\mathcal{N}\ast\rho_m|^2 \big] dxdt\\
\leq& C(T),
\end{aligned}
\end{equation*}
where the first inequality is the application of triangle inequality and the
second inequality is due to the Cauchy-Schwarz inequality. The proof is completed.
\end{proof}

\begin{remark}
 It should be emphasized that the first step  in the proof of Theorem~\ref{t12} is to compute the time derivative of $\rho_m^{m+1}$.  We can also begin with $\rho_m^{m}$ (not $P_m$), which requires to use, from the density dissipation formula \eqref{est:rho2} from Lemma~\ref{l10},
 \[
 \int_{Q_T} |\nabla \rho_m^m \cdot \nabla \rho_m| \rho_m^{-1} dx dt \leq C(T).
 \]
This bound needs the entropy $\rho_m \log \rho_m$, we don't do here for the initial assumption.
\end{remark}

\section{Complementarity relation and semi-harmonicity}

Thanks to the a priori regularity estimates provided by Lemmas~\ref{l10}--\ref{l6} and
Theorems~\ref{t10}--\ref{t12}, we can obtain the strong compactness on the
pressure gradient, which allows us to obtain the complementarity relation. Moreover, the semi-harmonicity follows from the AB estimate (Theorem~\ref{t11}).

\begin{theorem}[Complementarity relation and semi-harmonicity]$\label{t13}$
 Under the initial assumptions \eqref{c1}--\eqref{c4}, then, the complementarity
 relation  and the semi-harmonicity property, see Eq.~\eqref{z15}, hold. 
\end{theorem}

\begin{proof}
From Lemma~\ref{l10} and Theorems~\ref{t12}, we have
\begin{align*}
\|\nabla P_m\|_{L^2(Q_T)}\leq C(T),\quad \|\partial_t P_m\|_{L^1(Q_T)}\leq C(T).
\end{align*}
Then, after the extraction of subsequences, we obtain
\begin{equation*}
P_m\rightarrow P_\infty,\text{ strongly in } L^1_{loc}(Q_T),\text{ as }
m\to\infty,
\end{equation*}
with the help of Sobolev's compactness embedding. From Theorem~\ref{t10}, we
obtain the weak compactness of the pressure gradient, up to a subsequence, we
have
\begin{equation*}
\nabla P_m\rightharpoonup \nabla P_\infty,\text{ weakly in } L^3(Q_T),\text{ as }
m\to\infty.
\end{equation*}

We define a smooth cutoff function $0\leq \varphi \leq 1$, $\varphi(x)=1$, for $ |x|\leq 1$, $\varphi(x)=0$ for $|x|\geq 2$. Then,
for any $R>0$, we let $\varphi_R(x)=\varphi(\frac{x}{R})$ and $P_{m,R}=\varphi_{R}(x)P_m$. By direct computations, we obtain
\begin{equation}\label{a57}
\|\partial_{t}P_{m,R}\|_{L^1(Q_T)}\leq C(T,R),\quad
\|\nabla P_{m,R}\|_{L^3(Q_T)}\leq C(T,R),\quad
\|\Delta P_{m,R}\|_{L^1(Q_T)}\leq C(T,R).
\end{equation}
For the sake of the above three estimates \eqref{a57}, inspired by
\cite[Theorem~6.1]{r35}, we can establish
\begin{equation*}
\nabla P_{m,R}\to \nabla P_{\infty,R},\text{ strongly in } L^1(Q_T),\text{ as }
m\to\infty.
\end{equation*}
In other words, we can extract a subsequence such that
\begin{equation*}
\nabla P_m\rightarrow \nabla P_\infty,\text{ strongly  in } L^1_{loc}(Q_T),\text{
as } m\to\infty.
\end{equation*}
Then, using the uniform $L^3$ bound for the pressure gradient in
Theorem~\ref{t10}, we have
\begin{equation*}
\nabla P_m\rightarrow \nabla P_\infty,\text{ strongly in } L^q_{loc}(Q_T), \text{
for } 1\leq q<3.
\end{equation*}
Hence, in particular, the case $q=2$ is selected to achieve our goal.

Let $\zeta\in C_0^{\infty}(Q_T)$ be a test function,
 we multiply the pressure equation \eqref{d3} by $\zeta$ and integrate on $Q_T$,
 then it follows
\begin{equation*}
\begin{aligned}
-\frac{1}{m-1}\int_{Q_T} & P_m\partial_t\zeta+|\nabla
P_m|^2\zeta+\nabla P_m\cdot\nabla\mathcal{N}\ast\rho_m\zeta dxdt\\
=&\int_{Q_T}(-|\nabla P_m|^2\zeta-P_m\nabla
P_m\cdot\nabla\zeta+P_m\rho_m\zeta)dxdt.
\end{aligned}
\end{equation*}
Hence, passing to limit as $m\to\infty$, we obtain the complementarity relation
\begin{equation*}
\int_{Q_T}(-|\nabla P_\infty|^2\zeta-P_\infty\nabla
 P_\infty\cdot\nabla\zeta+P_\infty\rho_\infty\zeta)dxdt=0,
\end{equation*}
where $\rho_m P_m\rightharpoonup \rho_\infty P_{\infty},\text{ weakly in }
L^2(Q_T),\text{ as } m\to\infty$ results from \eqref{m38}.
This is equivalent to
\begin{equation*}
\int_{Q_T}P_\infty(\Delta P_\infty+\rho_\infty)\zeta dxdt=0,
\end{equation*}
which means that the complementarity relation of Eq.~\eqref{z15} holds.

From Theorem~\ref{t11}, we have $\int_{Q_T}|\Delta
P_m+\rho_m|_-^3dxdt\leq \frac{C(T)}{m}$. Let $\phi\in C_{0}^{\infty}(Q_T)$ be a
nonnegative test function in $Q_T$, then we attain
\begin{equation}
\begin{aligned}
\int_{0}^{T}\int_{\mathbb{R}^n}(\Delta P_\infty+\rho_\infty)\phi
dxdt=&\lim\limits_{m\to\infty}\int_{0}^{T}\int_{\mathbb{R}^n}(\Delta
P_m+\rho_m)\phi dxdt\\
\geq&-\lim\limits_{m\to\infty}\int_{0}^{T}\int_{\mathbb{R}^n}|\Delta
P_m+\rho_m|_-\phi dxdt\\
\geq&-\lim\limits_{m\to\infty}(\int_{0}^{T}\int_{\mathbb{R}^n}|\Delta
P_m+\rho_m|_-^3dxdt)^{\frac{1}{3}}(\int_{0}^{T}\int_{\mathbb{R}^n}\phi^{\frac{3}{2}}
dxdt)^{\frac{2}{3}}\\
=&0.
\end{aligned}
\end{equation}
Hence, we establish the second result (semi-harmonicity property) of Eq.~\eqref{z15}.

The proof is completed.
\end{proof}
\begin{remark}
This result tells us that the limit solution satisfies
\begin{equation*}
\begin{cases}
-\Delta P_\infty=1, &\text{ in }\ \Omega(t):=\{x:P_\infty(x,t)>0\},\\
P_\infty=0& \text{ on }\ \partial\Omega(t),
\end{cases}
\end{equation*}
 when enough regularity is available. This is related to the geometric form of the Hele-Shaw free boundary problem
 while Eq.~\eqref{z6} is the weak form which determines the motion of the free boundary.
\end{remark}
\begin{theorem}
Under the initial assumptions \eqref{c1}--\eqref{c4}, then we have
\begin{equation*}
\rho_\infty \nabla P_\infty=\nabla P_\infty,\quad\text{a.e. in }Q_T.
\end{equation*}
\end{theorem}
\begin{proof}
On the one hand, we have already proved in Theorem \ref{tAEbis} that $\rho_m \nabla P_m \to \nabla P_\infty$ weakly. On the other hand
$\rho_m  \to \rho_\infty$ weakly in $L^2(Q_T)$ and $ \nabla P_m \to \nabla P_\infty$ strongly in $L^2_{loc}(Q_T)$, by weak-strong convergence we obtain that $\rho_m \nabla P_m \to \rho_\infty \nabla P_\infty$  weakly and the result is proved.
\end{proof}

\section{Uniqueness, compact support and energy functional}

In order to prove the uniqueness of the solution to the Hele-Shaw limit system
\eqref{z6}--\eqref{z8}, we use the lifting method in $\dot{H}^{-1}$ as in~\cite{4,r46,
rDDP} rather than the duality method~\cite{5,r50} or the entropy method~\cite{IgN}. The main new  difficulty comes
from the nonlocal interaction. The  uniform upper bound for the limit
density, and  the property that the limit pressure is somehow monotone to the
limit density, allow us to use the energy method to prove the uniqueness as
in either \cite[Theorem~2.4]{4} or \cite[Theorem~3]{r46} for an aggregation
equation with degenerate diffusion.
\begin{proposition}[Uniqueness]$\label{p1}$
	Let $(\rho_1,P_1)$ and $(\rho_2,P_2)$ be two solutions to the Cauchy problem
Eq.~\eqref{z6}--\eqref{z8} with initial data satisfying
$\rho_{1}(x,0)=\rho_{2}(x,0)\in \dot{H}^{-1}(\mathbb{R}^n)$, then it follows
	\begin{equation*}
	\rho_1=\rho_2,\quad P_1=P_2, \quad\text{a.e. in }Q.
	\end{equation*}
\end{proposition}
\begin{proof}
First of all, we state that the pressure is somehow monotone to the density.	
Since $\rho_1P_1=P_1$ and $\rho_2P_2=P_2$ hold, we have
	\begin{equation}\label{z9}
	\begin{aligned}
	(\rho_1-\rho_2)(P_1-P_2)=&\rho_1 P_1+\rho_2 P_2-\rho_1 P_2-\rho_2 P_1
	\\[2pt]
	=&(1-\rho_2)P_1+(1-\rho_1)P_2\\
	\geq&0.
	\end{aligned}
	\end{equation}
We estimate the difference of weak solutions in $\dot{H}^{-1}(\mathbb{R}^n)$
motivated
by the fact that the pressure is somehow monotone to the density as \eqref{z9}.
Let $\psi=\mathcal{N}\ast(\rho_1-\rho_2)$, by the integrability and bound of
$\rho_1$ and $\rho_2$, we have $\psi\in L^{\infty}(Q_T)\cap
C(0,T,\dot{H}^1(\mathbb{R}^n))$ and $\nabla\psi\in L^{\infty}(0,T;L^{p}(\mathbb{R}^n))\cap
L^\infty(Q_T)$ for $2\leq p< \infty$, and $\partial_t \psi$ solves
\begin{equation*}
\Delta \partial_t\psi=\partial_{t}\rho_1-\partial_{t}\rho_2.
\end{equation*}
Since
$\|\rho_1(t)-\rho_2(t)\|_{\dot{H}^{-1}(\mathbb{R}^n)}=\|\nabla\psi(t)\|_{L^2(\mathbb{R}^n)}$,
we are going to show $\|\nabla\psi(t)\|_{L^2(\mathbb{R}^n)}=0$ for all $t\geq 0$.

 Let $\varphi\in C_0^{\infty}(\mathbb{R}^n)$ be smooth text function satisfying
 $\varphi=1$ in $B_1 $ and $0\leq \varphi \leq 1$ in $\mathbb{R}^n$.
Set $\varphi_R(x)=\varphi(\frac{x}{R})$ for $x\in\mathbb{R}^n$ and $R>1$, thanks to the regularity of $\psi$, possibly up to a set of measure zero, it holds
\begin{equation*}
\begin{aligned}
-<\partial_{t}\rho_1-\partial_{t}\rho_2,\psi>
=&\lim\limits_{R\to\infty}-<\partial_{t}\rho_1-\partial_{t}\rho_2,\mathcal{N}\ast[(\rho_1-\rho_2)\varphi_R]>\\
=&\lim\limits_{R\to\infty}<\nabla\partial_{t}\psi,\nabla\mathcal{N}\ast[(\rho_1-\rho_2)\varphi_R]>\\
=&<\nabla\psi,\nabla\partial_{t}\psi>
=\frac{1}{2}\frac{d}{dt}\int_{\mathbb{R}^n}|\nabla\psi|^2dx.
\end{aligned}
\end{equation*}
Therefore, using the definition of weak solution in Theorem~\ref{tAEbis}, we have
\begin{equation*}
\begin{aligned}
\frac{1}{2}\frac{d}{dt}\int_{\mathbb{R}^n}|\nabla\psi|^2dx=&\lim\limits_{R\to\infty}-\int_{\mathbb{R}^n}(P_1-P_2)\Delta(\psi\varphi_R)
dx+\int_{\mathbb{R}^n}(\rho_1\nabla\mathcal{N}\ast\rho_1-\rho_2\nabla\mathcal{N}\ast\rho_2)\cdot\nabla(\varphi_R\psi)
dx\\
=&-\int_{\mathbb{R}^n}(P_1-P_2)\Delta\psi
dx+\int_{\mathbb{R}^n}(\rho_1\nabla\mathcal{N}\ast\rho_1-\rho_2\nabla\mathcal{N}\ast\rho_2)\cdot\nabla\psi
dx\\
=&-\int_{\mathbb{R}^n}(P_1-P_2)(\rho_1-\rho_2)dx+\int_{\mathbb{R}^n}(\rho_1-\rho_2)\nabla\mathcal{N}\ast\rho_1\cdot\nabla\psi
dx\\
&+\int_{\mathbb{R}^n}\rho_2\nabla\mathcal{N}\ast(\rho_1-\rho_2)\cdot\nabla\psi
dx\\
:=&I_1+I_2+I_3.
\end{aligned}
\end{equation*}
From \eqref{z9}, we obtain
\begin{equation*}
I_1=-\int_{\mathbb{R}^n}(P_1-P_2)(\rho_1-\rho_2)dx\leq 0.
\end{equation*}
By integrating by parts, we have
\begin{equation}\label{z10}
\begin{aligned}
I_2&=\lim\limits_{R\to\infty}\int_{\mathbb{R}^n}\Delta \psi\nabla \mathcal{N}\ast\rho_1\cdot\nabla(\psi\varphi_R)
dx\\
&=-\lim\limits_{R\to\infty}\sum\limits_{ij}\int_{\mathbb{R}^n}\partial_i\psi\partial_j(\psi\varphi_R)\partial_{ij}\mathcal{N}\ast\rho_1dx-\lim\limits_{R\to\infty}\sum\limits_{ij}\int_{\mathbb{R}^n}\partial_i\psi\partial_{ij}(\psi\varphi_R)\partial_j\mathcal{N}\ast\rho_1dx\\
&=-\sum\limits_{ij}\int_{\mathbb{R}^n}\partial_i\psi\partial_j\psi\partial_{ij}\mathcal{N}\ast\rho_1dx-\sum\limits_{ij}\int_{\mathbb{R}^n}\partial_i\psi\partial_{ij}\psi\partial_j\mathcal{N}\ast\rho_1dx.
\end{aligned}
\end{equation}
Similarly, integrating by parts again, we get
 \begin{equation*}
 \begin{aligned}
 -\sum\limits_{ij}\int_{\mathbb{R}^n}\partial_i\psi\partial_{ij}\psi\partial_j\mathcal{N}\ast\rho_1dx
 =&\lim\limits_{R\to\infty} -\sum\limits_{ij}\int_{\mathbb{R}^n}\partial_i\psi\partial_{ij}\psi\partial_j\mathcal{N}\ast\rho_1\varphi_Rdx\\
 =& \lim\limits_{R\to\infty}\sum\limits_{ij}\frac{1}{2}\int_{\mathbb{R}^n}(\partial_i\psi)^2\partial_j(\partial_j\mathcal{N}\ast\rho_1\varphi_R)dx\\
 =&\frac{1}{2}\int_{\mathbb{R}^n}|\nabla\psi|^2\rho_1
 dx,
 \end{aligned}
 \end{equation*}
 which together with \eqref{z10} implies
 \begin{equation*}
 I_2\leq\int_{\mathbb{R}^n}|\nabla\psi|^2|\nabla^2\mathcal{N}\ast\rho_1|dx+\frac{1}{2}\|\nabla\psi\|_{L^2(\mathbb{R}^n)}^2.
 \end{equation*}
 By Holder's inequality and $\nabla\psi\in{L^{\infty}(Q_T)}$, for $p\geq 2$, we
 have
 \begin{equation}\label{z11}
 \begin{aligned}
 \int_{\mathbb{R}^n}|\nabla\psi|^2|\nabla^2\mathcal{N}\ast\rho_1|dx&\leq
 \|\nabla^2\mathcal{N}\ast\rho_1\|_{L^p(\mathbb{R}^n)}(\int_{\mathbb{R}^n}|\nabla\psi|^{\frac{2p}{p-1}})^{\frac{p-1}{p}}\\
 &\lesssim
 p\|\rho_1\|_{L^p(\mathbb{R}^n)}\|\nabla\psi\|_{L^\infty(\mathbb{R}^n)}^{\frac{2}{p}}(\int_{\mathbb{R}^n}|\nabla\psi|^2dx)^{\frac{p-1}{p}}\\
 &\lesssim p(\int_{\mathbb{R}^n}|\nabla\psi|^2dx)^{\frac{p-1}{p}},
 \end{aligned}
 \end{equation}
 where the implicit constant depends only on the uniformly controlled $L^p$ norms
 of $\rho_1$ and $\rho_2$ and the second step holds because of the singular
 integral theory (Lemma~\ref{l12}).\\
 \indent As for $I_3$, we may directly justify
 \begin{equation}\label{z12}
 I_3=\int_{\mathbb{R}^n}\rho_2\nabla\mathcal{N}\ast(\rho_1-\rho_2)\cdot\nabla\psi
 dx=\int_{\mathbb{R}^n}\rho_2|\nabla\psi|^2dx\lesssim
 \|\nabla\psi\|_{L^2(\mathbb{R}^n)}^2.
 \end{equation}
 Let $\gamma(t)=\int_{\mathbb{R}^n}|\nabla\psi(t)|^2dx$, both \eqref{z11} and
 \eqref{z12} imply the differential inequality
 \begin{equation*}
 	\frac{d}{dt}\gamma(t)\leq \hat{C}p
 \max\{\gamma(t)^{1-\frac{1}{p}},\gamma(t)\},
 \end{equation*}
 where $\hat{C}$ depends only on the uniformly controlled $L^p$ norm of $\rho_1$,
 $\rho_2$.
 All  the solutions of this  differential inequality are bounded from above by
 the maximal solution. Since $\gamma(0)=0$ and $\gamma(t)$ is continuous, there
 exists $t^*>0$ such that $0\leq \gamma(t)<1,t\in[0,t^*]$, therefore
     \begin{equation*}
    \frac{d}{dt}\gamma(t)\leq \hat{C}p \gamma(t)^{1-\frac{1}{p}}, \qquad \gamma(0)=0,
    \end{equation*}
and $\gamma(t)$ is a subfunction of the solution to the ordinary differential
equation
 \begin{equation}\label{z13}
 \frac{d}{dt}\bar{\gamma}(t)= \hat{C}p \bar{\gamma}(t)^{1-\frac{1}{p}}, \qquad \bar{\gamma}(0)=0,
 \end{equation}
 and $\bar{\gamma}(t)=(\hat{C}t)^{p}$ is the unique solution to \eqref{z13}.
  Consequently, we obtain
  \begin{equation}\label{z14}
  \gamma(t)\leq\bar{\gamma}(t)\leq 2^{-p}<1,
  \end{equation}
For $0<t<\frac{1}{2\hat{C}}$. Therefore, we can extend $t^*$ to be  long enough such that $t^*$ is more
      than $\frac{1}{4\hat{C}}$.
 \\

In Eq.~\eqref{z14}, we may take $p\to \infty$ to deduce that $\gamma(t)=0$ for $t\in[0,\frac{1}{4\hat{C}}]$ and the proof of uniqueness is complete.
\end{proof}

In fact, we are able to prove the time continuity and initial trace for the Hele-Shaw limit system \eqref{z6}--\eqref{z8}. So far the initial data is obtained in the weak sense of Def.~\ref{def:WS}. This means that the Hele-Shaw equation holds with the initial data
$\rho_{\infty,0}= w-lim \rho_{m,0}$. Notice that we know that $0\leq \rho_{\infty,0}\leq 1$ because the argument of Lemma~\ref{l6} still holds true.

We now prove a additional result, namely the initial density is also obtained by  time continuity.
\begin{proposition}[Almost everywhere time continuity]
 Assume that initial data $\rho_{m,0}$ and $\rho_{\infty,0}$ satisfy the assumption~\eqref{c1}.  Then it holds
\begin{equation*}
\lim_{t\to 0} \rho_{\infty}(t) := \rho_{\infty,0}\text{ a.e. in }\mathbb{R}^n.
\end{equation*}
\end{proposition}
\begin{proof}
Let $\varphi\in C_{0}^{\infty}(\mathbb{R}^n)$. By a standard variant of the test function in Def.~\ref{def:WS}, we have for a.e. $t>0$,
\begin{equation*}
\int_{\mathbb{R}^n}(\rho_{m}(t)-\rho_{m,0})\varphi(x)dx=-\frac{m-1}{m}\int_{0}^{t}\int_{\mathbb{R}^n}\rho_{m}\nabla
P_m\cdot\nabla\varphi
dx-\int_{0}^{t}\int_{\mathbb{R}^n}\rho_{m}\nabla\mathcal{N}\ast\rho_m\cdot\nabla
\varphi dx.
\end{equation*}
Passing to limit, then it holds
\begin{equation*}
\int_{\mathbb{R}^n}(\rho_{\infty}(t)-\rho_{\infty,0})\varphi(x)dx=-\int_{0}^{t}\int_{\mathbb{R}^n}\nabla
P_{\infty}\cdot\nabla\varphi
dx-\int_{0}^{t}\int_{\mathbb{R}^n}\rho_{\infty}\nabla\mathcal{N}\ast\rho_{\infty}\cdot\nabla
\varphi dx.
\end{equation*}

Multiplying (\ref{z6}) by $\varphi$ and integrating on $[0,t]$, we get
\begin{equation*}
\int_{\mathbb{R}^n}(\rho_{\infty}(t)-\rho_{\infty}^0)\varphi(x)dx=-\int_{0}^{t}\int_{\mathbb{R}^n}\nabla
P_{\infty}\cdot\nabla\varphi
dx-\int_{0}^{t}\int_{\mathbb{R}^n}\rho_{\infty}\nabla\mathcal{N}\ast\rho_{\infty}\cdot\nabla
\varphi dx,
\end{equation*}
therefore, we obtain that $\lim_{t\to 0} \rho_{\infty}(t) :=  \rho_{\infty}^0$ exists in weak-$L^2(\mathbb{R}^n)$ and
\begin{equation*}
\int_{\mathbb{R}^n}\rho_{\infty,0}\varphi dx=\int_{\mathbb{R}^n}\rho_{\infty}^0\varphi dx,
\end{equation*}
which supports our statement.
\end{proof}
Furthermore, we are about to show the compact support of the solution for the Hele-Shaw limit system \eqref{z6}--\eqref{z8}. To study the support of the limit density or the limit pressure, the main
 difficulty comes from the nonlocal interaction which prevent bounds on $\rho_m$. However, we may follow~\cite[Lemma~3.8]{CKY_2018} and obtain uniformly control of  the pressure, then we deduce that the speed of
propagation for the limit density is finite.

Firstly, we give this approximate equation
\begin{equation}\label{w1}
\partial_{t}\varrho_m=\Delta\varrho_m^m+\nabla\cdot(\varrho_m\nabla\Phi_{1/m}),
\end{equation}
where $\Phi_{1/m}(x,t)=\zeta_{1/m}\ast(\mathcal{N}\ast\rho_{\infty})$ and
$\rho_{\infty}$ is the unique limit density in Theorem~\ref{tAEbis}. Define the
corresponding pressure $\mathrm{P}_m=\frac{m}{m-1}\varrho_m^{m-1}$ that satisfies
the following equation
\begin{equation}\label{w2}
\partial_{t}\mathrm{P}_m=(m-1)\mathrm{P}_m\Delta\mathrm{P}_m+|\nabla
\mathrm{P}_m|^2+\nabla \mathrm{P}_m\cdot\nabla \Phi_{1/m}+(m-1)\mathrm{P}_m\Delta
\Phi_{1/m} .
\end{equation}
Similar to Theorems~\ref{tAE}--\ref{tAEbis}, we can get the following theorem with the same
initial data.
\begin{theorem}\label{t6}
Let $\rho_{m,0}$ and $P_{m,0}$ be the initial data of the density $\varrho_{m}$
and the pressure $\mathrm{P}_{m}$ respectively satisfying \eqref{c1} and $\rho_{\infty}$ be the unique limit density in Theorem~\ref{tAEbis}. Then, after the extraction of
subsequences, $\nabla\Phi_{1/m}$ converges for all $T>0$ strongly
	in $L^2(Q_{T})$ as $m\to\infty$ to limit
$\nabla\mathcal{N}\ast\rho_{\infty}$, $\varrho_m$ and $\varrho_{m}
\mathrm{P}_{m}$ converges weakly for all $T>0$ in $L^2(Q_T)$ as $m\to\infty$ to
limits $\varrho_{\infty}\in L^\infty(0,T;L^1(\mathbb{R}^n))\cap L^\infty(Q_T)$ and $\mathrm{P}_{\infty}\in L^2(0,T;H^1(\mathbb{R}^n))$ respectively.
Therefore, the following Hele-Shaw limit system for $(\mathrm{P}_\infty,\varrho_\infty)$ holds as
\begin{align}
	&\partial_{t}\varrho_{\infty}=\Delta
\mathrm{P}_{\infty}+\nabla\cdot(\varrho_{\infty}\nabla\mathcal{N}\ast\rho_{\infty}),
&&\text{in } \mathcal{D}'(Q_{T}),\label{w3}\\
	&(1-\varrho_{\infty})\mathrm{P}_{\infty}=0,&&\text{a.e.
in } Q_{T},\label{w4}\\
	&0\leq\varrho_{\infty}\leq1,&&\text{a.e. in }Q_{T}.\label{w6}
	\end{align}
\end{theorem}
\begin{proof}We omit the detailed proof of this theorem, because its proof is
similar to, but easier than, that of Theorems~\ref{tAE}--\ref{tAEbis}.
\end{proof}

It is easy for us to prove that the Hele-Shaw limit system \eqref{z6}--\ref{z8} and
\eqref{w3}--\eqref{w6} have same solutions if we have the same initial assumptions. In
other words, if we get a uniform support of $\varrho_{m}$ and $\mathrm{P}_{m}$,
naturally, we can obtain the supports of $\rho_{\infty}$ and $P_{\infty}$.
\begin{lemma}
For $(\rho_{\infty}$,$P_{\infty})$ in Theorem \ref{tAEbis} and for
$(\varrho_{\infty}$,$\mathrm{P}_{\infty})$ in Theorem \ref{t6} with the initial
assumption $\rho_{\infty}(x,0)=\varrho_{\infty}(x,0)\in \dot{H}^{-1}(\mathbb{R}^n)$, then it follows
\begin{equation*}
\rho_{\infty}=\varrho_{\infty},\quad P_\infty=\mathrm{P}_\infty,\quad \text{a.e. }(x,t)\in Q.
\end{equation*}
\end{lemma}
\begin{proof}
	The proof of this lemma is similar to but easier than the proof of
Proposition~\ref{p1}, hence, we omit the detailed processes.
	\end{proof}
	
\begin{lemma}
	Let $\varrho_m$ be a nonnegative weak solution to Eq.~\eqref{w1} for any
continuous, compactly supported initial data $\varrho_{m,0}$. Then the pressure
varible $\mathrm{P}_m$ is a viscosity solution to \eqref{w2}.
\end{lemma}

\begin{proof}
	The result follows from \cite[Corollary 3.11]{LeiKim}
\end{proof}
We now turn to the  $L^\infty$ estimate and the support of the solutions to
\eqref{w3}, which are uniform on $m$. The first ensures that if the initial data
is bounded
uniformly on $m$, it remains uniformly bounded within any finite time. The second
ensures that
if the support of the initial data is bounded uniformly on $m$, it likewise
remains uniformly bounded within any finite time.
\begin{lemma}
	[$L^\infty$ estimate and support of $\mathrm{P}_m$] $\label{l17}$Let
$\mathrm{P}_{m}$ be a viscosity
    solution to Eq.~\eqref{w3} with continuous, compactly supported initial data
    $\mathrm{P}_{m}(\cdot,0)$. Suppose that there exists $\mathcal{R}_{0}\geq 1$
    sufficiently large so that $ {\rm supp}(\mathrm{P}_{m}(\cdot,0))\subseteq
    B_{\mathcal{R}_0/2}$ and $\mathrm{P}_{m}(\cdot,0)\leq \frac{\mathcal{R}_{0}^2}{4n}$.
Then there exist $\mathcal{R}(t):=\big(\mathcal{R}_{0}+n\|\nabla\mathcal{N}\ast\rho_\infty\|_{L^\infty(Q)}\big)e^{\frac{t}{n}}-n\|\nabla\mathcal{N}\ast\rho_\infty\|_{L^\infty(Q)}$ such that
		\begin{itemize}
			\item $\{\varrho_m(\cdot,t)>0\}\subseteq B_{\mathcal{R}(t)}$, for all
$t\in[0,\infty)$,
			\item $\mathrm{P}_{m}(x,t)\leq \frac{\mathcal{R}^2(t)}{2n}$, for all
$(x,t)\in Q$.
		\end{itemize}
\end{lemma}
\begin{proof}
	The result follows from \cite[Lemma 3.8]{CKY_2018}.
\end{proof}
\medskip

 When the initial density $\rho_{\infty,0}$ satisfies $0\leq \rho_{\infty,0}\leq 1$ and  is compactly supported,
we show that the solution $(\rho_\infty, P_{\infty})$ to \eqref{z6}--\eqref{z8} are bounded and  compactly supported for all
times.
\begin{theorem}[Compact support]\label{t5}
	Suppose that there exists $\mathcal{R}_0$ sufficiently large such that both
$ {\rm supp}(\rho_{\infty}(\cdot,0))\subseteq B_{\mathcal{R}_0/2}$ and $1+\frac{n}{n-2}\leq \frac{\mathcal{R}_0^2}{4n}$, and also both
$\rho_{\infty}(\cdot,0)\in L^1_+(\mathbb{R}^n)$ and $0\leq
\rho_{\infty}(\cdot,0)\leq 1$ hold. Then, for all $t>0$, the solution $(\rho_\infty ,P_\infty)$  to the Cauchy problem for the
Hele-Shaw limit system \eqref{z6}--\eqref{z8} satisfies
\begin{itemize}
			\item $\{\rho_{\infty}(\cdot,t)>0\}\subseteq B_{\mathcal{R}(t)} $,
			\item $P_{\infty}(\cdot ,t)\leq \frac{\mathcal{R}^2(t)}{2n}$,
\end{itemize}
where $\mathcal{R}(t)$ is given in Lemma~\ref{l17}.
\end{theorem}
\begin{proof}
If we set $\rho_{m,0}=\rho_{\infty,0}$ for $m\geq 2n-1$, so it holds $P_{m,0}=\frac{m}{m-1}\rho_{m,0}^{m-1}\leq 1+\frac{n}{n-2}\leq \frac{\mathcal{R}_0^2}{4n}$.	According to Lemma~\ref{l17}, it  is easy to complete the proof of Theorem~\ref{t5}.
\end{proof}

In the end, we will establish the limit energy functional for the Cauchy problem of the Hele-Shaw problem \eqref{z6}--\eqref{z15}.
For the PKS model Eq.~\eqref{d1} with the diffusion exponent $1<m<\infty$, the
energy functional is given by
\begin{equation*}
\frac{dF_m(\rho_m)}{dt}+\int_{\mathbb{R}^n}\rho_{m}|\nabla
P_m+\nabla\mathcal{N}\ast\rho_{m}|^2dx=0.
\end{equation*}
The above equality shows that the free energy decreases as the time increases.
Formally, as $m\to\infty$, the limit free energy satisfies
\begin{equation*}
F_{\infty}(\rho_{\infty})=\frac{1}{2}\int_{\mathbb{R}^n}\rho_{\infty}\mathcal{N}\ast\rho_{\infty}dx,\
\quad 0\leq \rho_\infty\leq 1,
\end{equation*}
in which the diffusive effect is replaced by the height constraint of the limit
density,
and the limit energy functional is expressed as
\begin{equation}\label{w7}
\frac{dF_{\infty}(\rho_{\infty}(t))}{dt}+\int_{\mathbb{R}^n}\rho_{\infty}(t)|\nabla
P_\infty(t)+\nabla\mathcal{N}\ast\rho_{\infty}(t)|^2dx=0,\quad 0\leq \rho_\infty\leq 1.
\end{equation}

In the following theorem, we show that the limit energy functional \eqref{w7}
holds.
\begin{theorem}[Limit energy functional]$\label{t4}$
Under the initial assumptions \eqref{c1}--\eqref{c4} and the uniform support assumption for the initial density~\eqref{c5},  let $\rho_{\infty}$ and
$P_{\infty}$ be the limit density and the limit pressure respectively as in
Theorems~\ref{tAE}--\ref{t15}. Then, for a.e. $t\in[0,\infty)$, the
limit energy functional \eqref{w7} holds.
\end{theorem}
\begin{proof}
Under the initial assumptions \eqref{c1}--\eqref{c4} and the additional initial uniform support assumption~\eqref{c5}, due to Theorem~\ref{t5}, it holds
\begin{equation*}
 {\rm supp}(P_\infty(t))\subset  {\rm supp}(\rho_\infty(t))\subset B_{\mathcal{R}(t)} ,
\end{equation*}
for $\mathcal{R}(t):=2\big(2\mathcal{R}_{0}+\frac{\|\nabla\mathcal{N}\ast\rho_\infty\|_{L^\infty(Q)}}{n}\big)e^{\frac{t}{n}}-\frac{2\|\nabla\mathcal{N}\ast\rho_\infty\|_{L^\infty(Q)}}{n}$ with some $\mathcal{R}_0\geq R_0>0$.

From Theorems~\ref{tAE}--\ref{t15}, for any $T>0$, we have
\begin{equation}\label{fe1}
\rho_\infty\in L^\infty\big(0,T;L^1_+(\mathbb{R}^n)\big)\cap L^\infty(Q_T),\quad P_\infty\in L^2\big(0,T;H^1(\mathbb{R}^n)\big),\quad \nabla P_\infty\in L^3(Q_T),
\end{equation}
\begin{equation}\label{fe2}
\mathcal{N}\ast\rho_{\infty}\in\mathcal{C}(0,T;\dot{W}^{1,r}(\mathbb{R}^n))\cap L^\infty(0,T;\cap \dot{W}^{2,s}(\mathbb{R}^n)) \text{ for }2\leq r\leq\infty \text{ and } 1<s<\infty.
\end{equation}Furthermore, we obtain
\begin{equation}\label{fe3}
\Delta P_\infty\in L^2(0,T;\dot{H}^{-1}(\mathbb{R}^n)\cap \dot{H}^{-2}(\mathbb{R}^n)),\quad \partial_t\rho_\infty \in L^2(0,T;\dot{H}^{-1}(\mathbb{R}^n)).
\end{equation}

	Thanks to the complementarity relation \eqref{z15} and Eq.~\eqref{z8}, integrating \eqref{z15} on $\mathbb{R}^n$ and integrating by
parts, then it follows from the regularities \eqref{fe1}--\eqref{fe3} that
\begin{equation}\label{w9b}
\int_{\mathbb{R}^n}P_{\infty}(\Delta
P_{\infty}+\rho_{\infty})dx=\int_{\mathbb{R}^n}P_{\infty}\rho_{\infty}-|\nabla
P_{\infty}|^2dx=0.
\end{equation}
Therefore, it follows
\begin{equation}\label{w8}
\begin{aligned}
\int_{\mathbb{R}^n}\rho_{\infty}|\nabla P_\infty+\nabla\mathcal{N}\ast\rho_{\infty}|^2dx
=&\int_{\mathbb{R}^n}\rho_{\infty}(|\nabla P_{\infty}|^2+2\nabla
P_{\infty}\cdot\nabla\mathcal{N}\ast\rho_{\infty}+|\nabla\mathcal{N}\ast\rho_{\infty}|^2)dx\\
=&\int_{\mathbb{R}^n}(|\nabla
P_{\infty}|^2-2P_{\infty}\rho_{\infty}+\rho_{\infty}|\nabla\mathcal{N}\ast\rho_{\infty}|^2)dx\\
=&\int_{\mathbb{R}^n}(-\rho_{\infty}P_{\infty}+\rho_{\infty}|\nabla\mathcal{N}\ast\rho_{\infty}|^2)dx,
\end{aligned}
\end{equation}
where \eqref{w9b} is used in the last equality. Multiplying \eqref{z6} by
$\mathcal{N}\ast\rho_{\infty}$ and integrating on $\mathbb{R}^n$, according to
the symmetry of convolution operator, we have
\begin{equation}\label{fi}
\frac{1}{2}\frac{d}{dt}\int_{\mathbb{R}^n}\rho_{\infty}\mathcal{N}\ast\rho_{\infty}dx-\int_{\mathbb{R}^n}(\Delta
P_\infty+\nabla\cdot(\rho_{\infty}\nabla\mathcal{N}\ast\rho_{\infty})\mathcal{N}\ast\rho_{\infty}dx=0,
\end{equation}
integrating by parts and using (\ref{w8}), then it holds due to \eqref{fe1}--\eqref{fe3} that
\begin{equation}\label{f2}
\begin{aligned}
-\int_{\mathbb{R}^n}(\Delta
P_\infty+\nabla\cdot(\rho_{\infty}\nabla\mathcal{N}\ast\rho_{\infty})\mathcal{N}\ast\rho_{\infty}dx
=&\int_{\mathbb{R}^n}(-P_\infty\rho_{\infty}+\rho_{\infty}|\nabla\mathcal{N}\ast\rho_{\infty}|^2)dx\\
=&\int_{\mathbb{R}^n}\rho_{\infty}|\nabla
P_\infty+\nabla\mathcal{N}\ast\rho_{\infty}|^2dx.
\end{aligned}
\end{equation}
Combining \eqref{fi} and \eqref{f2}, for almost everywhere time $t$, we obtain
the limit energy functional
\begin{equation*}
\frac{1}{2}\frac{d}{dt}\int_{\mathbb{R}^n}\rho_{\infty}\mathcal{N}\ast\rho_{\infty}dx+\int_{\mathbb{R}^n}\rho_{\infty}|\nabla
P_\infty +\nabla\mathcal{N}\ast\rho_{\infty}|^2dx=0,\quad 0\leq \rho_\infty\leq
1.
\end{equation*}
 \end{proof}
\begin{remark}
The result of Theorem~\ref{t4} implies that the limit free energy
$F_{\infty}(\rho_{\infty}(t))$ is non-increasing as time $t$ increases.
\end{remark}

\section{Incompressible limit of stationary state}
This section is devoted to showing that the incompressible (Hele-Shaw) limit for the stationary state of Patlak-Keller-Segel (SPKS) model~Eq.~\eqref{SPKS} is the stationary state Eq.~\eqref{HSSPKS} of the Hele-Shaw problem Eq.~\eqref{hs}--\eqref{d5}. By direct computations, the equation of the corresponding pressure $P_{m,s}=\frac{m}{m-1}\rho_{m,s}^{m-1}$ is expressed by
\begin{equation}\label{e14}
(m-1)P_{m,s}(\Delta P_{m,s}+\rho_{m,s})+\nabla P_{m,s}\cdot(\nabla P_{m,s}+\nabla\mathcal{N}\ast\rho_{m,s})=0\quad\text{for } x\in\mathbb{R}^n.
\end{equation}

The following preliminary lemma is combined and extracted from \cite{r18,rr10,rr12}.
\begin{lemma}[Preliminary lemma]\label{ls1}
Assume that  $\rho_{m,s}\in L^1(\mathbb{R}^n)\cap L^{\infty}(\mathbb{R}^n)$. Then the solution to the SPKS model Eq.~\eqref{SPKS} exists. Moreover,  the solution to the SPKS model Eq.~\eqref{SPKS} is radially decreasing symmetric, unique up to a translation, and  compactly supported.
\end{lemma}

In the rest of this section, we carry on the incompressible limit of the stationary state of PKS (SPKS) model Eq.~\eqref{SPKS} in the framework of radial symmetry , and $C$ is a positive constant independent of the exponent $m$.

For any given mass $M>0$, we show that the solution to the SPKS model Eq.~\eqref{SPKS} is uniformly bounded on $m$.
\begin{lemma}[Uniform bound of pressure]$\label{ls2}$
Let $\rho_{m,s}$ be a weak solution to the SPKS model Eq.~\eqref{SPKS} in the sense of Def.~\ref{d1} with $\|\rho_{m,s}\|_{L^1(\mathbb{R}^n)}=M>0$, $m\geq 3$, and $\int_{\mathbb{R}^n}x\rho_{m,s}(x)dx=0$. Set $\alpha_m:=\rho_{m,s}(0)=\|\rho_{m,s}\|_{L^\infty(\mathbb{R}^n)}$, then it holds
\begin{equation*}
\begin{aligned}
&\alpha_{m}\leq \frac{1+\sqrt{1+\frac{8M}{n(n-2)\omega_n}}}{2},\quad\alpha_{m}^{m-1}\leq\frac{1+\sqrt{1+\frac{8M}{n(n-2)\omega_n}}}{2}+\frac{2M}{n(n-2)\omega_n}.
\end{aligned}
\end{equation*}
\end{lemma}
\begin{proof}
From the SPKS model Eq.~\eqref{SPKS}, we have
\begin{equation*}
\rho_{m,s}(\nabla P_{m,s}+\nabla\mathcal{N}\ast\rho_{m,s})=0.
\end{equation*}
Due to the radially decreasing symmetric property of $\rho_{m,s}$, there exists a constant $C>0$ such that
\begin{equation}\label{ff1}
P_{m,s}=(-\mathcal{N}\ast\rho_{m,s}-C)_+\quad\text{for } x\in\mathbb{R}^n
\end{equation}
and \begin{equation}\label{ff2}
C\leq \|-\mathcal{N}\ast\rho_{m,s}\|_{L^{\infty}(\mathbb{R}^n)}.
\end{equation}
Since $\alpha_{m}=\|\rho_{m,s}\|_{L^{\infty}(\mathbb{R}^n)}=\rho_{m,s}(0)$, we have
\begin{equation}\label{ff3}
\begin{aligned}
-\mathcal{N}\ast\rho_{m,s}&=\frac{1}{n(n-2)\omega_n}\int_{\mathbb{R}^n}\frac{\rho_{m,s}(x-y)}{|y|^{n-2}}dy\\
&\leq \frac{1}{n(n-2)\omega_n}\int_{|y|>1}\frac{\rho_{m,s}(x-y)}{|y|^{n-2}}dy+\alpha_m\frac{1}{n(n-2)\omega_n}\int_{|y|\leq 1}\frac{1}{|y|^{n-2}}dy\\
&\leq \frac{1}{n(n-2)\omega_n}M+\frac{\alpha_m}{2(n-2)}\\
&\leq \frac{\alpha_m}{2}+\frac{1}{n(n-2)\omega_n}M.
\end{aligned}
\end{equation}
Combining \eqref{ff1}--\eqref{ff3}, we obtain
\begin{equation}\label{ff4}
\alpha_{m}^{m-1}\leq \frac{m}{m-1}\alpha_{m}^{m-1}\leq \alpha_m+\frac{2M}{n(n-2)\omega_n},
\end{equation}
where $1\leq\frac{m}{m-1}\leq2$ for $m\geq3$ is used.

The positive solution to the quadratic equations with one unknownn $y^2_2=y_2+\frac{2M}{n(n-2)\omega_n}$ is
\begin{equation*}
y_2=\frac{1+\sqrt{1+\frac{8M}{n(n-2)\omega_n}}}{2}>1.
\end{equation*}
Thus it follows from the property of algebraic equation that
\begin{equation}\label{ff5}
0\leq \alpha_m\leq y_2\quad\text{for }m\geq 3.
\end{equation}
We combine \eqref{ff4}--\eqref{ff5} and obtain
\begin{equation*}
\alpha_{m}^{m-1}\leq \alpha_m+\frac{2M}{n(n-2)\omega_n}\leq \frac{1+\sqrt{1+\frac{8M}{n(n-2)\omega_n}}}{2}+\frac{2M}{n(n-2)\omega_n}.
\end{equation*}
\end{proof}
\begin{remark}
It should be pointed out that the conclusion of Lemma~\ref{ls2} still holds with the assumption $\rho_{m,s}\in L^\infty(\mathbb{R}^n)\cap L^1(\mathbb{R}^n)$,  the radially symmetric property of solution is not necessary.
\end{remark}

Next, we show uniformly bounded support of density, which can prevent the mass from escapeing to infinity as $m\to\infty$.  Let $\|\rho_{m,s}\|_{L^1(\mathbb{R}^n)}=M$ and $ {\rm supp}(\rho_{m,s})= B_{R_m(M)} $. We show that there exists a constant $R_{*}(M)$ (only depending on $M$) such that
\begin{equation*}
R_m(M)\leq R_*(M)\text{ for all } m\geq3.
\end{equation*}

Define $\Psi_{m}(r)=\Psi_m(|x|)=\rho_{m,s}^{m-1}(x)$ with $r=|x|$, then the SPKS model Eq.~\eqref{SPKS} as introduced in \cite[Lemma 2.1]{rr10} can be transformed as a
dynamical system
\begin{equation}\label{ff15}
\begin{cases}
\Psi_m''(r)+\frac{n-1}{r}\Psi_m'(r)=-\frac{m-1}{m}\Psi_m^{1/(m-1)}(r)\quad\text{for all }0<r<R_m(M),&\\
\Psi_m(0)=\alpha_m^{m-1}, \quad\Psi_m'(0)=0,&
\end{cases}
\end{equation}
and the following conditions hold:
\begin{equation*}
\begin{cases}
\Psi_m(r)>0,\quad\Psi_m'(r)<0,&\text{on }\big(0,R_m(M)\big),\\
\frac{\Psi_m'(r)}{r}\to -\frac{1}{n}\frac{m-1}{m}\alpha_m^{1/(m-1)},&\text{as }r\to0^+,\\
 \Psi_m(r)\to0,&\text{as }r\to R_m(M)_-.
\end{cases}
\end{equation*}

To show the uniform bound of $R_m(M)$, we are going to give a plane autonomous system.
Let
\begin{equation*}
u_{m}(r)=-\frac{m}{m-1}\frac{r\Psi_{m}^{1/(m-1)}(r)}{\Psi_m'(r)} \text{ and }v_m(r)=-\frac{r\Psi_m'(r)}{\Psi_m(r)},
\end{equation*}
by direct computations, there is a plane autonomous system of $(u_m,v_m)$ as
\begin{equation}\label{ff6}
\begin{cases}
r\frac{du_m}{dr}=u_m(n-u_m-\frac{v_m}{m-1}),&\\
r\frac{dv_m}{dr}=v_m\big(-(n-2)+u_m+v_m\big),&
\end{cases}
\end{equation}
for $r>0$ with the initial data
\begin{equation}\label{ff7}
\lim\limits_{r\to0^+}u_m(r)=n,\quad \lim\limits_{r\to0^+}v_m(r)=0,\text{ and }\lim\limits_{r\to0^+}\frac{v_m(r)}{r^2}=\frac{m-1}{mn}\alpha_{m}^{2-m}.
\end{equation}

 The strategy is to find $R_*(M)<\infty$ satisfying
 \begin{equation*}
 \lim\limits_{r\to R_*(M)_-}v_m(r)=+\infty,
 \end{equation*}
 which means that $R_{m}(M)<R_*(M)$ for any $m\geq3$ holds.
\begin{lemma}\label{ls3}
Suppose that $(u_m,v_m)$ is a solution to the initial value problem Eqs.~\eqref{ff6}--\eqref{ff7}, then it holds for $r>0$ that
\begin{equation*}
\begin{aligned}
 n< u_m+\frac{v_m}{m-1},\quad 0<u_m<n,\quad\text{and}\quad 0<v_m.
\end{aligned}
\end{equation*}
\end{lemma}
\begin{proof}
 Lemma~\ref{ls3} is a direct result of \cite[Lemma 2.2]{rr10}.
\end{proof}
\begin{lemma}[Uniform support of density]$\label{ls4}$
Suppose that $(u_m,v_m)$ is a solution to the initial value problem Eqs.~\eqref{ff6}--\eqref{ff7} with a given mass $M>0$. Then, there exists $R_*(M):=\log\Big(1+\exp\Big[2n(n-1)\Big(\frac{1+\sqrt{1+\frac{8M}{n(n-2)\omega_n}}}{2}+\frac{2M}{n(n-2)\omega_n}\Big)\Big]^{1/2}\Big)$ such that
\begin{equation*}
\lim\limits_{r\to R_*(M)_-}v_m(r)=\infty\quad\text{for all } m\geq3.
\end{equation*}
Furthermore, it holds
\begin{equation*}
 {\rm supp}(\rho_{m,s})\subset B_{R_*(M)} \quad\text{for all } m\geq3.
\end{equation*}
\end{lemma}
\begin{proof}
From Lemma~\ref{ls3}, we have
\begin{equation*}
u_m+\frac{v_m}{m-1}>n\quad\text{and}\quad u_m,\ v_m>0,
\end{equation*}
then it holds
\begin{equation*}
u_m+v_m>u_m+\frac{v_m}{m-1}>n.
\end{equation*}
Combing the above inequality and $\eqref{ff6}_2$-\eqref{ff7}, we obtain
\begin{equation}\label{ff8}
\begin{cases}
r\frac{dv_m}{dr}>2v_m&\text{for } r>0,\\
\lim\limits_{r\to0^+}v_m(r)=0,& \lim\limits_{r\to0^+}\frac{v_m(r)}{r^2}=\frac{m-1}{nm}\alpha_{m}^{2-m}.
\end{cases}
\end{equation}
We give an ordinary differential equation:
\begin{equation}\label{ff9}
\begin{cases}
r\frac{dv}{dr}=2v&\text{for }r>0,\\
\lim\limits_{r\to0^+}v(r)=0,& \lim\limits_{r\to0^+}\frac{v(r)}{r^2}=\frac{m-1}{mn}\alpha_{m}^{2-m}.
\end{cases}
\end{equation}
It is easy to obtain  the unique solution to Eq.~\eqref{ff9} as
\begin{equation}\label{ff10}
v(r)=\frac{m-1}{mn}\alpha_{m}^{2-m}r^2.
\end{equation}
Since the solution  $v(r)$ \eqref{ff10} to Eq.~\eqref{ff9} is a sub-solution to Eq.~\eqref{ff8}, then we have
\begin{equation}\label{ff11}
v_m(r)\geq v(r)=\frac{m-1}{mn}\alpha_{m}^{2-m}r^2\quad\text{for all }r>0.
\end{equation}

 With the help of Lemma~\ref{ls2} for a given mass $M>0$, a uniform lower bound independent of $m$ is obtained as
\begin{equation}\label{fff12}
\begin{aligned}
v_m(r)&\geq \frac{m-1}{mn}\alpha_{m}^{2-m}r^2=\frac{m-1}{mn}(\alpha_{m}^{m-1})^{(2-m)/(m-1)}r^2\\
&\geq \frac{m-1}{mn}\big(\frac{1+\sqrt{1+\frac{8M}{n(n-2)\omega_n}}}{2}+\frac{2M}{n(n-2)\omega_n}\big)^{(2-m)/(m-1)}r^2\\
&\geq \frac{1}{2n}\big(\frac{1+\sqrt{1+\frac{8M}{n(n-2)\omega_n}}}{2}+\frac{2M}{n(n-2)\omega_n}\big)^{-1}r^2\quad\text{for all }m\geq3,
\end{aligned}
\end{equation}
where $\frac{m-1}{m}>\frac{1}{2}$ and $\frac{2-m}{m-1}>-1$ are used.
There exists a positive constant $R'(M)$ (only depending on $M$ and $n$)  such that
\begin{equation}\label{fff13}
\frac{1}{2n}\Big(\frac{1+\sqrt{1+\frac{8M}{n(n-2)\omega_n}}}{2}+\frac{2M}{n(n-2)\omega_n}\Big)^{-1}R'^2(M)=n-1,
\end{equation}
and $R'(M)$ can be precisely written like
\begin{equation*}
R'(M)=\Big[2n(n-1)\Big(\frac{1+\sqrt{1+\frac{8M}{n(n-2)\omega_n}}}{2}+\frac{2M}{n(n-2)\omega_n}\Big)\Big]^{1/2}>1.
\end{equation*}
Then it follows from \eqref{ff11}--\eqref{fff13} that
\begin{equation}\label{ff12}
v_{m}(r)\geq n-1\quad\text{for all } m\geq3\text{ and }r\geq R'(M).
\end{equation}
 Combining $\eqref{ff6}_2$ and \eqref{ff12}, we have
\begin{equation}\label{fff13b}
r\frac{dv_m}{dr}\geq v_m\Big(v_m-(n-2)\Big)\geq \Big(v_m-(n-2)\Big)^2\quad\text{for all } m\geq3\text{ and }r\geq R'(M).
\end{equation}

Set $r:=\log s$ with $R'(M):=\log S'(M)$, we define $v_m(r)=v_m(\log s):=v_m(s)$, then it follows from \eqref{fff13b} that
\begin{equation*}
\frac{dv_m}{ds}\geq \Big(v_m-(n-2)\Big)^2\quad\text{for all } m\geq3\text{ and }s\geq S'(M).
\end{equation*}
On the other hand, we consider the following ordinary differential equation:
\begin{equation*}\label{fff14}
\begin{cases}
\frac{d\omega}{ds}=\Big(\omega-(n-2)\Big)^2\quad\text{for all }s\geq S'(M),&\\
\omega\big(S'(M)\big)=n-1,&
\end{cases}
\end{equation*}
Hence, it holds
\begin{equation*}
\omega(s)=\frac{\omega\big(S'(M)\big)-(n-2)}{1-(s-S'(M))\Big(\omega\big(S'(M)\big)-(n-2)\Big)}+n-2.
\end{equation*}
We find
\begin{equation*}
\omega(s)\to\infty\quad\text{as }s\to\frac{1+S'(M)\Big(\omega(S'(M))-(n-2) \Big)}{\omega(S'(M))-(n-2)}=1+S'(M):=S_*(M).
\end{equation*}
Since $\omega(s)$ is a sub-function of $v_m(s)$ for all $s\geq S'(M)$ and $m\geq 3$, we have
\begin{equation*}
v_m(s)\to\infty,\quad\text{as }s\to S_*(M).
\end{equation*}
Set \begin{equation*}
\begin{aligned}
R_*(M)=&\log S_*(M)=\log(1+e^{R'(M)})\\
=&\log\Big(1+\exp\Big[2n(n-1)\Big(\frac{1+\sqrt{1+\frac{8M}{n(n-2)\omega_n}}}{2}+\frac{2M}{n(n-2)\omega_n}\Big)\Big]^{1/2}\Big),
\end{aligned}
\end{equation*}
then it follows
\begin{equation}\label{ff16}
v_m(r)\to\infty\quad\text{as }r\to R_*(M)_-.
\end{equation}

We need to show $\Psi_m(R_*(M))=0$. If not, due to the monotonicity of $\Psi_m$,  we assume that
\begin{equation*}
\Psi_m(r)\geq \Psi_m(R_*(M))>0\quad\text{for all }r\in[0,R_*(M)].
\end{equation*}
We multiply Eq.~$\eqref{ff15}_2$ by $r^{n-1}$ and obtain
\begin{equation*}
[r^{n-1}\Psi_m'(r)]'=-\frac{m-1}{m}r^{n-1}\Psi_m^{1/(m-1)}(r).
\end{equation*}
Integrating the above equation on $[0,r]$ for any $r\in(0,R_*(M)]$, it holds
\begin{equation}\label{ff17}
r^{n-1}\Psi_m'(r)=-\frac{m-1}{m}\int_{0}^{r}s^{n-1}\Psi_m^{1/(m-1)}(s)ds.
\end{equation}
It follows from the definition of $\Psi_m$ that
\begin{equation*}
\begin{aligned}
\frac{m-1}{m}\int_{0}^{r}s^{n-1}\Psi_m^{1/(m-1)}(s)ds=&\frac{m-1}{m}\int_{0}^{r}s^{n-1}\rho_{m,s}(s)ds\\
=&\frac{m-1}{nm\omega_n}\int_{B_r}\rho_{m,s}dx\\
\leq& \frac{m-1}{nm\omega_n}M,
\end{aligned}
\end{equation*}
which together with \eqref{ff17} implies that there exists a small $\delta>0$ (may depending on $m$) such that
\begin{equation*}
|\Psi_m'(r)|< \infty\quad \text{for all }r\in[\delta,R_*(M)].
\end{equation*}
Therefore, we have
\begin{equation*}
|v_m(r)|=|-\frac{r\Psi_m'(r)}{\Psi_m(r)}|<\infty\quad \text{for all }r\in[\delta,R_*(M)],
\end{equation*}
which is contradicted with \eqref{ff16}. In this way, one can show that
\begin{equation*}
\Psi_m\big(R_*(M)\big)=0,
\end{equation*}
and it holds
\begin{equation*}
 {\rm supp}(\rho_{m,s})\subset \text{B}_{R_*(M)} \quad\text{for all }m\geq3.
\end{equation*}
\end{proof}
\begin{remark}
 It should be emphasized that $R_*(M)=\log\Big(1+\exp\Big[2n(n-1)\Big(\frac{1+\sqrt{1+\frac{8M}{n(n-2)\omega_n}}}{2}+\frac{2M}{n(n-2)\omega_n}\Big)\Big]^{1/2}\Big)$ strictly increases as the mass $M>0$ strictly increases, which is consistent with the geometric induction that higher mass means larger support.
\end{remark}
We try to obtain the regularity estimate on convolution term. Indeed, to obtain the weak convergence of the nonlinear term $\{\rho_{m,s}\nabla\mathcal{N}\ast\rho_{m,s}\}_{m>1}$, one way is to prove the strong convergence of $\{\nabla\mathcal{N}\ast\rho_{m,s}\}_{m>1}$ by means of the weak-strong convergence.
\begin{lemma}[Regularity estimate on convolution term]$\label{ls5}$
 Let $\rho_{m,s}$ be a weak solution to the SPKS model Eq.~\eqref{SPKS} in the sense of Def.~\ref{d1} with $\|\rho_{m,s}\|_{L^1(\mathbb{R}^n)}=M$ and $m\geq3$, then
\begin{align*}
&\|\nabla \mathcal{N}\ast\rho_{m,s}\|_{L^2\cap L^\infty(\mathbb{R}^n)}\leq C,\quad \|\nabla^2 \mathcal{N}\ast\rho_{m,s}\|_{L^p(\mathbb{R}^n)}\leq C(M,p),
\end{align*}
where $C(M,p)\sim\frac{1}{p-1}$ for $0<p-1\ll 1$ and $C(M,p)\sim p$ for $p\gg1$.

Furthermore, thanks to Sobolev's embedding theorem, there exists $\nabla\mathcal{N}\ast\rho_{\infty,s}\in L^2(\mathbb{R}^n)\cap L^\infty(\mathbb{R}^n)$ such that
\begin{equation*}
\nabla \mathcal{N}\ast\rho_{m,s}\to \nabla \mathcal{N}\ast\rho_{\infty,s},\quad\text{strongly in }L^p_{loc}(\mathbb{R}^n)\text{ for }1\leq p<\infty,\quad\text{as }m\to\infty.
\end{equation*}
\end{lemma}
\begin{proof}
With the help of Lemma~\ref{ls2}, we have
\begin{equation}\label{rho}
\|\rho_{m,s}\|_{L^p(\mathbb{R}^n)}\leq \|\rho_{m,s}\|_{L^1(\mathbb{R}^n)}^{1/p}\|\rho_{m,s}\|_{L^{\infty}(\mathbb{R}^n)}^{(p-1)/p}\leq M^{1/p}\alpha_m^{(p-1)/p}\leq C,
\end{equation}
where $\alpha_m\leq \frac{1+\sqrt{1+\frac{8M}{n(n-2)\omega_n}}}{2}$ from Lemma~\ref{ls2} and $C=\max\{M, \frac{1+\sqrt{1+\frac{8M}{n(n-2)\omega_n}}}{2}\}$.

It similarly follows from \eqref{m43} that
\begin{equation*}
\begin{aligned}
\int_{\mathbb{R}^n}\nabla\mathcal{N}\ast\rho_{m,s}\cdot\nabla\mathcal{N}\ast\rho_{m,s}dx\leq C.
\end{aligned}
\end{equation*}
The $L^\infty$ estimate of $\nabla\mathcal{N}\ast\rho_{m,s}$ easily holds:
\begin{equation}\label{bn}
\begin{aligned}
|\nabla\mathcal{N}\ast\rho_{m,s}|
\leq& C\int_{|x-y| \leq1}\frac{\rho_{m,s}(y,t)}{|x-y|^{n-1}}dy+C\int_{|x-y|>1}\frac{\rho_{m,s}(y)}{|x-y|^{n-1}}dy\\
\leq& C\alpha_m \int_{|x-y|\leq 1}\frac{1}{|x-y|^{n-1}}dy+C\int_{|x-y|>1}\rho_{m,s}dy\\
\leq& C.
\end{aligned}
\end{equation}

Similar to \eqref{m25}, it follows from \eqref{rho} that
\begin{equation*}
\|\nabla^2\mathcal{N}\ast\rho_{m,s}\|_{L^p(\mathbb{R}^n)}\leq C(p)\|\rho_{m,s}\|_{L^p(\mathbb{R}^n)}\leq C(M,p)
\end{equation*}
where $C(p)\sim\frac{1}{p-1}$ for $0<p-1\ll 1$ and $C(p)\sim p$ for $p\gg1$.

Then, thanks to Sobolev's embedding theorem, there exists $\nabla\mathcal{N}\ast\rho_{\infty,s}\in L^2\cap L^\infty(\mathbb{R}^n)$ such that
\begin{equation*}
\nabla \mathcal{N}\ast\rho_{m,s}\to \nabla \mathcal{N}\ast\rho_{\infty,s},\quad\text{strongly in }L^p_{loc}(\mathbb{R}^n)\text{ for }1\leq p<\infty,\quad\text{as }m\to\infty.
\end{equation*}
\end{proof}

In the following, we establish the Aronson-B\'enilan (AB) estimate  corresponding tio the stationary state and thus a second order spatial derivative estimate of the pressure. Similar to the case of PKS model, we use the notation
       \begin{equation}\label{a21}
       \omega_{m,s}:=\Delta P_{m,s}+\rho_{m,s}.
       \end{equation}

       \begin{lemma}[Aronson-B\'enilan estimate]$\label{ls6}$
       Let $\rho_{m,s}$ be a weak solution to the SPKS model Eq.~\eqref{SPKS} in the sense of Def.~\ref{d1} with $\|\rho_{m,s}\|_{L^1(\mathbb{R}^n)}=M$, then,  for all $m\geq 3$,
       \begin{equation}\label{mainreult}
        \||\omega_{m,s}|_-\|_{L^3(\mathbb{R}^n)}^3\leq C/m,\quad\||\omega_{m,s}|_-\|_{L^1(\mathbb{R}^n)}\leq C/m^{1/3},\quad\|\Delta P_{m,s}\|_{L^1(\mathbb{R}^n)}\leq C.
       \end{equation}
      
       \end{lemma}
       \begin{proof}
        To begin with,  we rewrite the SPKS model Eq.~\eqref{SPKS} as
       \begin{equation*}\label{a17}
      0=\Delta \rho_{m,s}^m+\nabla\cdot(\rho_{m,s}\nabla\mathcal{N}\ast\rho_{m,s})
       =\rho_{m,s}\omega_{m,s}+\nabla\rho_{m,s}\cdot(\nabla P_{m,s}+\nabla\mathcal{N}\ast\rho_{m,s}),
       \end{equation*}
       and the pressure equation Eq.~\eqref{e14} is
       \begin{equation*}
       (m-1)P_{m,s}\omega_{m,s}+\nabla P_{m,s}\cdot\nabla P_{m,s}+\nabla P_{m,s}\cdot\nabla\mathcal{N}\ast\rho_{m,s}=0.
       \end{equation*}
      Take the Laplace operator $(\Delta)$ action on the above equation, then we have
       \begin{equation*}\label{a18}
       \begin{aligned}
    (m-1)\Delta(P_{m,s}& \omega_{m,s})+\nabla(\Delta P_m)\cdot(\nabla\mathcal{N}\ast\rho_{m,s}+\nabla P_{m,s})\\
       &+\nabla P_{m,s}\cdot\nabla \omega_{m,s}+2\nabla^2P_{m,s}:(\nabla^2P_{m,s}+\nabla^2\mathcal{N}\ast\rho_{m,s})=0.
       \end{aligned}
       \end{equation*}
     Hence, the equation of $\omega_{m,s}$ is as follows
       \begin{equation*}
       \begin{aligned}
      (m-1)\Delta(P_{m,s} &\omega_{m,s})+\nabla\omega_{m,s}\cdot\nabla\mathcal{N}\ast\rho_{m,s}+2\nabla^2P_{m,s}:(\nabla^2P_{m,s}+\nabla^2\mathcal{N}\ast\rho_{m,s})\\
       &+\rho_{m,s}\omega_{m,s}+2\nabla P_{m,s}\cdot\nabla\omega_{m,s}=0.
       \end{aligned}
       \end{equation*}
       Then, it follows from Kato's inequality that
       \begin{equation}\label{f18}
       \begin{aligned}
0\leq& -(m-1)\Delta(P_{m,s}\omega_{m,s})-\nabla\omega_{m,s}\cdot\nabla\mathcal{N}\ast\rho_{m,s}+\frac{1}{2}(\nabla^2\mathcal{N}\ast\rho_{m,s})^2\\
       &-\rho_{m,s}\omega_{m,s}-2\nabla P_{m,s}\cdot\nabla\omega_{m,s},
       \end{aligned}
       \end{equation}
     where we use the fact
\begin{equation*}
\begin{aligned}
2\nabla^2P_{m,s}:(\nabla^2P_{m,s}+\nabla^2\mathcal{N}\ast\rho_{m,s})
\geq -\frac{1}{2}(\nabla^2\mathcal{N}\ast\rho_{m,s})^2.
\end{aligned}
\end{equation*}

   Multiplying \eqref{f18} by $|\omega_{m,s}|_-$ and thanks to Kato's inequality, we have
   \begin{equation}\label{ff20}
   \begin{aligned}
   0\leq &(m-1)\Delta(P_{m,s}|\omega_{m,s}|_-)|\omega_{m,s}|_-+\frac{1}{2}\nabla|\omega_m|_-^2\cdot\nabla\mathcal{N}\ast\rho_m+\nabla|\omega_m|_-^2\cdot\nabla P_{m,s}\\
&+\frac{1}{2}(\nabla^2\mathcal{N}\ast\rho_{m,s})^2|\omega_{m,s}|_-+\rho_{m,s}|\omega_{m,s}|_-^2.
   \end{aligned}
   \end{equation}
  Similar to \eqref{OMEGA}, it easily holds
\begin{equation*}
\begin{aligned}
(m-1)\int_{\mathbb{R}^n} &\Delta(P_{m,s}|\omega_{m,s}|_-)  |\omega_{m,s}|_-dx= -\frac{1}{2}(m-1)\int_{\mathbb{R}^n} |\omega_{m,s}|_-^3dx
\\
&-\frac{1}{2}(m-1)\int_{\mathbb{R}^n}\rho_{m,s}|\omega_{m,s}|_-^2dx -(m-1)\int_{\mathbb{R}^n}P_{m,s}|\nabla|\omega_{m,s}|_-|^2dx.
\end{aligned}
\end{equation*}
Then, integrating \eqref{ff20} on $\mathbb{R}^n$, we use the above inequality and obtain
\begin{equation}\label{ff}
\begin{aligned}
2(m-1)\int_{\mathbb{R}^n} &P_{m,s}|\nabla|\omega_{m,s}|_-|^2dx+(m-4)\int_{\mathbb{R}^n}\rho_{m,s}|\omega_{m,s}|_-^2dx\\
&+(m-3)\int_{\mathbb{R}^n}|\omega_{m,s}|_-^3dx
-\int_{\mathbb{R}^n}(\nabla^2\mathcal{N}\ast\rho_{m,s})^2|\omega_{m,s}|_-dx\leq0.
\end{aligned}
\end{equation}
By Young's inequality, Lemma~\ref{ls2}, and the singular integral theory for Newtonian potential (Lemma~\ref{l12}), we have
\begin{equation}\label{f21}
\begin{aligned}
\int_{\mathbb{R}^n}(\nabla^2\mathcal{N}\ast\rho_{m,s})^2|\omega_{m,s}|_-dx&\leq\sum_{ij}\frac{2n}{3^{3/2}}\int_{\mathbb{R}^n}|\partial_{ij}^2\mathcal{N}\ast\rho_{m,s}|^3dx+\sum_{ij}\frac{1}{n^2}\int_{\mathbb{R}^n}|\omega_{m,s}|_-^3dx\\
&\leq \sum_{ij}\frac{2n}{3^{3/2}}C\int_{\mathbb{R}^n}\rho_{m,s}^3dx+\int_{\mathbb{R}^n}|\omega_{m,s}|_-^3dx\\
&\leq \int_{\mathbb{R}^n}|\omega_{m,s}|_-^3dx+CM\alpha_m^{2} \leq \int_{\mathbb{R}^n}|\omega_{m,s}|_-^3dx+C.
\end{aligned}
\end{equation}
Taking (\ref{f21}) into (\ref{ff}), then we attain
\begin{equation*}
2(m-1)\int_{\mathbb{R}^n}P_{m,s}|\nabla|\omega_{m,s}|_-|^2dx+(m-4)\int_{\mathbb{R}^n}\rho_{m,s}|\omega_{m,s}|_-^2dx+(m-4)\int_{\mathbb{R}^n}|\omega_{m,s}|_-^3dx
\leq C.
\end{equation*}
Thus, the first estimate of \eqref{mainreult} holds
\begin{equation*}
\begin{aligned}
\int_{\mathbb{R}^n}|\omega_{m,s}|_-^3dx\leq C/m.
\end{aligned}
\end{equation*}

From Lemma~\ref{ls4} with a given mass $M>0$, there exists a positive constant $R_*(M)$ (only depending on $M$) such that
\begin{equation*}
 {\rm supp}(|\omega_{m,s}|_-)\subset B_{R_*(M)} ,
\end{equation*}
then we have
\begin{equation*}
\int_{\mathbb{R}^n}|\omega_{m,s}|_-dx=\int_{B_{R_*(M)} }|\omega_{m,s}|_-dx\leq |B_{R_*(M)} |^{2/3}\big(\int_{B_{R_*(M)} }|\omega_{m,s}|_-^3dx\big)^{1/3}\leq C/m^{1/3},
\end{equation*}
so the second estimate of \eqref{mainreult} holds. For the last estimate of \eqref{mainreult}, by triangle inequality and direct computations, we obtain
\begin{equation*}
\begin{aligned}
\int_{\mathbb{R}^n}|\Delta P_{m,s}|dx\leq& \int_{\mathbb{R}^n}|\Delta P_{m,s}+\rho_{m,s}|dx+M\\
=&\int_{\mathbb{R}^n}(\Delta P_{m,s}+\rho_{m,s})dx+2\int_{\mathbb{R}^n}|\omega_{m,s}|_-dx+M\\
=&2\int_{\mathbb{R}^n}|\omega_{m,s}|_-dx+2M\\
\leq& C.
\end{aligned}
\end{equation*}
The proof is completed.
       \end{proof}
Next, we turn to show the $L^\infty$ estimate of the pressure gradient.

\begin{lemma}[$L^\infty$ estimate on pressure gradient]$\label{ls7}$ Let $\rho_{m,s}$ be a weak solution to the SPKS model Eq.~\eqref{SPKS} in the sense of Def.~\ref{d1} with a given mass $M>0$. Then it holds for all $m\geq 3$ that
\begin{equation*}
\|\nabla P_{m,s}\|_{L^\infty\cap L^{1}(\mathbb{R}^n)}\leq C.
\end{equation*}
\end{lemma}
\begin{proof}
Since
\begin{equation*}
\rho_{m,s}\nabla P_{m,s}=-\rho_{m,s}\nabla\mathcal{N}\ast\rho_{m,s},
\end{equation*}
we have
\begin{equation*}
\nabla P_{m,s}=-\nabla\mathcal{N}\ast\rho_{m,s}\quad\text{for }x\in  {\rm supp}(P_{m,s}).
\end{equation*}
It follows from Lemma~\ref{ls5} that
\begin{equation*}
|\nabla P_{m,s}|\leq C\quad\text{for }x\in  {\rm supp}(P_{m,s}).
\end{equation*}

  From Lemma~\ref{ls1}, there exists $R_{m}(M)\leq R_{*}(M)$ for all $m\geq3$ such that
\begin{equation*}
 {\rm supp}(P_{m,s})=B_{R_{m}(M)} \subset B_{R_{*}(M)} ,
\end{equation*}
which together with the radially symmetric property of $P_{m,s}$ (Lemma~\ref{ls1}) means
\begin{equation*}
\|\nabla P_{m,s}\|_{L^{\infty}\cap L^{1}(\mathbb{R}^n)}\leq C.
\end{equation*}
\end{proof}
In the end, with the regularity estimates on $\rho_{m,s},P_{m,s}$, and $\mathcal{N}\ast\rho_{m,s}$, we are going to prove the incompressible (Hele-Shaw) limit of the SPKS model Eq.~\eqref{SPKS}.
\begin{lemma}[Incompressible limit]$\label{ls8}$
Let $\rho_{m,s}$ be a weak solution to the SPKS
Eq.~\eqref{SPKS} in the sense of Def.~\ref{d1} with $\int_{\mathbb{R}^n}x\rho_{m,s}(x)dx=0$, $\|\rho_{m,s}\|_{L^1(\mathbb{R}^n)}=M$, and $m\geq3$. Then, after extracting the subsequence, there exist $P_{\infty,s},\nabla P_{\infty,s}\in L^1(\mathbb{R}^n)\cap L^{\infty}(\mathbb{R}^n)$ such that
\begin{align}
&\nabla P_{m,s}\to\nabla P_{\infty,s},\quad\text{strongly in }L^r(\mathbb{R}^n)\text{ for }1\leq r<\infty,&&\text{ as }m\to\infty,\label{k1}\\
&P_{m,s}\to P_{\infty,s},\quad\quad\  \text{ strongly in }L^1(\mathbb{R}^n)\cap L^\infty(\mathbb{R}^n),&&\text{ as }m\to\infty.\label{k2}
\end{align}

Furthermore, there exists $\rho_{\infty,s}\in L^1(\mathbb{R}^n)\cap L^{\infty}(\mathbb{R}^n)$ such that
\begin{align}
&\|\rho_{\infty,s}\|_{L^1(\mathbb{R}^n)}=M,&&\int_{\mathbb{R}^n}x\rho_{\infty,s}dx=0,\label{k3}\\
&0\leq \rho_{\infty,s}\leq 1, &&\text{a.e. in }\mathbb{R}^n,\label{k4}\\
&(1-\rho_{\infty,s})P_{\infty,s}=0, &&\text{a.e. in }\mathbb{R}^n,\label{k5}\\
&(1-\rho_{\infty,s})\nabla P_{\infty,s}=0, &&\text{a.e. in }\mathbb{R}^n,\label{k6}\\
&\Delta P_{\infty,s}+\rho_{\infty,s}\geq 0,&&\text{in }\mathcal{D}'(\mathbb{R}^n),\label{k7}\\
&\nabla P_{\infty,s}+\rho_{\infty,s}\nabla \mathcal{N}\ast\rho_{\infty,s}=0,&&\text{in }\mathcal{D}'(\mathbb{R}^n).\label{k8}
\end{align}

Moreover, it holds for $R(M)>0$ satisfying $|B_{R(M)} |_n=M$ that
\begin{equation}\label{sss}
\rho_{\infty,s}=\chi_{\{P_{\infty,s}>0\}}=\chi_{B_{R(M)} }\quad\text{a.e. in }\mathbb{R}^n.
\end{equation}
\end{lemma}

\begin{proof}
Since $\|\nabla P_{m,s}\|_{L^1(\mathbb{R}^n)\cap L^\infty(\mathbb{R}^n)}\leq C$ (Lemma~\ref{ls7}) and $ {\rm supp}(P_{m,s})\subset B_{R_*(M)} $ (Lemma~\ref{ls4}), then there exist $\nabla P_{\infty,s}\in L^1(\mathbb{R}^n)\cap L^\infty(\mathbb{R}^n)$ and $ {\rm supp}(P_{\infty,s})\subset B_{R_*(M)} $ such that
\begin{equation}\label{ff33}
\begin{aligned}
\nabla P_{m,s}\rightharpoonup \nabla P_{\infty,s},\quad\text{weakly in }L^r(\mathbb{R}^n)\text{ for }1<r<\infty,\quad\text{as }m\to\infty.
\end{aligned}
\end{equation}
 Thanks to $\|\Delta P_{m,s}\|_{L^1(\mathbb{R}^n)}\leq C$ (Lemma~\ref{ls6})  and $ {\rm supp}(P_{m,s})\subset B_{R(M)} $ (Lemma~\ref{ls4}), then it follows from the compactness criterion in \cite[(21)]{boga} and $P_{m,s}\in W^{\infty}_{0}(B_{R_*(M)} )$ that
 \begin{equation}\label{ff34}
 \nabla P_{m,s}\to \nabla P_{\infty,s},\quad\text{strongly in }L^r(\mathbb{R}^n)\text{ for } 1\leq r<\infty,\quad\text{as }m\to\infty.
 \end{equation}
By Sobolev's inequality for gradient (Theorem~\ref{t8}), we obtain
 \begin{equation*}
 \|P_{m,s}-P_{\infty,s}\|_{L^\infty(\mathbb{R}^n)}\leq \|\nabla P_{m,s}-\nabla P_{\infty,s}\|_{L^n(\mathbb{R}^n)}\to0,\quad\text{as }m\to\infty.
 \end{equation*}
So Eq.~\eqref{k1}--\eqref{k2} hold.

In addition, Eq.~\eqref{k3} follows from $\|\rho_{m,s}\|_{L^1(\mathbb{R}^n)}=M$, $\int_{\mathbb{R}^n}x\rho_{m,s}dx$=0, and $ {\rm supp}(\rho_{m,s})\subset B_{R_*(M)} $ show that as $m\to \infty$
\begin{equation*} 
\begin{aligned}
\int_{\mathbb{R}^n}\rho_{m,s}dx\to\int_{\mathbb{R}^n}\rho_{\infty,s}dx=M, \qquad 
 \int_{\mathbb{R}^n}x\rho_{m,s}dx\to\int_{\mathbb{R}^n}x\rho_{\infty,s}dx=0.
\end{aligned}
\end{equation*}
Since
\begin{equation*}
\|\rho_{m,s}\|_{L^r(\mathbb{R}^n)}\leq \|\rho_{m,s}\|_{L^{\infty}(\mathbb{R}^n)}^{(r-1)/r}\|\rho_{m,s}\|_{L^1(\mathbb{R}^n)}^{1/r}\leq \alpha_m^{(r-1)/r(m-1)} M^{1/r}\leq C^{(r-1)/r(m-1)} M^{1/r} ,
\end{equation*}
there exists $\rho_{\infty,s}\in L^{r}(\mathbb{R}^n)$ for $1<r<\infty$ such that
\begin{equation*}
\rho_{m,s}\rightharpoonup \rho_{\infty,s},\quad \text{weakly in } L^r(\mathbb{R}^n)\text{ for }1<r<\infty,\quad\text{as }m\to\infty.
\end{equation*}

According to the weak semi-continuity of $L^r$ norm for $1<r<\infty$, we have
\begin{equation*}
\begin{aligned}
\|\rho_{\infty,s}\|_{L^r(\mathbb{R}^n)}
\leq & \liminf\limits_{m\to\infty}\|\rho_{m,s}\|_{L^r(\mathbb{R}^n)} 
\leq \liminf\limits_{m\to\infty}(\alpha_m^{m-1})^{(r-1)/r(m-1)}M^{1/r}\\
\leq &\liminf\limits_{m\to\infty}(\frac{1+\sqrt{1+\frac{4M^2}{n^2(n-2)^2}}}{2}+\frac{2M}{n(n-2)\omega_n})^{(r-1)/r(m-1)}M^{1/r}\\
\leq &M^{1/r}.
\end{aligned}
\end{equation*}
We take $r\to\infty$ and obtain $\|\rho_{\infty,s}\|_{L^\infty(\mathbb{R}^n)}\leq 1$, 
which shows Eq.~\eqref{k4}.
\\

Let $\varphi(x)\in C_{0}^{\infty}(\mathbb{R}^n)$ be a smooth test function. Due to the definition of weak solution~\eqref{def:WS}, we have
\begin{equation}\label{f36}
\int_{\mathbb{R}^n}\rho_{m,s}\nabla P_{m,s}\cdot\nabla\varphi dx+\int_{\mathbb{R}^n}\rho_{m,s}\nabla\mathcal{N}\ast\rho_{m,s}\cdot\nabla\varphi dx=0.
\end{equation}
Passing \eqref{f36} to limit, then we obtain
\begin{equation}\label{f37}
\int_{\mathbb{R}^n}\rho_{\infty,s}\nabla P_{\infty,s}\cdot\nabla\varphi dx+\int_{\mathbb{R}^n}\rho_{\infty,s}\nabla\mathcal{N}\ast\rho_{\infty,s}\cdot\nabla\varphi dx=0.
\end{equation}
Since $\rho_{m,s}=(\frac{m-1}{m}P_{m,s})^{1/(m-1)}$, we have
\begin{equation}\label{f38}
\begin{aligned}
&\int_{\mathbb{R}^n}\rho_{\infty,s}P_{\infty,s}\varphi dx=\lim\limits_{m\to\infty}\int_{\mathbb{R}^n}\rho_{m,s}P_{m,s}\varphi dx\\
=&\lim\limits_{m\to\infty}\int_{\mathbb{R}^n}(\frac{m-1}{m})^{1/(m-1)}(P_{m,s})^{m/(m-1)}\varphi dx
=&\int_{\mathbb{R}^n}P_{\infty,s}\varphi dx.
\end{aligned}
\end{equation}
Similarly, it holds for $i=1,...,n$ that
\begin{equation}\label{f39}
\begin{aligned}
\int_{\mathbb{R}^n}\rho_{\infty,s}\partial_iP_{\infty,s}\varphi dx=&\lim\limits_{m\to\infty}\int_{\mathbb{R}^n}\rho_{m,s}\partial_iP_{m,s}\varphi dx\\
=&\lim\limits_{m\to\infty}\int_{\mathbb{R}^n}(\frac{m-1}{m})^{m/(m-1)}P_{m,s}^{m/(m-1)}\partial_i\varphi dx\\
=&-\int_{\mathbb{R}^n}P_{\infty,s}\partial_i\varphi dx 
=\int_{\mathbb{R}^n}\partial_i P_{\infty,s}\varphi dx.
\end{aligned}
\end{equation}
Using $\rho_{\infty,s}, P_{\infty,s},\nabla P_{\infty,s}\in L^{r}(\mathbb{R}^n)$ for $1<r<\infty$, Eqs.~\eqref{k5}--\eqref{k6} follows from \eqref{f38}--\eqref{f39}.

From Lemma~\ref{ls6}, we obtain
 \begin{equation*}
 \int_{\mathbb{R}^n}(\Delta P_{m,s}+\rho_{m,s})\varphi dx\geq -\int_{\mathbb{R}^n}|\omega_{m,s}|_-\varphi dx\geq -C\||\omega_{m,s}|_-\|_{L^1(\mathbb{R}^n)}\geq -C/m^{1/3},
 \end{equation*}
 which means that, after taking the limit, Eq.~\eqref{k7} holds for a nonnegative smooth test function as
 \begin{equation*}
  \int_{\mathbb{R}^n}(\Delta P_{\infty,s}+\rho_{\infty,s})\varphi dx\geq 0.
 \end{equation*}

Combining \eqref{f37} and{f40}, the last statement Eq.~\eqref{k8} is gotten as
\begin{equation*}
\nabla P_{\infty,s}+\rho_{\infty,s}\nabla\mathcal{N}\ast\rho_{\infty,s}=0,\quad\text{a.e. in }\mathbb{R}^n.
\end{equation*}

Since $\rho_{m,s}$ is a solution to the SPKS model Eq.~\eqref{SPKS} with $\int_{\mathbb{R}^n}x\rho_{m,s}(x)dx=0$, $\rho_{m,s}$ and $P_{m,s}$ are radially decreasing symmetric. Therefore, $\rho_{\infty,s}$ and $P_{\infty,s}$ are radially decreasing symmetric, and $\mathcal{N}\ast\rho_{\infty,s}$ is radially symmetric (Lemma~\ref{ls1}). Since $\rho_{\infty,s}P_{\infty,s}=P_{\infty,s}$, there exists $R_1\leq R_2$ such that
\begin{equation*}
 {\rm supp}(P_{\infty,s})=B_{R_1} ,\quad  {\rm supp}(\rho_{\infty,s})=B_{R_2} .
\end{equation*}
If $R_1<R_2$, it follows from Lemma~\ref{ls8} that
\begin{equation}\label{pp}
\rho_{\infty,s}(r)\frac{\partial}{\partial_r}\mathcal{N}\ast\rho_{\infty,s}(r)=0\quad\text{for }x\in(R_1,R_2).
\end{equation}
However, it follows from \cite[(2.2)]{IY} that for $r>0$
\begin{equation*}\label{nm}
\begin{aligned}
\frac{\partial}{\partial r}\mathcal{N}\ast\rho_{\infty,s}(r)=&\frac{1}{|\partial B_r|_{n-1}}\int_{\partial B_r}|\nabla\mathcal{N}\ast\rho_{\infty,s}|dS_x
=\frac{1}{|\partial B_r |_{n-1}}\int_{\partial B_r}\nabla\mathcal{N}\ast\rho_{\infty,s}\cdot\nu dS_x
\\
=&\frac{1}{|\partial B_r |_{n-1}}\int_{B_r}\Delta \mathcal{N}\ast\rho_{\infty,s} dx \frac{1}{|\partial B_r |_{n-1}}\int_{B_r }\rho_{\infty,s} dx>0,
\end{aligned}
\end{equation*}
which contradicts \eqref{pp}. Thus, we have $R_1=R_2$, and the last statement Eq.~\eqref{sss} is proved.
\end{proof}

\section{Conclusion, extensions and perspectives}
In order to prove the incompressible limit of the Patlak-Keller-Segel system, and
establish the corresponding Hele-Shaw free boundary equation, we have followed
the same lines of proof as initiated in~\cite{5}, with the gradient estimate as
in~\cite{r35}. This has the advantage to also prove optimal second order
estimates. A fundamental new ingredient is a uniform $L^1$ estimate on the time
derivative of pressure. With this new estimate, another possible route to
establish the complementarity condition, the hard part of the problem, would be
to use the pure compactness method in \cite{r62,rPrX}. Still another possible
route is through the obstacle problem, see~\cite{GKM} and the references therein.
We have also established uniqueness, finite propagation speed, and limit energy
functional of solutions to this Hele-Shaw type free boundary problem. In addition, we have studied the incompressible limit for the stationary state of the PKS model, which is new for the diffusion-aggregation equations.

We would like to point out that our analysis for the PKS model is compatible with growth terms, as
they appear naturally when dealing with mechanical models of tumor growth, even
though the technical details have to be checked. Also we treated dimension
$n\geq 3$ to avoid technical issues with the Sobolev inequalities but we do not
expect difficulties in two dimensions.

Several papers have treated of linear drift terms, see~\cite{r20, r50}. But it is
difficult to extend these cases or Newtonian potential to more general attractive
potential, because our proof of the time derivative estimate of pressure strongly
depends on the structure of Newtonian potential. Among open problems, let us also
mention the convergence rate with $m\to \infty$, which has been obtained in few
papers, \cite{r108, rDDP}. Finally the case of systems is only treated without drift, see \cite{r53, r62}. Large time asymptotic of solutions to the Hele-Shaw system \eqref{z6}--\eqref{z8} is an interesting topic.
\cite{CKY_2018} treated  the 2-dimensional Hele-Shaw model when the initial density is a patch function.
But that for the n-dimensional ($n\geq3$) Hele-Shaw case is  still largely open. The regularity of free boundary for a
  Hele-Shaw problem of tumor growth was obtained in~\cite{r60}, also for the porous medium equation with an external drift, cf.~\cite{r20}.

\paragraph{Acknowledegements}

The authors would like to thank Noemi David and Markus Schmidtchen for helpful
discussions and comments.
\\
 The research of Q. H. and H-L. L. was supported partially by National Natural Science Foundation of China (No.11931010, 11671384 and 11871047), and by the key research project of Academy for Multidisciplinary Studies, Capital Normal
 University, and by the Capacity Building for Sci-Tech Innovation-Fundamental
 Scientific Research Funds (No.007/20530290068).
 \\
B.P. has received funding from the European Research Council (ERC) under the
European Union's Horizon 2020 research and innovation programme (grant agreement
No.740623).

\appendix
\section{Appendix: $L^\infty$ bound}
\label{AAA}

Under the initial assumption~\eqref{c1}, we show the solution to \eqref{d1}  is bounded in any finite time using Moser's iteration technique. In order to justify the $L^\infty$ estimate of density, we introduce the following approximate equations of the Cauchy problem \eqref{d1}:
\begin{equation}\label{ae}
\begin{cases}
\partial_t \rho_{\varepsilon}=\Delta \rho_{\varepsilon}^m+\varepsilon\Delta \rho_{\varepsilon}+\nabla\cdot\big(\rho_{\varepsilon}\nabla (J_\varepsilon\ast\mathcal{N}_{\varepsilon}\ast\rho_{\varepsilon})\big),&\\
\rho_{\varepsilon}(x,0)=\rho_{0,\varepsilon}(x)=J_\varepsilon\ast\rho_{0},\\
\|\rho_\varepsilon(t)\|_{L^1(\mathbb{R}^n)}=\|\rho_{0,\varepsilon}\|_{L^1(\mathbb{R}^n)}\leq \|\rho_0\|_{L^1(\mathbb{R}^n)},
\end{cases}
\end{equation}
where $J_{1}\in C^{\infty}(\mathbb{R}^n)$ with both $\|J_{1}\|_{L^1(\mathbb{R}^n)}=1$ and $ {\rm supp}(J_1)\subset B_1 $ is the standard mollifier  and $J_{\varepsilon}:=\frac{1}{\varepsilon^d}J_1(\frac{x}{\varepsilon})$. We have

\begin{lemma}\label{lr}
Let the initial assumption~\eqref{c1} hold. Then, the solution $\rho_{\varepsilon}$ to the Cauchy problem Eq.~\eqref{ae} satisfies for any $T>0$, $r \in [1, \infty)$,  
\begin{equation*}
\sup\limits_{0\leq t\leq T}\|\rho_{\varepsilon}(t)\|_{L^r(\mathbb{R}^n)}\leq C(M,T,m,n,r,\|\rho_0\|_{L^\infty(\mathbb{R}^n)}).
\end{equation*}
\end{lemma}

\begin{proof}
Multiplying Eq.~\eqref{ae} by $\rho_{\varepsilon}^{r-1}$ and integrating on $\mathbb{R}^n$, we find
\begin{equation}\label{AA:dtrho}
\begin{aligned} 
\frac{d}{dt}\| \rho_{\varepsilon}\|_{L^r}^{r}+\frac{4mr(r-1)}{(r+m-1)^2}\int_{\mathbb{R}^n}|\nabla\rho_{\varepsilon}^{\frac{r+m-1}{2}}|^2dx 
\leq&-r\int_{\mathbb{R}^n}\rho_\varepsilon\nabla\mathcal{N}\ast(J_\varepsilon\ast\rho_\varepsilon)\cdot\nabla\rho_{\varepsilon}^{r-1}dx \\
\leq&(r-1)\int_{\mathbb{R}^n}\rho_{\varepsilon}^{r}J_{\varepsilon}\ast\rho_{\varepsilon}dx\\
\leq &(r-1)\int_{\mathbb{R}^n}\rho_{\varepsilon}^{r+1}dx, 
\end{aligned}
\end{equation}
because it follows from Holder's inequality and convolution-type Young's inequality that
\begin{equation*}
\begin{aligned}
\int_{\mathbb{R}^n}\rho_{\varepsilon}^rJ_{\varepsilon}\ast\rho_{\varepsilon}dx &\leq (\int_{\mathbb{R}^n}\rho_{\varepsilon}^{r+1}dx)^{r/(r+1)} (\int_{\mathbb{R}^n}(J_\varepsilon\ast\rho_{\varepsilon})^{r+1}dx)^{1/(r+1)}\\
&\leq \|\rho_{\varepsilon}\|_{L^{r+1}(\mathbb{R}^n)}^r\|J_\varepsilon\|_{L^1(\mathbb{R}^n)}\|\rho_\varepsilon\|_{L^{r+1}(\mathbb{R}^n)}\\
&=\|\rho_{\varepsilon}\|_{L^{r+1}(\mathbb{R}^n)}^{r+1}.
\end{aligned}
\end{equation*}
Thanks to the interpolation inequality (Theorem~\ref{t1}) and Sobolev's inequality for gradient~(Theorem~\ref{t8}), we have
\begin{equation*}
\begin{aligned}
\|\rho_{\varepsilon}\|_{L^{r+1}(\mathbb{R}^n)}^{r+1}&\leq(\int_{\mathbb{R}^n}\rho_{\varepsilon}dx)^{\frac{r+mn-2n+1}{rn+mn-2n+1}}(\int_{\mathbb{R}^n}\rho_{\varepsilon}^{\frac{(r+m-1)n}{n-1}}dx)^{\frac{r(n-1)}{rn+mn-2n+1}}\\
&\leq C(M,m,n,r)(\frac{4mr}{(r+m-1)^2}\int_{\mathbb{R}^n}|\nabla \rho_{\varepsilon}^{\frac{m+r-1}{2}}|^2dx)^{\frac{rn}{rn+mn-2n+1}}\\
&\leq C(M,m,n,r)+\frac{4mr}{(r+m-1)^2}\int_{\mathbb{R}^n}|\nabla \rho_{\varepsilon}^{\frac{m+r-1}{2}}|^2dx,
\end{aligned}
\end{equation*}
and it follows from \eqref{AA:dtrho} that
\begin{equation}\label{aaa1}
\frac{d}{dt}\|\rho_{\varepsilon}\|_{L^r}^{r}\leq C(M,m,n,r).
\end{equation}
But the initial data is independent of~$\varepsilon$ since for $r\geq 1$, we have 
\begin{equation*}
\|J_\varepsilon\ast\rho_0\|_{L^r(\mathbb{R}^n)}\leq \|J_\varepsilon\|_{L^1(\mathbb{R}^n)}\|\rho_0\|_{L^r(\mathbb{R}^n)}\leq \|\rho_0\|_{L^1(\mathbb{R}^n)}^{\frac{1}{r}}\|\rho_0\|_{L^\infty(
\mathbb{R}^n)}^{\frac{r-1}{r}}\leq C.
\end{equation*}
Therefore, integrating \eqref{aaa1} on $[0,t]$ the proof of the lemma is completed.
\end{proof}

The $L^\infty$ estimate of the density $\rho_{\varepsilon}$ is obtained by adapting the proof in~\cite{Sugi06}.
\begin{lemma}\label{ub}
Under the initial assumptions~\eqref{c1},  the solution $\rho_{\varepsilon}$ to the Cauchy problem Eq.~\eqref{ae} satisfies for all $T>0$
\begin{equation*}
\sup\limits_{0\leq t\leq T}\|\rho_{\varepsilon}(t)\|_{L^\infty(\mathbb{R}^n)}\leq C(M,T,m,n,\|\rho_0\|_{L^\infty(\mathbb{R}^n)}).
\end{equation*}
\end{lemma}

\begin{proof}
Similar to \eqref{m22}, by Young's inequality we have, for $t\in[0,T]$, 
\begin{equation*}
\begin{aligned}
|\nabla(J_\varepsilon\ast\mathcal{N}\ast \rho_{\varepsilon}(t))| =&|\nabla\mathcal{N}\ast(J_\varepsilon\ast\rho_{\varepsilon}(t))|\\
\leq&
C(n)\Big(\int_{\mathbb{R}^n}(J_\varepsilon \ast\rho_\varepsilon)^{2n-1}(y,t)dy\Big)^{\frac{1}{2n-1}}+
C(n)\|J_\varepsilon \ast\rho_\varepsilon(t)\|_{L^1(\mathbb{R}^n)}\\
\leq& C(n)(\|\rho_\varepsilon(t)\|_{L^{2n-1}(\mathbb{R}^n)}+\|\rho_\varepsilon(t)\|_{L^1(\mathbb{R}^n)})\\
\leq& C(M,T,m,n,\|\rho_0\|_{L^\infty(\mathbb{R}^n)}).
\end{aligned}
\end{equation*}
Multiplying Eq.~\eqref{ae} by $\rho_{\varepsilon}^{r-1}$ and integrating on $\mathbb{R}^n$, it follows that, for $r\geq m\geq1$, 
\begin{equation*}
\begin{aligned}
\frac{1}{r}\frac{d}{dt}& \|\rho_{\varepsilon}\|_{L^r}^{r}+\frac{4m(r-1)}{(r+m-1)^2}\int_{\mathbb{R}^n}|\nabla\rho_{\varepsilon}^{\frac{r+m-1}{2}}|^2dx
\leq(r-1)\int_{\mathbb{R}^n}\rho_{\varepsilon}^{r-1}\nabla \rho_{\varepsilon}\cdot\nabla (J_\varepsilon\ast\mathcal{N}\ast\rho_{\varepsilon})dx\\
\leq& \frac{2m(r-1)}{(r+m-1)^2}\int_{\mathbb{R}^n}|\nabla\rho_\varepsilon^{\frac{r+m-1}{2}}|^2dx+\frac{(r-1)}{2 m}\|\nabla(J_\varepsilon\ast\mathcal{N}\ast\rho_{\varepsilon})(t)\|_{L^\infty(\mathbb{R}^n)}^2\int_{\mathbb{R}^n}\rho_{\varepsilon}^{r+1-m}dx.
\end{aligned}
\end{equation*}
Furthermore, it holds by H\"older's and Young's inequalities that
\begin{equation}\label{ed1}
\begin{aligned}
\frac{1}{r}\frac{d}{dt} \|\rho_{\varepsilon}\|_{L^r}^{r}+&\frac{2m(r-1)}{(r+m-1)^2}\int_{\mathbb{R}^n}|\nabla\rho_{\varepsilon}^{\frac{r+m-1}{2}}|^2dx\\
\leq& (r-1)C(M,T,m,n,\|\rho_{0}\|_{L^\infty(\mathbb{R}^n)})(\|\rho_{\varepsilon}\|_{L^1(\mathbb{R}^n)})^{\frac{m-1}{r-1}}(\|\rho_{\varepsilon}\|_{L^r(\mathbb{R}^n)}^r)^{\frac{r-m}{r-1}}\\
\leq&C^r(M,T,m,n,\|\rho_{0}\|_{L^\infty(\mathbb{R}^n)})+r^2\|\rho_{\varepsilon}\|_{L^r(\mathbb{R}^n)}^r,\text{ for }r+1\geq 2m.
\end{aligned}
\end{equation}

By means of the interpolation inequality (Theorem~\ref{t1}), Sobolev's inequality for the gradient (Theorem~\ref{t8}) and Young's inequality, we also have
\begin{align} 
r^2\|\rho_{\varepsilon}\|_{L^{r}(\mathbb{R}^n)}^r & \leq r^2\|\rho_{\varepsilon}\|_{L^{\frac{r}{4}}(\mathbb{R}^n)}^{\frac{[r+(m-1)n]r}{(3n+1)r+4(m-1)n}}\|\rho_{\varepsilon}\|_{L^{\frac{(m+r-1)n}{n-1}}(\mathbb{R}^n)}^{\frac{[3nr+3(m-1)n]r}{(3n+1)r+4(m-1)n}} \notag \\
&=r^2\|\rho_{\varepsilon}\|_{L^{\frac{r}{4}}(\mathbb{R}^n)}^{\frac{[r+(m-1)n]r}{(3n+1)r+4(m-1)n}}(S_n^{-2}\int_{\mathbb{R}^n}|\nabla\rho_\varepsilon^{\frac{m+r-1}{2}}|^2dx)^{\frac{3nr}{(3n+1)r+4(m-1)n}} \notag \\
&\leq \frac{m(r-1)}{(m+r-1)^2}\|\nabla\rho_\varepsilon^{\frac{r+m-1}{2}}\|_{L^2(\mathbb{R}^n)}^2
+\frac{r+4(m-1)n}{(3n+1)r+4(m-1)n}r^{\frac{2(3n+1)r+8(m-1)n}{r+4(m-1)n}}  \notag  \\
&\qquad \qquad  (S_n^{-2}\frac{(r+m-1)^2}{m(r-1)}\frac{3nr}{(3n+1)r+4(m-1)n})^{\frac{3nr}{r+4(m-1)n}}\|\rho_\varepsilon\|_{L^{\frac{r}{4}}(\mathbb{R}^n)}^{\frac{[r+(m-1)n]r}{r+4(m-1)n}}  \notag  \\
&\leq \frac{m(r-1)}{(m+r-1)^2}\|\nabla\rho_\varepsilon^{\frac{r+m-1}{2}}\|_{L^2(\mathbb{R}^n)}^2+r^{N}\|\rho_\varepsilon\|_{L^{\frac{r}{4}}(\mathbb{R}^n)}^r+C^{r}(m,n), \label{ed2}
\end{align} 
where $N:=9n+2$, $r\geq\max\{4,2m\}$, and in the last inequality we use,
\begin{equation*}
\big(S_n^{-2}\frac{(r+m-1)^2}{mr(r-1)}\frac{3nr}{(3n+1)r+4(m-1)n}\big)^{\frac{3n}{m-1}}\to(\frac{S_n^{-2}3n}{m(3n+1)})^{\frac{3n}{m-1}},\quad\text{as }r\to\infty.
\end{equation*}
Inserting \eqref{ed2}  into \eqref{ed1}, it holds
\begin{equation}\label{ed3}
\frac{d}{dt}\|\rho_{\varepsilon}\|_{L^r(\mathbb{R}^n)}^r\leq r C^r(m,n)+r^{N+1}\|\rho_{\varepsilon}\|_{L^{\frac{r}{4}}(\mathbb{R}^n)}.
\end{equation}
Integrating \eqref{ed3} on $[0,t]$ for any $t\in(0,T]$, we have
\begin{equation*}
\begin{aligned}
\sup\limits_{0\leq t\leq T}\|\rho_{\varepsilon}(t)\|_{L^r(\mathbb{R}^n)}
\leq&r^{\frac{N+1}{r}}\sup\limits_{0\leq t\leq T}\|\rho_\varepsilon(t)\|_{L^{\frac{r}{4}}(\mathbb{R}^n)}+\|\rho_{\varepsilon,0}\|_{L^r(\mathbb{R}^n)}+r^{\frac{1}{r}}C(m,n)\\
\leq& r^{\frac{N+1}{r}}\sup\limits_{0\leq t\leq T}\|\rho_\varepsilon(t)\|_{L^{\frac{r}{4}}(\mathbb{R}^n)}+\|\rho_{0}\|_{L^r(\mathbb{R}^n)}+r^{\frac{1}{r}}C(m,n)\\
\leq& r^{\frac{N+1}{r}}\max\{M,\|\rho_0\|_{L^\infty(\mathbb{R}^n)},C(T,m,n),\sup\limits_{0\leq t\leq T}\|\rho_\varepsilon(t)\|_{L^{\frac{r}{4}}(\mathbb{R}^n)}\}.
\end{aligned}
\end{equation*}

Let $r=4^p$ for $p\in \mathbb{N}^+$, by iteration, we obtain
\begin{equation*}
\begin{aligned}
\sup\limits_{0\leq t\leq T} & \|\rho_{\varepsilon}(t)\|_{L^{4^p}(\mathbb{R}^n)}
\leq 4^{[\frac{p}{4^p}+\frac{p-1}{4^{(p-1)}}]}\max\{M,\|\rho_{0}\|_{L^\infty(\mathbb{R}^n)},C(T,m,n),\sup\limits_{0\leq t\leq T}\|\rho_{\varepsilon}\|_{L^{4^{(p-2)}}(\mathbb{R}^n)}\}\\
\leq& 4^{(N+1)[\frac{p}{4^p}+\frac{p-1}{4^{(p-1)}}+\cdots+\frac{p^*}{4^{p^*}}]}\max\{M,\|\rho_{0}\|_{L^\infty(\mathbb{R}^n)},C(T,m,n),\sup\limits_{0\leq t\leq T}\|\rho_{\varepsilon}(t)\|_{L^{4}(\mathbb{R}^n)}\}\\
\leq& C(M,T,m,n,\|\rho_0\|_{L^\infty(\mathbb{R}^n)}),
\end{aligned}
\end{equation*}
where $\sup\limits_{0\leq t\leq T}\|\rho_{\varepsilon}(t)\|_{L^{4^{p^*}}(\mathbb{R}^n)}\leq C(M,T,m,n,\|\rho_0\|_{L^\infty(\mathbb{R}^n)})$ for $p^*:=\inf\{p\in\mathbb{N}^+:4^{p}\geq\max\{4,2m\}\}$ holds from Lemma~\ref{lr}. Taking $p\to\infty$, we find the announced bound.
\end{proof}

\section{Compact support property}
\label{sec:comsupp}

We prove the $m$ dependent compact support property 
\begin{lemma} 
Assume~\eqref{c1}, then, the global weak solution to the Cauchy problem Eq.~\eqref{d1} in the sense of Def.~\ref{d1} is compactly supported 
 \begin{equation*} \begin{cases}
{\rm supp}(\rho_m(t))\subset B_{\mathcal{R}_{m}(T,t)},\qquad \quad  \forall\ t\in[0,T], \, \forall T>0,
\\
\mathcal{R}_m(T,t):=(\mathcal{R}_{m,0}+\frac{n\|\nabla\mathcal{N}\ast\rho_{m}\|_{L^\infty(Q_T)}}{A_{m,T}})e^{\frac{A_{m,T}t}{n}}-\frac{n\|\nabla\mathcal{N}\ast\rho_{m}\|_{L^\infty(Q_T)}}{A_{m,T}},
\end{cases}
\end{equation*}
with $\mathcal{R}_{m,0}=  2\max(R_m, \sqrt{nP_{m,0}} )$ and $A_{m,T}:=\max\big(\|\rho_m\|_{L^\infty(Q_T)},1\big)$.  
\end{lemma}

Set $V_m(x,t):=\nabla\mathcal{N}\ast\rho_m$, it follows from the $L^\infty$ bound in  Lemma~\ref{ub} that
\begin{equation*}
\begin{aligned}
&\sup\limits_{0\leq t\leq T}\|V_{m}(t)\|_{L^\infty(\mathbb{R}^n)}\leq C(M,T,m,n,\|\rho_{m,0}\|_{L^\infty(\mathbb{R}^n)}),\\
&\sup\limits_{0\leq t\leq T}\|\nabla\cdot V_{m}(t)\|_{L^\infty(\mathbb{R}^n)}\leq C(M,T,m,n,\|\rho_{m,0}\|_{L^\infty(\mathbb{R}^n)}),
\end{aligned}
\end{equation*}
where the first inequality is obtained as \eqref{m22}.
We rewrite Eq.~\eqref{d1} as
\begin{equation*}
\partial_t\rho_m=\Delta \rho_m^m+\nabla\cdot(\rho_m V_m),
\end{equation*}
which is now considered as a porous medium equation with a given drift $V_m$. Then, similar to \cite[Lemma 3.8]{CKY_2018}, we can construct a viscosity sup-solution with a uniform compact support in any finite time for the equation of pressure $P_m=\frac{m}{m-1}\rho_m^{m-1}$,
 \begin{equation*}
 \partial_t P_m=(m-1)P_m(\Delta P_m+\nabla\cdot V_m)+|\nabla P_m|^2+\nabla P_m\cdot V_m.
 \end{equation*}
 Hence, the solution $\rho_m$ to the Cauchy problem \eqref{d1}  under the initial assumption~\eqref{c1} is compactly supported on $[0,T]$ for any $T>0$.

 More precisely, let $\mathcal{R}_m(T,t):=(\mathcal{R}_{m,0}+\frac{n\|\nabla\mathcal{N}\ast\rho_{m}\|_{L^\infty(Q_T)}}{A_{m,T}})e^{\frac{A_{m,T}t}{n}}-\frac{n\|\nabla\mathcal{N}\ast\rho_{m}\|_{L^\infty(Q_T)}}{A_{m,T}}$ for any $T>0$ with $\mathcal{R}_{m,0}\geq1$ satisfying $2R_m\leq\mathcal{R}_{m.0}$ and $P_{m,0}\leq \frac{\mathcal{R}_{m,0}^2}{4n}$, and $A_{m,T}:=\max\{\|\rho_m\|_{L^\infty(Q_T)},1\}$. Set $\phi_{m,T}(x,t):=\frac{A_{m,T}\mathcal{R}_{m,T}^2\big(1-\frac{|x|^2}{\mathcal{R}_{m,T}^2}\big)_+}{2n}$, it is easy to verify on the support of $\phi_{m,T}$ that
 \[\partial_t \phi_{m,T}\geq (m-1)\phi_{m,T}(\Delta\phi_{m,T}+\nabla\cdot V_m)+|\nabla \phi_{m,T}|^2+\nabla \phi_{m,T}\cdot V_m\quad\text{for }t\in [0,T]\]
 with $P_{m,0}\leq \phi_{m,T}(x,0)$ in $\mathbb{R}^n$. Therefore, it holds by the comparison principle that
 \begin{equation*}
 \begin{aligned}
 &P_m(x,t)\leq \phi_{m,T}(x,t),\quad&&\forall\ (x,t)\in\mathbb{R}^n\times[0,T],\\
 & {\rm supp}(\rho_m(t))\subset B_{\mathcal{R}_{m}(T,t)} ,\quad &&\forall\ t\in[0,T].
 \end{aligned}
 \end{equation*}
The proof of the compact support property is completed.

\section{Appendix: various functional inequalities}

\begin{lemma}[Singular integral for Newtonian potential]\label{l12}
\cite{r81,r46}
Let $\mathcal{N}$ be the Newtonian potential. For $f(x)\in
L^p(\mathbb{R}^n),1<p<\infty$, we have
\begin{equation*}
\|\nabla^2\mathcal{N}\ast f(x)\|_{L^p(\mathbb{R}^n)}\leq
C(p,n)\|f(x)\|_{L^p(\mathbb{R}^n)},
\end{equation*}
where $C(p,n)\sim \frac{1}{p-1}$ for $0<p-1\ll 1$ and $C(p,n)\sim p$ for $p\gg
1$.
\end{lemma}

\begin{theorem}[Hardy-Littlewood-Sobolev inequality]\label{t7} \cite{1}
 Let $p,r>1$ and $0<\lambda<n$ with
$\frac{1}{p}+\frac{\lambda}{n}+\frac{1}{r}=2$. Let $f\in L^p(\mathbb{R}^n)$ and
$h\in L^r(\mathbb{R}^n)$. Then, we have 
\begin{equation}\label{w10}
\Big|\int_{\mathbb{R}^n}\int_{\mathbb{R}^n}f(x)|x-y|^{-\lambda}h(y)dxdy\Big|\leq
C(n,\lambda,p)||f||_{L^p(\mathbb{R}^n)}||h||_{L^r(\mathbb{R}^n)},
\end{equation}
where the constant satisfies $
C(n,\lambda,p)=\frac{n}{n-\lambda}(\frac{|\mathbf{S}^{n-1}|}{n})^{\frac{\lambda}{n}}\frac{1}{pr}\big((\frac{p\lambda}{np-n})^{\frac{\lambda}{n}}+(\frac{r\lambda}{nr-n})^{\frac{\lambda}{n}}\big)$.

For $\lambda=n-2$ and $p=r=\frac{2n}{n+2}$, we have
$
C(n,p,\lambda)=C(n)=\pi^{\frac{n-2}{2}}\frac{\Gamma(1)}{\Gamma(\frac{n+2}{2})}\{\frac{\Gamma(\frac{n}{2})}{\Gamma(n)}\}^{-\frac{2}{n}}.
$
\end{theorem}

\begin{theorem}[Interpolation inequality]\label{t1} Let $\Omega \subset \mathbb{R}^n$ be a
measurable domain, let $f(x)\in L^{p}(\Omega)\cap L^q(\Omega)$, then,  with  $0\leq \beta=\frac{q-r}{q-p}\leq 1$, we have
\begin{equation*}
 \int_{\Omega}|f|^rdx\leq(\int_{\Omega}|f|^pdx)^{\beta}(\int_{\Omega}|f|^qdx)^{1-\beta}.
\end{equation*}
\end{theorem}

\begin{theorem}[Sobolev's inequality for gradient]\label{t8} \cite{1} For $n\geq
3$, let $f \in L^1_{loc}(\mathbb{R}^n)$ with $\|\nabla
f\|_{L^2}<\infty$. Then $f\in L^q(\mathbb{R}^n)$ and the
following inequality holds:
\begin{equation}\label{w9}
S_n\|f\|_{L^q(\mathbb{R}^n)}^2\leq\|\nabla f\|_{L^2(\mathbb{R}^n)}^2, \qquad q=\frac{2n}{n-2}, 
\end{equation}
where $S_n:=\frac{n(n-1)}{4}2^{\frac{2}{n}}\pi^{1+\frac{2}{n}}\Gamma(\frac{n+1}{2})^{-\frac{1}{n}}$ is the optimal constant for~\eqref{w9}.
\end{theorem}

\begin{theorem}[Gagliardo-Nirenberg-Sobolev inequality]\label{t9}
	Assume $q,r$  satisfy $1\leq q,r \leq \infty$ and $j,m\in \mathbb{Z}^+$
satisfy $0\leq j<m$. For any $f(x)\in C_0^{\infty}(\mathbb{R}^n)$, then we have
	\begin{equation}
	\|D^jf(x)\|_{L^p(\mathbb{R}^n)}\leq
C\|D^mf(x)\|_{L^r(\mathbb{R}^n)}^{\alpha}\|u\|_{L^q(\mathbb{R}^n)}^{1-\alpha}\label{w11},
	\end{equation}
	where
$\frac{1}{p}-\frac{j}{n}=\alpha(\frac{1}{r}-\frac{m}{n})+(1-\alpha)\frac{1}{q}$,$\frac{j}{m}\leq\alpha\leq1$
and C depends on $m,n,j,q,r,\alpha$. If $m-j-\frac{n}{r}$ is a nonnegative
integer, then \eqref{w11}is established for $\frac{j}{m}\leq\alpha<1$.
\end{theorem}

\end{document}